\pdfoutput=1
\RequirePackage[l2tabu, orthodox]{nag}

\documentclass[reqno]{amsart}

\usepackage{lmodern}
\usepackage[T1]{fontenc}
\usepackage[utf8]{inputenc}
\usepackage[english]{babel}
\usepackage{microtype} 

\usepackage{amsthm,mathtools,amssymb,amsmath,mathrsfs,latexsym,mathdots,booktabs,enumerate,tikz,bm,url,comment}
\usepackage{faktor}
\usepackage{hyperref}
\usepackage[font=footnotesize]{caption}
\usepackage[enableskew,vcentermath]{youngtab}
\usepackage[centertableaux]{ytableau}

\usetikzlibrary{arrows,positioning,shapes}
\usepackage[capitalize,noabbrev]{cleveref}

\usepackage{todonotes}
\presetkeys%
{todonotes}%
{inline,backgroundcolor=yellow}{}

\usepackage{graphicx}

\usepackage{ifpdf}
\graphicspath{{./}{./figures/}}
\usepackage{epstopdf}
\DeclareGraphicsExtensions{.png,.jpg,.eps,.epsf,.pdf}

\usepackage{shuffle}
\usepackage{xparse}

\newtheorem{theorem}{Theorem}
\newtheorem{proposition}[theorem]{Proposition}
\newtheorem{lemma}[theorem]{Lemma}
\newtheorem{corollary}[theorem]{Corollary}
\newtheorem{conjecture}[theorem]{Conjecture}

\theoremstyle{definition}
\newtheorem{thedefinition}[theorem]{Definition}
\newtheorem{theexample}[theorem]{Example}

\newtheorem{question}{Question}
\newtheorem{problem}[theorem]{Problem}

\newcommand{\qeddef}{\ensuremath{\Box}}
\newcommand{\qeddefhere}{\tag*{\qeddef}}
\NewDocumentEnvironment{definition}{s}{
	\begin{thedefinition}
}{
    \IfBooleanTF{#1}
        {}
        {\hfill\qeddef}
	\end{thedefinition}
}\NewDocumentEnvironment{example}{s}{
	\begin{theexample}
}{
    \IfBooleanTF{#1}
        {}
        {\hfill\qeddef}
	\end{theexample}
}

\newcommand{\setN}{\mathbb{N}}
\newcommand{\setP}{\mathbb{N}^+}

\newcommand{\setQ}{\mathbb{Q}}

\newcommand{\qvec}{\mathbf{q}}

\newcommand{\xvec}{\mathbf{x}}
\newcommand{\xvar}{\mathbf{x}}
\newcommand{\yvec}{\mathbf{y}}

\newcommand{\ILinExt}{\mathcal{L}}

\newcommand{\symS}{\mathfrak{S}}

\newcommand{\psumP}{\mathrm{p}}
\newcommand{\elementaryE}{\mathrm{e}}
\newcommand{\completeH}{\mathrm{h}}

\newcommand{\schurS}{\mathrm{s}}
\newcommand{\QS}{\mathcal{QS}}
\newcommand{\RS}{\mathcal{RS}}
\newcommand{\LLT}{\mathrm{G}}
\newcommand{\chrom}{\mathrm{X}}
\newcommand{\gessel}{\mathrm{F}}
\newcommand{\qmonom}{\mathrm{M}}

\newcommand{\coloring}{{\kappa}}

\newcommand{\SYT}{\mathrm{SYT}}

\DeclareMathOperator{\DES}{DES}
\DeclareMathOperator{\DEX}{DEX}
\DeclareMathOperator{\EXC}{EXC}

\DeclareMathOperator{\des}{des}

\DeclareMathOperator{\exc}{exc}

\DeclareMathOperator{\length}{\ell}

\DeclareMathOperator{\inv}{inv}

\DeclareMathOperator{\asc}{asc}

\newcommand{\refq}[1]{\eqref{eq:#1}}
\newcommand{\rpp}{\mathcal{A}^r}
\newcommand{\compset}[1]{S_{#1}}
\newcommand{\compb}[2]{B_{#1}^{#2}}
\newcommand{\compp}[2]{P_{#1}^{#2}}

\newcommand{\abs}[1]{\left| #1 \right|}

\newcommand{\banners}{\mathfrak{B}}
\newcommand{\CONS}[2]{\textsc{Cons}(#1,#2)}
\newcommand{\necklaces}{\mathfrak{N}}

\newcommand{\posets}{\mathscr{P}}

\newcommand{\down}{\mathtt{d}}
\newcommand{\up}{\mathtt{u}}
\newcommand{\step}[1]{a_{#1}}
\newcommand{\classes}{\mathscr{C}}
\newcommand{\opsurj}{\mathcal{O}}

\newcommand{\poseta}{x}
\newcommand{\posetb}{y}
\newcommand{\posetc}{z}
\newcommand{\poseti}[1]{x_{#1}}
\newcommand{\posetj}[1]{y_{#1}}

\newcommand{\matBasisB}{\mathcal{B}}

\newcommand{\oeis}[1]{\href{http://oeis.org/#1}{#1}}

\title{
$P$-partitions and $p$-positivity}

\author{Per Alexandersson}
\address{Dept. of Mathematics, Royal Institute of Technology, SE-100 44 Stockholm, Sweden}
\email{per.w.alexandersson@gmail.com}

\author{Robin Sulzgruber}
\address{Dept. of Mathematics, Royal Institute of Technology, SE-100 44 Stockholm, Sweden}
\email{robinsul@kth.se}

\keywords{Chromatic quasisymmetric functions, LLT polynomials, P-partitions, power sums, quasisymmetric functions, Tutte polynomials, unimodality}
\subjclass[2010]{Primary~05E05; Secondary~06A07, 05A05.}
\date{July 2018}

\begin{document}

\begin{abstract}

Using the combinatorics of $\alpha$-unimodal sets, we establish two new results in the theory of quasisymmetric functions.
First, we obtain the expansion of the fundamental basis into quasisymmetric power sums.
Secondly, we prove that generating functions of reverse $P$-partitions expand positively into quasisymmetric power sums.
Consequently any nonnegative linear combination of such functions is $p$-positive whenever it is symmetric.
As an application we derive positivity results for chromatic quasisymmetric functions, unicellular and vertical strip LLT polynomials, multivariate Tutte polynomials and the more general $B$-polynomials, matroid quasisymmetric functions, and certain Eulerian quasisymmetric functions, thus reproving and improving on numerous results in the literature.

\end{abstract}
\maketitle


\section{Introduction}
\label[section]{sec:introduction}

Whenever a new family of symmetric functions is discovered, one of the most logical first steps to take is to expand them in one of the many interesting bases of the space of symmetric functions.
This paradigm can be traced from Newton's identities to modern textbooks such as \cite{Macdonald79symmetric}.
Of special interest are expansions in which all coefficients are nonnegative integers.
Such coefficients frequently encode highly nontrivial combinatorial or algebraic information.

One of the most well-studied bases is formed by the power sum symmetric functions.
Symmetric functions that expand into power sum symmetric functions with nonnegative coefficients are called \emph{$p$-positive}.
Recent works in which $p$-positivity is discussed include \cite{ShareshianWachs2010,SaganShareshianWachs2011,Athanasiadis15,ShareshianWachs2016,Ellzey2016,AlexanderssonPanova2016}.
The expansion of a symmetric function into power sum symmetric functions can be useful, for instance, when one is working with plethystic substitution~\cite{LoehrRemmel2010}, or evaluating certain polynomials at roots of unities~\cite{Desarmenien1983,SaganShareshianWachs2011}.

Suppose $X$ is a symmetric function for which we would like to know the expaonsion into power sum symmetric functions.
In some of the papers mentioned above the following pattern recurs.
First, expand $X$ into fundamental quasisymmetric functions using R.~Stanley's theory of $P$-partitions.
Secondly, conduct some analysis specific to the function $X$ at hand to obtain the $p$-expansion of $X$.
Ideally, one would ask for a more uniform approach.

\begin{question}\label[question]{q1}
Is there a uniform method for deriving the expansion of a given symmetric function into power sum symmetric functions whenever the theory of $P$-partitions is applicable?
\end{question}

In practice the functions of interest often belong to the larger space of quasisymmetric functions.
Clearly if a quasisymmetric function expands into power sum symmetric functions, positively or not, then it has to be symmetric.
This leads to results of the following type:
``Suppose $X$ belongs to some special family of quasisymmetric functions $\mathcal{F}$.
Then $X$ is $p$-positive if and only if $X$ is symmetric.''
In this case it is very natural to ask the following.

\begin{question}\label[question]{q2}
Is there a more general positivity phenomenon hiding in the background, which encompasses all quasisymmetric functions that belong to $\mathcal{F}$\,?
\end{question}

In this paper we answer both \cref{q1} and \cref{q2} in the affirmative.
Key to these answers are the quasisymmetric power sums $\Psi_{\alpha}$.
Quasisymmetric power sums originate in the work of I.~Gelfand et al.~\cite{GelfandKrobLascouxLeclercRetakhThibon1995} on noncommutative symmetric functions, and were recently investigated by C.~Ballantine et al.~\cite{BallantineDaughertyHicksMason2017}.\footnote{There are several different quasisymmetric power sum bases. We use one denoted by $\Psi_{\alpha}$ in \cite{BallantineDaughertyHicksMason2017}.}
The family $\Psi_{\alpha}$, where $\alpha$ ranges over all compositions of $n$, forms a basis of the space of homogeneous quasisymmetric functions of degree $n$, and refines the power sum symmetric functions.

This paper has two main results that easily fit into the existing theory of quasisymmetric functions.
The first result is a formula for the expansion of the fundamental quasisymmetric functions into quasisymmetric power sums.

\begin{theorem}[\cref{thm:gesselInPSum}]
\label[theorem]{t1}
Let $n\in\setN$ and $S \subseteq [n-1]$.
Then the fundamental quasisymmetric function $\gessel_{n,S}$ expands into quasisymmetric power sums as
\begin{equation}\label{eq1}
\gessel_{n,S}(\xvec)
= \sum_{\alpha
}
\frac{\Psi_{\alpha}(\xvec)}{z_{\alpha}} (-1)^{|S\setminus \compset{\alpha}|}
\,,
\end{equation}
where the sum ranges over all compositions $\alpha$ of $n$ such that $S$ is $\alpha$-unimodal.
\end{theorem}

Here we use the standard notation $[n] \coloneqq \{1,2,\dotsc,n\}$.
The definitions of $\alpha$-unimodal sets and the set $\compset{\alpha}$ are found in Sections~\ref{sec:unimodal}.
The quasisymmetric functions $\gessel_{n,S}$ and $\Psi_{\alpha}$ and the factor $z_{\alpha}$ are defined in~\cref{sec:gessel}.
The proof of \cref{t1} relies on a new result on $\alpha$-unimodal sets, and on the hook-length formula for forests.

\cref{t1} yields a new proof of results due to Y.~Roichman~\cite[Thm.~4]{Roichman1997} and C.~Athanasiadis~\cite[Prop.~3.2]{Athanasiadis15}, both of which feature $\alpha$-unimodal sets as well.

The second main result concerns reverse $P$-partitions, which were introduced by R.~Stanley~\cite{Stanley1972}.
In the simplest case reverse $P$-partitions are order-preserving maps from a finite poset $P$ to the positive integers.
The generating function of reverse $P$-partitions is defined as 
\[
K_P(\xvec)
\coloneqq
\sum_{\substack{f : P\to \setP \\ \poseta <_P \posetb \Rightarrow f(\poseta)\leq f(\posetb)}}
\prod_{\poseta \in P } \xvar_{f(\poseta)}.
\]
The function $K_P$ is a homogeneous quasisymmetric function of degree $n=\abs{P}$.
We prove that $K_P$ expands positively into quasisymmetric power sums $\Psi_{\alpha}$ and provide two combinatorial interpretations for the involved coefficients.

\begin{theorem}[Theorems~\ref{thm:KP_positive} and~\ref{thm:KP_opsurj}]
\label{t2}
Let $(P,w)$ be a naturally labeled poset with $n$ elements. Then
\begin{equation}
K_{P}(\xvec)
=
\sum_{\alpha\vDash n}
\frac{\Psi_{\alpha}(\xvec)}{z_{\alpha}}
|\ILinExt^\ast_{\alpha}(P,w)|
=
\sum_{\alpha\vDash n}
\frac{\Psi_{\alpha}(\xvec)}{z_{\alpha}}
|\opsurj_{\alpha}^\ast(P)|
\,,
\end{equation}
were both sums range over all compositions $\alpha$ of $n$.
In particular, the quasisymmetric function $K_P$ is $\Psi$-positive.
\end{theorem}
The set $\ILinExt^\ast_{\alpha}(P,w)$ consists of certain $\alpha$-unimodal linear extensions\footnote{To be precise, the elements of $\ILinExt_{\alpha}^{\ast}(P,w)$ lie in the Jordan--H{\"o}lder set of $(P,w)$, that is, they are perhaps more accurately described as inverses of linear extensions of $(P,w)$.} of $P$.
The definition is given in \cref{sec:unimodal}.
The set $\opsurj_{\alpha}^{\ast}(P)$ consists of certain order-preserving surjections from $P$ onto a chain.
The definition is given in \cref{sec:opsurj}.
The proof of \cref{t2} uses the well-known expansion of $K_P$ into the fundamental basis, \cref{t1}, and a sign-reversing involution closely related to an involution constructed by B.~Ellzey in~\cite[Thm.~4.1]{Ellzey2016}.

It follows from \cref{t2} that any \emph{symmetric} function which is a positive linear combination of functions $K_P$ for posets $P$ is $p$-positive.
This affirms \cref{q1}.

It is a manifestation of the ubiquity of reverse $P$-partitions in algebraic combinatorics that many interesting families of symmetric and quasisymmetric functions can be expressed as nonnegative linear combinations of functions $K_P$. 
By \cref{t2} each function $X$ that belongs to such a family is $\Psi$-positive. 
This answers \cref{q2} for a large class of families $\mathcal{F}$.

As an application we give positivity results and combinatorial interpretations for the coefficients in the expansion into (quasi)symmetric power sums for the following families of quasisymmetric functions:
\begin{itemize}
\item The chromatic quasisymmetric functions of J.~Shareshian and M.~Wachs, \cite{ShareshianWachs2016}.
We prove a generalization of a recent result by B.~Ellzey~\cite{Ellzey2016},
that applies to all directed graphs, and not only those with a symmetric chromatic quasisymmetric functions.
Our result also extends to a $q$-generalization of so 
called $k$-balanced chromatic quasisymmetric functions that were introduced by B.~Humpert in \cite{Humpert2011}.

\item Unicellular and vertical-strip LLT polynomials, which are of special 
interest in \cite{CarlssonMellit2017} and in the study of diagonal harmonics.
This generalizes an observation in \cite{AlexanderssonPanova2016,HaglundWilson2017} and answers an open problem in \cite{AlexanderssonPanova2016}. 
Furthermore, this result provides more supporting evidence regarding a related $e$-positivity conjecture.

\item The multivariate Tutte polynomials introduced by 
R.~Stanley \cite{Stanley98Chromatic}, and the more general $B$-polynomals on directed graphs due to J.~Awan and O.~Bernardi, \cite{AwanBernardi2016}.

\item The quasisymmetric functions associated to matroids due to L.~Billera, N.~Jia and V.~Reiner, \cite{BilleraJiaReiner2009}.

\item Certain Eulerian quasisymmetric functions introduced by J.~Shareshian and M.~Wachs in \cite{ShareshianWachs2010}.
\end{itemize}

\cref{fig:polynomialGraph} gives an overview of some of the bases of symmetric and quasisymmetric functions that are mentioned in this paper.
\begin{figure}[t]
\includegraphics[width=0.7\textwidth]{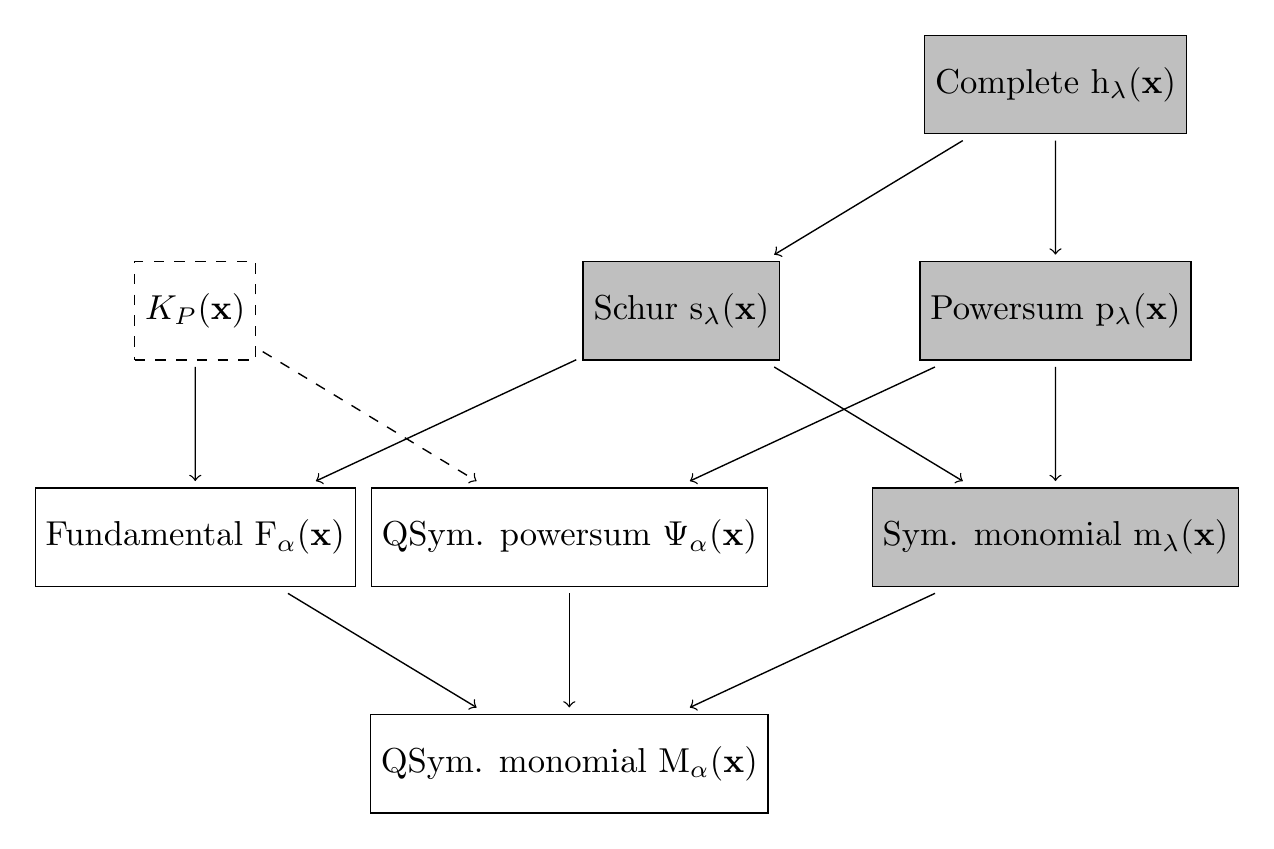}
\caption{
An overview of the families of functions we discuss. The shaded families are symmetric functions (and bases for the corresponding space).
The remaining families are bases for the space of quasisymmetric functions, except for $K_P$ which is too large to be a basis.
The arrows represent the relation ``expands positively in'' (which of course is a transitive relation).
The dashed line is the result in \cref{t2}.
}\label{fig:polynomialGraph}
\end{figure}

\subsection{Outline}

In \cref{sec:unimodal} we engage in the combinatorics of $\alpha$-unimodal permutations and sets.
We prove two new results which are instrumental in the proofs of our main theorems.
In particular we define the set $\ILinExt_{\alpha}^{\ast}(P,w)$ attached to a labeled poset as the set of certain $\alpha$-unimodal linear extensions of $P$.
In \cref{sec:gessel} we give a short introduction to quasisymmetric functions and define quasisymmetric power sums.
We proceed to prove \cref{t1} and conclude \cite[Prop.~3.2]{Athanasiadis15} as a corollary.
In \cref{sec:revPP} we define reverse $P$-partitions and prove the first half of \cref{t2}.
\cref{sec:opsurj} is dedicated to order-preserving surjections onto chains.
It contains the definition of the set $\opsurj_{\alpha}^{\ast}(P)$ and the proof of the second half of \cref{t2}.
In \cref{sec:revPPeq} we generalize \cref{t2} to include weighted posets or, equivalently, reverse $P$-partitions with forced equalities.
This is perhaps the most technical section and it is not required to understand the rest of the paper.
In \cref{sec:applications} we use the developed tools to derive $\Psi$-expansions of some of the most commonly used bases of the space of symmetric functions ($h_{\lambda}$, $p_{\lambda}$ and $s_{\lambda}$), including Roichman's formula~\cite[Thm.~4]{Roichman1997}.
Moreover we obtain the positivity results mentioned above.
Finally, in \cref{sec:directions} we mention several interesting direction that could be pursued in the future, as well as some ideas that, sadly, do not work.

\subsection{Acknowledgements}

The authors would like to thank Svante Linusson.
The first author is funded by the \emph{Knut and Alice Wallenberg Foundation} (2013.03.07).

\numberwithin{theorem}{section}

\section{\texorpdfstring{$\alpha$}{alpha}-unimodal combinatorics}
\label[section]{sec:unimodal}

In this section we investigate $\alpha$-unimodal permutations, sets and compositions.
Our main objective is to prove two bijective results, namely \cref{thm:CONS} and \cref{thm:ILinExt}, 
which we apply to the theory of quasisymmetric functions in the subsequent sections.
However, we contend that $\alpha$-unimodal combinatorics is an interesting topic in its own right.

The so called \emph{$\alpha$-unimodal} sets, where $\alpha$ is a composition, first appear in a recursive 
formula for Kazhdan--Lusztig characters of the Hecke algebra of type $A_{n-1}$ due to Y.~Roichman~\cite[Thm.~4]{Roichman1997}.
A bijective treatment of this formula was later given by A.~Ram~\cite{Ram1998}.
The term $\alpha$-unimodal was coined in~\cite{AdinRoichman2015}.
We refer to \cite{ElizaldeRoichman2013,Thibon2001} for more results on unimodal permutations.
More recently $\alpha$-unimodal sets were used in the works of C.~Athanasiadis~\cite{Athanasiadis15} and B.~Ellzey~\cite{Ellzey2016} 
to derive the power sum expansions of certain families of symmetric functions.

A word $\sigma_1\dotsm\sigma_n$ is \emph{unimodal} if there exists an index $k\in[n]$ such that
\begin{equation}\label{eq:unimodal}
\sigma_1>\dotsm>\sigma_k<\dotsm<\sigma_n.
\end{equation}
For instance, note that a permutation $\sigma\in\symS_n$ is unimodal if and only if the set
$\{\sigma^{-1}({1}),\dotsc,\sigma^{-1}({i})\}$ is a subinterval of $[n]$ for all $i\in[n]$.
We remark that this definition, which is borrowed from~\cite{AdinRoichman2015,Athanasiadis15,Ellzey2016}, is not 
standard in the study of unimodal sequences, where one would usually assume $\sigma_1<\dotsm<\sigma_k>\dotsm>\sigma_n$ instead.
However, \eqref{eq:unimodal} is more natural if one wants to use descents rather than ascents, 
which is common practice when working with tableaux and quasisymmetric functions.
In Sections~\ref{sec:applications} and~\ref{sec:directions} we discuss unimodality of polynomials.
In that case the standard definition is used, that is, a polynomial $a_0+a_1q+\dotsm+a_dq^d$ is \emph{unimodal} if there exists $k\in\{0,\dotsc,d\}$ with $a_0\leq\dotsm\leq a_k\geq\dotsm\geq a_d$.

Given a composition $\alpha$ of $n$ with $\ell$ parts define
\[
\compset{\alpha}\coloneqq\{\alpha_1,\alpha_1+\alpha_2,\dotsc,\alpha_1+\dotsm+\alpha_{\ell-1}\}
\,.
\]
The correspondence $\alpha\mapsto\compset{\alpha}$ is a well-known bijection between compositions of $n$ and subsets of $[n-1]$.
Sometimes it is more convenient to 
work with the set $\compset{\alpha}^0\coloneqq\compset{\alpha}\cup\{0\}$ instead.
Let $\alpha,\beta$ be compositions of $n$.
We say $\alpha$ \emph{refines} $\beta$ if $\compset{\beta}\subseteq\compset{\alpha}$.
Refinement on compositions is denoted by $\alpha\leq\beta$.
Equivalently, $\alpha\leq\beta$ if $\beta$ can be obtained by adding contiguous parts of $\alpha$.
We also say $\beta$ is \emph{coarser} than $\alpha$.

Define the \emph{blocks} of $\alpha$ as the sets
\[
\compb{1}{\alpha} \coloneqq \{1,2,\dotsc,\alpha_1\},\quad
\compb{2}{\alpha} \coloneqq \{\alpha_1+1,\dotsc,\alpha_1+\alpha_2\},
\]
and so on, such that $|\compb{i}{\alpha}|= \alpha_i$ for all $i\in[\ell]$ and $\{\compb{1}{\alpha},\compb{2}{\alpha},\dotsc,\compb{\ell}{\alpha}\}$ forms a set partition of $[n]$.

\begin{example}*
Let $\alpha=11213311$ and $\beta=535$.
Then $\compset{\alpha}=\{1,2,4,5,8,11,12\}$, $\compset{\beta}=\{5,8\}$ and $\alpha\leq\beta$.
The blocks of $\alpha$ are 
\[
\{1\},\{2\},\{3,4\},\{5\},\{6,7,8\},\{9,10,11\},\{12\} \text{ and } \{13\}.
\qeddefhere
\]
\end{example}

Let $\alpha$ be a composition of $n$.
A permutation $\sigma\in\symS_n$ is \emph{$\alpha$-unimodal} if the word obtained by 
restricting $\sigma$ to the block $\compb{i}{\alpha}$ is unimodal for all $i\in[\ell]$.
That is, for all $i\in[\ell]$ if $\compb{i}{\alpha}=[a,b]$ then there exists $k\in[a,b]$ such that
$\sigma_a>\dotsm>\sigma_k<\dotsm\sigma_b$.

The \emph{descent set} of a permutation $\sigma\in\symS_n$ is defined as
\[
\DES(\sigma) \coloneqq \{i\in[n-1]: \sigma(i) > \sigma(i+1) \},
\]
Furthermore, set $\des(\sigma)\coloneqq\abs{\DES(\sigma)}$.

A set $S \subseteq [n-1]$ is \emph{$\alpha$-unimodal} if it is the descent set of an $\alpha$-unimodal permutation.
Equivalently, $S\subseteq[n-1]$ is $\alpha$-unimodal if for all $i \in [\ell]$ the intersection of $S$ with $\compb{i}{\alpha}\setminus\compset{\alpha}$
is a prefix of the latter.

There is yet another equivalent description of $\alpha$-unimodal sets.
Define the binary sequence $a_1,\dots,a_{n-1}$ by letting $a_i=1$ if $i\in S$, and $a_i=0$ otherwise.
Similarly define $b_1,\dots,b_{n-1}$ by $b_i=1$ if $i\in\compset{\alpha}$, and $b_i=0$ otherwise.
Then $S$ is $\alpha$-unimodal if and only if the two-line arrangement
\begin{equation*}
\begin{matrix}
a_1&a_2&\cdots&a_{n-1}\\
b_1&b_2&\cdots&b_{n-1}
\end{matrix}
\qquad
\text{avoids the pattern}
\qquad
\begin{matrix}
0&1\\
0&0
\end{matrix}
\quad,
\end{equation*}
that is,
\begin{equation*}
\begin{matrix}
a_i&a_{i+1}\\
b_i&b_{i+1}
\end{matrix}
\quad
\neq
\quad
\begin{matrix}
0&1\\
0&0
\end{matrix}
\end{equation*}
for all $i\in[n-2]$.

Let $U_{\alpha}$ denote the set of $\alpha$-unimodal subsets of $[n-1]$.

\begin{example}
Let $\alpha=33$. Then $\compb{1}{\alpha}\setminus\compset{\alpha}= \{1,2\}$ and $\compb{2}{\alpha}\setminus\compset{\alpha}= \{4,5,6\}$.
The $\alpha$-unimodal subsets of $[5]$ are the following sets:
\begin{center}
\begin{tabular}{ccccccccc}
$\emptyset$ & 1 & 3 & 4 & 12 & 13 & 14 & 34 & 45 \\
123 & 124 & 134 & 145 & 345 & 1234 & 1245 & 1345 & 12345 \\
\end{tabular}
\end{center}
In short, if $2$ appears, then $1$ must also appear and, similarly, if $5$ appears then $4$ must also be there.
Equivalently, if the two-line arrangement
\begin{equation*}
\begin{matrix}
a_1&a_2&a_3&a_4&a_5\\
0&0&1&0&0
\end{matrix}
\end{equation*}
avoids the forbidden pattern, then $a_2=1$ implies $a_1=1$, and $a_5=1$ implies $a_4=1$.

\end{example}

We start out by collecting some basic combinatorial facts on $\alpha$-unimodality.
All of them are straightforward to prove and should be viewed as a warm up to get acquainted with $\alpha$-unimodal sets.

\begin{proposition}\label{lem:unimodal}
Let $n\in\setN$ and $\alpha,\beta$ be compositions of $n$ with $\ell$ parts.

\begin{enumerate}[(i)]
\item\label{item:unimodal_iff}
A subset $S\subseteq[n-1]$ is $\alpha$-unimodal if and only if for all $k\in S\setminus\compset{\alpha}$ we have $k-1\in S\cup\compset{\alpha}^0$.

\item\label{item:unimodal_subset}
If $\alpha\leq\beta$, then $U_{\beta}\subseteq U_{\alpha}$.

\item\label{item:unimodal_leq}
If $\alpha\leq\beta$ then $\compset{\beta}$ is $\alpha$-unimodal.

\item\label{item:unimodal_self}
In particular, $\compset{\alpha}$ is $\alpha$-unimodal.

\item\label{item:unimodal_countU}
The number of $\alpha$-unimodal subsets of $[n-1]$ is given by 
\begin{equation}\label{eq:Ua}
|U_{\alpha}|=2^{\ell-1}\cdot \alpha_1\cdot \alpha_2 \dotsm \alpha_\ell 
.
\end{equation}

\item\label{item:unimodal_posetU}
The set $U_{\alpha}$ is closed under unions and intersections and therefore forms a sublattice of the Boolean lattice.

\item\label{item:unimodal_Moebius}
The lattice $U_{\alpha}$ is the direct product of chains.
Its M{\"o}bius function is given by
\[
\mu(S)=
\begin{cases}
(-1)^{\abs{S}}&\quad\text{if }S\subseteq\compset{\alpha}\cup\{k+1:k\in\compset{\alpha}^0\},\\
0&\quad\text{otherwise}.
\end{cases}
\]
\item\label{item:unimodal_Va}
Let $V_{\alpha}\coloneqq\{\gamma\vDash n:\compset{\alpha}\in U_{\gamma}\}$.
Then
\begin{equation}\label{eq:Va}
|V_{\alpha}|
=
2^{n-1}\cdot\left(\frac{3}{4}\right)^{m}
\,,
\end{equation}
where $m$ denotes the number of indices $i\in[\ell-1]$ such that $\alpha_i>1$.
\item\label{item:unimodal_posetV} The set $V_{\alpha}$ defined in \eqref{item:unimodal_Va} is an order ideal and thus
a meet-semilattice in the lattice of compositions ordered by refinement.
However, $V_{\alpha}$ is generally not a sublattice.

\item\label{item:unimodal_gf}
The $q,t$-generating function of $\alpha$-unimodal sets is given by
\begin{equation*}
F(q,t,z)
\coloneqq
\sum_{n\geq 1}
\sum_{\alpha\vDash n}
\sum_{S\in U_{\alpha}} q^{\abs{S}}t^{\ell(\alpha)}z^n
=
\frac{tz}{1-(1+q)(1+t)z+qz^2}
\,.
\end{equation*}
\end{enumerate}
\end{proposition}

\begin{proof}
Claims~\eqref{item:unimodal_iff}, \eqref{item:unimodal_subset}, \eqref{item:unimodal_leq}, \eqref{item:unimodal_self} and~\eqref{item:unimodal_posetU} follow directly from the definition of $\alpha$-unimodal sets.
Claim~\eqref{item:unimodal_posetV} is an immediate consequence of Claim~\eqref{item:unimodal_subset}.

Given $S\subseteq[n-1]$ let $\varphi(S)\coloneqq(S\cap\compset{\alpha},r_1,\dots,r_{\ell})$, where $r_i\coloneqq\abs{S\cap(\compb{i}{\alpha}\setminus\compset{\alpha})}$ for all $i\in[\ell]$.
It is not difficult to see that this defines a bijection
\begin{equation*}
\varphi:U_{\alpha}
\to
\{T\subseteq\compset{\alpha}\}
\times
\prod_{i=1}^{\ell}\{0,\dotsc,\alpha_i-1\}
\,,
\end{equation*}
which yields \eqref{item:unimodal_countU}.
In fact, the map $\varphi$ is an isomorphism of partially ordered sets, where by definition $(T,r_1,\dotsc,r_{\ell})\leq(T',r_1',\dotsc,r_{\ell}')$ if and only if $T\subseteq T'$ and $r_i\leq r_i'$ for all $i\in[\ell]$.
Thus Claim~\eqref{item:unimodal_Moebius} follows from standard techniques for computing the M{\"o}bius function of finite posets, see~\cite[Sec.~3.8]{StanleyEC1}.

The remaining two claims are best understood using the definition of $\alpha$-unimodality via two-line arrangements of zeroes and ones.
To see Claim~\eqref{item:unimodal_Va} let $a_1,\dotsc,a_{n-1}$ be a fixed binary sequence encoding the set $\compset{\alpha}$.
We are looking to determine the number of compatible sequences $b_1,\dotsc,b_{n-1}$.
Each part of $\alpha$ except the last part corresponds to a pattern $a_ia_{i+1}=01$ if it is greater than $1$.
There are therefore three choices for $b_ib_{i+1}$, namely $01,10$ or $11$.
In total this contributes the factor of $3^m$ in \eqref{eq:Va}.
All other entries of $b$ can be chosen arbitrarily, contributing a factor of $2^{n-2m-1}$.

In order to see Claim~\eqref{item:unimodal_gf} note that a pair of compatible sequences $a,b\in\{0,1\}^{n}$ is either empty, or it can be obtained from a pair of shorter compatible sequences $a',b'\in\{0,1\}^{n-1}$ by appending one of the four patterns
\begin{equation*}
\begin{matrix}
0\\0
\end{matrix}
\qquad
\begin{matrix}
1\\0
\end{matrix}
\qquad
\begin{matrix}
0\\1
\end{matrix}
\qquad
\text{or}
\qquad
\begin{matrix}
1\\1
\end{matrix}
\quad.
\end{equation*}
Sequences $a,b$ that are obtained in this way and do contain the forbidden pattern are of the form
\begin{equation*}
\begin{matrix}
a''&0&1\\
b''&0&0
\end{matrix}
\end{equation*}
where $a'',b''\in\{0,1\}^{n-2}$ avoid the forbidden pattern.
Thus we conclude the recursion
\begin{equation*}
F(q,t,z)
=
zt+(1+q+t+qt)zF(q,t,z)-qz^2F(q,t,z)
\end{equation*}
and Claim~\eqref{item:unimodal_gf} follows.
\end{proof}

The relation defined on the subsets of $[n-1]$ by letting $S\prec\compset{\alpha}$ if and only if $S$ is $\alpha$-unimodal is neither symmetric, antisymmetric nor transitive.
It follows from \cref{lem:unimodal}~\eqref{item:unimodal_gf} that the total number $f(n)$ of such relations satisfies the recursion $f(n) = 4f(n-1) - f(n-2)$ with $f(0)=0$ and $f(1)=1$, and is therefore equal to the sequence \oeis{A001353} in \cite{OEIS}.
This relates $\alpha$-unimodal sets to, for example, spanning trees in a $2\times n$ grid.

It is also easy to obtain $q$-analogues of \cref{lem:unimodal}~\eqref{item:unimodal_countU} and~\eqref{item:unimodal_Va} --- this is left as an exercise.

\medskip

In order to state the first main result of this section we need one more definition, 
which is due to C.~Ballantine et al.~and appears in the study of quasisymmetric 
analogues of the power sum symmetric functions~\cite{BallantineDaughertyHicksMason2017}.

Let $\alpha \leq \beta$ be compositions of $n$.
Given a permutation $\sigma\in\symS_n$ and $i\in[\ell(\alpha)]$ define the 
subword $\sigma^{(i)}\coloneqq\sigma_a\dotsm\sigma_b$ of $\sigma$ where $\compb{i}{\alpha}=[a,b]$.
A permutation $\sigma\in\symS_n$ is called \emph{consistent with $\alpha\leq\beta$} if the following two conditions are satisfied:
\begin{enumerate}[(i)]
\item For each $i\in[\ell(\alpha)]$ the maximum of $\sigma^{(i)}$ is in last position.
\item For each $k\in[\ell(\beta)]$ the subwords $\sigma^{(i)},\dots,\sigma^{(j)}$ 
are sorted increasingly with respect to their maximal elements, where $i,j\in[\ell(\alpha)]$ are determined by
\[
\bigcup_{r=i}^j\compb{r}{\alpha}=\compb{k}{\beta}
\,.
\]
\end{enumerate}
Let $\CONS{\alpha}{\beta}$ denote the set of permutations $\sigma \in \symS_n$ that are consistent with $\alpha\leq\beta$.

\begin{example}\label{ex:cons}
Let $\sigma=438756219$, $\alpha=12123$ and $\beta=315$.
We shall see that $\sigma \in \CONS{\alpha}{\beta}$.
We separate $\beta$-blocks by vertical lines and put $\alpha$-blocks into parentheses:
\[
(4)(38)|(7)|(56)(219)
\]
In each $\alpha$-block the maximum is in last position.
Moreover the maxima are increasing within each $\beta$-block:
\[
4<8\mid7\mid6<9
\]
Thus $\sigma\in\CONS{\alpha}{\beta}$.
On the other hand
\[
(4)(38)|(7)|(65)(219)
\notin\CONS{\alpha}{\beta}
\]
because the maximum of $\sigma^{(4)}=65$ is not in last position.
Similarly
\[
(4)(38)|(7)|(59)(216)
\notin\CONS{\alpha}{\beta}
\]
because the maxima in the third $\beta$-block are not in increasing order.
\end{example}

The definitions of both $\alpha$-unimodal permutations and consistent permutations are somewhat out of the blue at first glance, however, these two concepts interplay in an interesting fashion.
The following theorem will allow us to expand Gessel's fundamental basis into quasisymmetric power sums in \cref{sec:gessel}.

\begin{theorem}\label{thm:CONS}
Let $n\in\setN$ and $\beta,\gamma$ be compositions of $n$.
Then
\begin{equation}\label{eq:CONS}
\sum_{\alpha}
|\CONS{\alpha}{\beta}|\, (-1)^{|\compset{\gamma}\setminus \compset{\alpha}|}
=
\begin{cases}
n!&\quad \text{if }\beta \leq \gamma,\\
0&\quad \text{otherwise,}
\end{cases}
\end{equation}
where the sum ranges over all compositions $\alpha$ of $n$ such that $\alpha\leq\beta$ and $\compset{\gamma}$ is $\alpha$-unimodal.
\end{theorem}
\begin{proof}
Let $R(\beta,\gamma)$ denote the set of all compositions $\alpha$ of $n$ such that $\alpha\leq\beta$ and $\compset{\gamma}$ is $\alpha$-unimodal.

\noindent
\textbf{Case 1.}
First assume $\beta\leq\gamma$.
Then $\compset{\gamma}$ is $\alpha$-unimodal for all $\alpha\leq\beta$ by \cref{lem:unimodal}~\eqref{item:unimodal_leq}.
Moreover $\compset{\gamma}\setminus\compset{\alpha}=\emptyset$ for all $\alpha\in R(\beta,\gamma)$.
Thus it suffices to give a bijection
\[
\varphi:\bigsqcup_{\alpha\leq\beta}\CONS{\alpha}{\beta}\to\symS_{n}\,.
\]
This is accomplished simply by reversing each subword $\sigma^{(i)}$.
The same idea appears in the well-known method for switching between cycle notation 
and one line notation by forgetting the parentheses, see~\cite[Sec.~1.3]{StanleyEC1}.
Note that $\varphi^{-1}$ depends on $\beta$.

For example, consider $\beta=3$ when $\alpha$ varies over all compositions of $3$.
\[
\begin{matrix}
(123)&\to&321\\
(213)&\to&312\\
(12)(3)&\to&213\\
\end{matrix}
\qquad\qquad
\begin{matrix}
(1)(23)&\to&132\\
(2)(13)&\to&231\\
(1)(2)(3)&\to&123\\
\end{matrix}
\]
\medskip

\noindent
\textbf{Case 2.}
On the other hand, if $\beta\nleq\gamma$ then let $i$ be the minimal element of $\compset{\gamma}\setminus\compset{\beta}$.

\noindent
If $i-1\in\compset{\beta}^0$ then $i-1\in\compset{\alpha}^0$ for all $\alpha\in R(\beta,\gamma)$.
Moreover the set $R(\beta,\gamma)$ is partitioned into pairs $\{\alpha,\alpha'\}$ such that $\compset{\alpha'}=\compset{\alpha}\cup\{i\}$ and $\compset{\alpha}=\compset{\alpha'}\setminus\{i\}$.
We claim that
\[
\CONS{\alpha}{\beta}
=\CONS{\alpha'}{\beta}.
\]
This becomes clear via the correspondence
\[
\dotsm)|(\sigma_i\sigma_{i+1}\dotsm\sigma_j)\dotsm
\qquad\longleftrightarrow\qquad
\dotsm)|(\sigma_i)(\sigma_{i+1}\dotsm\sigma_j)\dotsm
\]
where we use the notation of \cref{ex:cons}.
Moreover,
\[
(-1)^{|\compset{\gamma}\setminus\compset{\alpha}|}
=-(-1)^{|\compset{\gamma}\setminus\compset{\alpha'}|}
\,.
\]
Thus the left hand side of \eqref{eq:CONS} vanishes.

If $i-1\notin\compset{\beta}^0$ then $i-1\notin\compset{\gamma}$ by minimality of $i$.
Thus for $\compset{\gamma}$ to be $\alpha$-unimodal we must have $i-1\in\compset{\alpha}$ or $i\in\compset{\alpha}$ by \cref{lem:unimodal}~\eqref{item:unimodal_iff}.
In fact, the set $R(\beta,\gamma)$ is partitioned into sets of three, $\{\alpha,\alpha',\alpha''\}$, such that
\begin{align*}
\compset{\alpha}
=
T\cup\{i-1\}
\,,
&&
\compset{\alpha'}
=
T\cup\{i\}
\,,
&&
\compset{\alpha''}
=
T\cup\{i-1,i\}
\end{align*}
for some $T\subseteq [n-1]\setminus \{i-1,i\}$. 
We claim that there exists a bijection
\[
\varphi:\CONS{\alpha}{\beta}\to
\CONS{\alpha'}{\beta}\sqcup
\CONS{\alpha''}{\beta}
\,.
\]
Define the map $\varphi$ as follows.
For $\sigma\in\CONS{\alpha}{\beta}$ set
\begin{equation*}
\varphi(\sigma)
\coloneqq
\begin{cases}
\sigma\circ(i-1,i)\in\CONS{\alpha'}{\beta}&\quad\text{if }\sigma_{i-1}>\sigma_i,\\
\sigma\in\CONS{\alpha''}{\beta}&\quad\text{if }\sigma_{i-1}<\sigma_i.
\end{cases}
\end{equation*}
With the notation from \cref{ex:cons} this amounts to the following:
\[
\dotsm(\sigma_h\dotsm\sigma_{i-1})(\sigma_{i}\sigma_{i+1}\dotsm\sigma_j)\dotsm
\quad\longleftrightarrow\quad
\begin{cases}
\dotsm(\sigma_h\dotsm\sigma_{i}\sigma_{i-1})(\sigma_{i+1}\dotsm\sigma_j)\dotsm\\
\dotsm(\sigma_h\dotsm\sigma_{i-1})(\sigma_{i})(\sigma_{i+1}\dotsm\sigma_j)\dotsm
\end{cases}
\]
The claim follows from
\[
(-1)^{|\compset{\gamma}\setminus\compset{\alpha}|} =
-(-1)^{|\compset{\gamma}\setminus\compset{\alpha'}|}
=-(-1)^{|\compset{\gamma}\setminus\compset{\alpha''}|}
\,.\qedhere
\]
\end{proof}

We now turn to labeled posets, for which we adopt the same conventions as in \cite{StanleyEC1}.
A \emph{labeled poset} $(P,w)$ is a finite poset $P$ equipped with a bijection $w:P\to[n]$.
We call $(P,w)$ a \emph{naturally labeled poset} if $w$ is order-preserving, that is, $w(\poseta)<w(\posetb)$ for all $\poseta,\posetb\in P$ with $\poseta<_P \posetb$.

The \emph{Jordan--Hölder set} of a labeled poset $(P,w)$ with $n$ elements is defined as
\[
\ILinExt(P,w) \coloneqq \big\{ \sigma \in \symS_n : \sigma^{-1}\circ w(\poseta) < \sigma^{-1}\circ w(\posetb) \text{ for all }\poseta,\posetb\in P\text{ with }\poseta<_P \posetb\big\}.
\]
That is, $\sigma\in\ILinExt(P,w)$ if and only if $\sigma^{-1}\circ w$ is a linear extension of $P$.

To avoid ambiguity we refer to the values $w(\poseta)$, where $\poseta\in P$, as \emph{labels}.
Other functions $f:P\to\setN$ (such as $\sigma^{-1}\circ w$, where $\sigma\in\ILinExt(P,w)$) we sometimes call \emph{colorings} and their values $f(\poseta)$ \emph{colors}.
With this convention the elements $\sigma\in\ILinExt(P,w)$ map colors to labels.

\begin{example}\label{ex:linExtExamples}
For example, let $(P,w)$ be the labeled poset below:
\begin{align*}
(P,w)=
\begin{tikzpicture}[xscale=0.8,yscale=0.8,baseline=25pt,
circ/.style={circle,draw,inner sep=0.8pt, minimum width=4pt}
]
\node[circ] (1) at ( 0, 0) {$1$};
\node[circ] (2) at ( 1, 0) {$2$};
\node[circ] (3) at ( 0, 1.6) {$3$};
\node[circ] (4) at ( 1, 1.6) {$4$};
\node[circ] (5) at ( 0.5, 3.2) {$5$};
\draw[black,thick] (4)--(1)--(3)--(5)--(4)--(2);
\end{tikzpicture}
&&
\ILinExt(P,w) = 
\begin{tikzpicture}[xscale=0.8,yscale=0.8,baseline=25pt,
circ/.style={circle,draw,inner sep=0.8pt, minimum width=4pt}
]
\node[circ] (1) at ( 0, 0) {$1$};
\node[circ] (2) at ( 1, 1.6) {$2$};
\node[circ] (3) at ( 0, 0.8) {$3$};
\node[circ] (4) at ( 1, 2.4) {$4$};
\node[circ] (5) at ( 0.5, 3.2) {$5$};
\draw[black,thick] (4)--(1)--(3)--(5)--(4)--(2);
\node (6) at ( 0.5,-1.0) {$13245$};
\end{tikzpicture},\;
\begin{tikzpicture}[xscale=0.8,yscale=0.8,baseline=25pt,
circ/.style={circle,draw,inner sep=0.8pt, minimum width=4pt}
]
\node[circ] (1) at ( 0, 0) {$1$};
\node[circ] (2) at ( 1, 0.8) {$2$};
\node[circ] (3) at ( 0, 1.6) {$3$};
\node[circ] (4) at ( 1, 2.4) {$4$};
\node[circ] (5) at ( 0.5, 3.2) {$5$};
\draw[black,thick] (4)--(1)--(3)--(5)--(4)--(2);
\node (6) at ( 0.5,-1.0) {$12345$};
\end{tikzpicture},\;
\begin{tikzpicture}[xscale=0.8,yscale=0.8,baseline=25pt,
circ/.style={circle,draw,inner sep=0.8pt, minimum width=4pt}
]
\node[circ] (1) at ( 0, 0.8) {$1$};
\node[circ] (2) at ( 1, 0.0) {$2$};
\node[circ] (3) at ( 0, 1.6) {$3$};
\node[circ] (4) at ( 1, 2.4) {$4$};
\node[circ] (5) at (0.5,3.2) {$5$};
\draw[black,thick] (4)--(1)--(3)--(5)--(4)--(2);
\node (6) at ( 0.5,-1.0) {$21345$};
\end{tikzpicture},\;
\begin{tikzpicture}[xscale=0.8,yscale=0.8,baseline=25pt,
circ/.style={circle,draw,inner sep=0.8pt, minimum width=4pt}
]
\node[circ] (1) at ( 0, 0) {$1$};
\node[circ] (2) at ( 1, 0.8) {$2$};
\node[circ] (3) at ( 0, 2.4) {$3$};
\node[circ] (4) at ( 1, 1.6) {$4$};
\node[circ] (5) at ( 0.5, 3.2) {$5$};
\draw[black,thick] (4)--(1)--(3)--(5)--(4)--(2);
\node (6) at ( 0.5,-1.0) {$12435$};
\end{tikzpicture},\;
\begin{tikzpicture}[xscale=0.8,yscale=0.8,baseline=25pt,
circ/.style={circle,draw,inner sep=0.8pt, minimum width=4pt}
]
\node[circ] (1) at ( 0, 0.8) {$1$};
\node[circ] (2) at ( 1, 0.0) {$2$};
\node[circ] (3) at ( 0, 2.4) {$3$};
\node[circ] (4) at ( 1, 1.6) {$4$};
\node[circ] (5) at ( 0.5, 3.2) {$5$};
\draw[black,thick] (4)--(1)--(3)--(5)--(4)--(2);
\node (6) at ( 0.5,-1.0) {$21435$};
\end{tikzpicture}
\end{align*}
Thus $\ILinExt(P,w)  = \{13245,12345,21345,12435,21435\}$.
\end{example}

Given a labeled poset $(P,w)$ with $n$ elements and a composition $\alpha$ of $n$ with $\ell$ parts, let
\[
\ILinExt_{\alpha}(P,w)\coloneqq
\{ \sigma \in \ILinExt(P,w) : \sigma\text{ is $\alpha$-unimodal}\}.
\]
Furthermore, given $\sigma\in\symS_n$, define the subposets
\begin{equation}\label{eq:subposet}
\compp{i}{\alpha}(\sigma)\coloneqq
\{w^{-1}(\sigma_j):j\in\compb{i}{\alpha}\}\subseteq P
\,.
\end{equation}
Finally denote by $\ILinExt^\ast_{\alpha}(P,w)$ the subset of $\ILinExt_{\alpha}(P,w)$ 
that consists of the elements $\sigma\in\ILinExt_{\alpha}(P,w)$ such 
that $\compp{i}{\alpha}(\sigma)$ contains a unique minimal element for all $i\in[\ell]$.

\begin{example}*\label[example]{ex:ILinExtast}
Let $(P,w)$ be as in \cref{ex:linExtExamples}. 
If $\alpha = (2,3)$, then 
\[
\ILinExt_{23}(P,w) = \{
\hat{1}3|\hat{2}45, 
\hat{1}\hat{2}|\hat{3}\hat{4}5, 
\hat{2}\hat{1}|\hat{3}\hat{4}5, 
\hat{1}\hat{2}|\hat{4}\hat{3}5,
\hat{2}\hat{1}|\hat{4}\hat{3}5
\}
\quad\text{and}\quad
\ILinExt^\ast_{23}(P,w) = \{13245\}
,
\]
where we have marked the labels of the minimal elements in each subposet.
Similarly, if $\alpha=(4,1)$ then
\[
\ILinExt_{41}(P,w) = \{
\hat{1}\hat{2}34|\hat{5}, 
\hat{2}\hat{1}34|\hat{5}
\}
\quad\text{and}\quad
\ILinExt^\ast_{41}(P,w) = \emptyset.
\qeddefhere
\]
\end{example}

The set $\ILinExt_{\alpha}^{\ast}(P,w)$ 
associated to a naturally labeled poset, turns out to be a highly useful concept in the study of power sum symmetric functions.
The following lemma gives a necessary condition for a permutation $\sigma$ to lie in $\ILinExt_{\alpha}^{\ast}(P,w)$.

\begin{lemma}\label[lemma]{lemma:no_descents}
Let $(P,w)$ be a naturally labeled poset with $n$ elements, let $\alpha$ 
be a composition of $n$, and let $\sigma\in\ILinExt_{\alpha}^{\ast}(P,w)$.
Then $\DES(\sigma)\subseteq\compset{\alpha}$.
\end{lemma}

\begin{proof}
Let $i\in[\ell(\alpha)]$ and $\compb{i}{\alpha}=[a,b]$.
Suppose $\poseta,\posetb\in\compp{i}{\alpha}(\sigma)$ with $\poseta<_Py$.
Then $w(\poseta)<w(\posetb)$ since $w$ is natural, and $\sigma^{-1}\circ w(\poseta)<\sigma^{-1}\circ w(\posetb)$ by definition of $\ILinExt(P,w)$.
Consequently the pair $(\sigma^{-1}\circ w(\poseta),\sigma^{-1}\circ w(\posetb))$ is a noninversion of $\sigma$.
Since $\compp{i}{\alpha}(\sigma)$ has a unique minimal element we obtain $\sigma_a<\sigma_j$ for all $j\in[a+1,b]$.
Thus by $\alpha$-unimodality
$\sigma_{a}<\dotsm<\sigma_{b}$.
\end{proof}

\cref{ex:ILinExtast} shows that $\DES(\sigma)\subseteq\compset{\alpha}$ is not a sufficient condition for $\sigma\in\ILinExt_{\alpha}^{\ast}(P,w)$.

We now come to the second main result of this section.
The following theorem will be used to prove that several families of quasisymmetric or symmetric functions expand positively into power sums.

\begin{theorem}\label{thm:ILinExt}
Let $(P,w)$ be a naturally labeled poset 
with $n$ elements, and let $\alpha$ be a composition of $n$.
Then
\begin{equation}\label{eq:L_sum}
\sum_{\sigma \in \ILinExt_{\alpha}(P,w)}
(-1)^{|\DES(\sigma)\setminus \compset{\alpha}|}
=|\ILinExt^\ast_{\alpha}(P,w)|
\,.
\end{equation}
In particular, the left-hand side of \refq{L_sum} is nonnegative.
\end{theorem}

\begin{proof}
We prove the theorem by the use of a sign-reversing involution
\[
\varphi:\ILinExt_{\alpha}(P,w)\setminus\ILinExt_{\alpha}^{\ast}(P,w)\to\ILinExt_{\alpha}(P,w)\setminus\ILinExt_{\alpha}^{\ast}(P,w)
\]
similar to an involution due to B.~Ellzey~\cite[Thm.~4.1]{Ellzey2016}.

Let $\ell$ be the number of parts of $\alpha$, and let $\sigma\in\ILinExt_{\alpha}(P,w)\setminus\ILinExt_{\alpha}^{\ast}(P,w)$.
Then there exists a minimal index $i\in[\ell]$ such that $\compp{i}{\alpha}(\sigma)$ has at least two minimal elements.
Suppose $\compb{i}{\alpha}=[a,b]$.
Since $\sigma$ is $\alpha$-unimodal there exists $k\in\compb{i}{\alpha}$ with
\[
\sigma_{a}>\dotsm>\sigma_k<\dotsm<\sigma_{b}.
\]
Define $M\in[n]$ as
\[
M
\coloneqq
\max\big\{w(\poseta):
\poseta\in\compp{i}{\alpha}(\sigma)\text{ and }
\poseta\leq_Py\text{ for all } \posetb\in\compp{i}{\alpha}(\sigma)\big\}.
\]
That is, $M$ is the maximal label of a minimal element of $\compp{i}{\alpha}(\sigma)$.
Since $\sigma$ is $\alpha$-unimodal, there exist indices $j,m \in\compb{i}{\alpha}$ with $j<m$ and $j\leq k\leq m$ such that
\[
[j,m]=\{r\in\compb{i}{\alpha}:\sigma_r\leq M\}.
\]
In particular, $M=\sigma_j$ or $M=\sigma_m$.
We distinguish between these two cases.

\noindent
\textbf{Case 1.} If $M=\sigma_j$ then set
$\varphi(\sigma)\coloneqq \sigma\circ(j,\dotsc, m)$.

\noindent
\textbf{Case 2.} If $M=\sigma_m$ then set
$\varphi(\sigma)\coloneqq  \sigma\circ(j, \dotsc,m)^{-1}$.

It is straightforward to verify that $\varphi(\sigma)$ is $\alpha$-unimodal in both cases.
Moreover, $\compp{r}{\alpha}(\sigma)=\compp{r}{\alpha}(\varphi(\sigma))$ for all $r\in[\ell]$, so by the definition of $\ILinExt_{\alpha}^{\ast}(P,w)$ we have $\varphi(\sigma)\notin\ILinExt_{\alpha}^{\ast}(P,w)$.
Another consequence is that $\varphi(\sigma)^{-1}\circ w(\poseta)<\varphi(\sigma)^{-1}\circ w(\posetb)$ for all $\poseta\in\compp{r}{\alpha}(\sigma)$ and $\posetb\in\compp{s}{\alpha}(\sigma)$ with $r<s$.
To show that $\varphi$ is well-defined, it therefore suffices to verify that the restriction of $\varphi(\sigma)^{-1}\circ w$ to $\compp{i}{\alpha}(\sigma)$ is a linear extension.

Suppose we are in Case~1 and set $\poseta=w^{-1}(M)$.
In order to prove $\varphi(\sigma)\in\ILinExt(P,w)$ we need to verify
\[
\varphi(\sigma)^{-1}\circ w(\poseta)<\varphi(\sigma)^{-1}\circ w(\posetb)
\]
for all $\posetb\in\compp{i}{\alpha}(\sigma)$ with $\poseta<_Py$.
Thus assume that $\posetb\in\compp{i}{\alpha}(\sigma)$ satisfies $\poseta<_Py$.
Then $M=w(\poseta)<w(\posetb)$ implies $\sigma^{-1}\circ w(\posetb)>j$.
By the defining property of $j$ and $m$ we also have $\sigma^{-1}\circ w(\posetb)>m$, and therefore
\[
\varphi(\sigma)^{-1}\circ w(\poseta)
=(j,\dotsc,m)^{-1}\circ\sigma^{-1}(M)
=m<\sigma^{-1}\circ w(\posetb)
=\varphi(\sigma)^{-1} \circ w(\posetb),
\]
as claimed.

Next assume we are in Case~2 and set $\posetb=w^{-1}(M)$.
To show that $\varphi(\sigma)\in\ILinExt(P,w)$ we need to verify
\[
\varphi(\sigma)^{-1}\circ w(\poseta)<\varphi(\sigma)^{-1}\circ w(\posetb)
\]
for all $\poseta\in\compp{i}{\alpha}(\sigma)$ with $\poseta<_Py$.
But this is trivially true because $\posetb$ is a minimal element of $\compp{i}{\alpha}(\sigma)$.

To see that $\varphi$ is an involution note that $M>\sigma_k$ since $\compp{i}{\alpha}(\sigma)$ 
has at least two minimal elements by assumption.
Thus $\sigma$ belongs to Case~1 if and only if $\varphi(\sigma)$ belongs to Case~2.

It is also clear that $\varphi$ is sign-reversing.
Indeed
\[
|\DES(\varphi(\sigma))\setminus \compset{\alpha}|
=|\DES(\sigma)\setminus \compset{\alpha}|-1
\]
whenever $\sigma$ belongs to Case~1, and equivalently, 
\[
|\DES(\varphi(\sigma))\setminus \compset{\alpha}|
=|\DES(\sigma)\setminus \compset{\alpha}|+1
\]
if $\sigma$ belongs to Case~2.
Hence
\[
\sum_{\sigma\in\ILinExt_{\alpha}(P,w)}(-1)^{|\DES(\sigma)\setminus \compset{\alpha}|}
=\sum_{\sigma\in\ILinExt_{\alpha}^{\ast}(P,w)}(-1)^{|\DES(\sigma)\setminus \compset{\alpha}|}
\,.
\]
The claim in \eqref{eq:as_sum} now follows from \cref{lemma:no_descents}, which guarantees that
$\DES(\sigma)\subseteq \compset{\alpha}$
for all $\sigma\in\ILinExt_{\alpha}^{\ast}(P,w)$.
\end{proof}

\section{
Quasisymmetric functions}
\label[section]{sec:gessel}

The main result of this section, \cref{thm:gesselInPSum}, 
is the expansion of Gessel's fundamental basis into quasisymmetric power sums.
As an immediate consequence we obtain a new proof of a recent 
result of C.~Athanasiadis~\cite[Prop.~3.2]{Athanasiadis15} (\cref{prop:fundToPowerSumLift} below).
We start out with a brief introduction to quasisymmetric functions.
For more background the reader is referred to~\cite{StanleyEC2,LuotoEtAl2013IntroQSymSchur}.

A \emph{quasisymmetric function} $f$ is a formal power series $f \in \setQ[[x_1,x_2,\dotsc]]$
such that the degree of $f$ is finite, and for every composition $(\alpha_1,\dotsc,\alpha_\ell)$
the coefficient of $\xvar_{i_2}^{\alpha_2} \dotsm \xvar_{i_\ell}^{\alpha_\ell}$ in $f$
is the same for all integer sequences $1\leq i_1 < i_2 < \dotsb < i_\ell$.

Given a composition $\alpha$ with $\ell$ parts, the \emph{monomial quasisymmetric function} $\qmonom_\alpha$
is defined as
\[
\qmonom_\alpha(\xvec) \coloneqq \sum_{i_1 < i_2 < \dotsb < i_\ell} \xvar_{i_1}^{\alpha_1} \xvar_{i_2}^{\alpha_2} \dotsm \xvar_{i_\ell}^{\alpha_\ell}
\,.
\]
The functions $\qmonom_\alpha$, where $\alpha$ ranges over all compositions of $n$,
constitute a basis for the space of homogeneous quasisymmetric functions of degree $n$.

Another basis for the space of quasisymmetric functions are the fundamental quasisymmetric functions.
The \emph{fundamental quasisymmetric functions} of degree $n$ are indexed by subsets $S\subseteq [n-1]$ and defined as
\begin{align*}
\gessel_{n,S}(\xvec) \coloneqq  \sum_{\substack{j_1 \leq j_2 \leq \dotsc \leq j_n \\ i\in S \Rightarrow j_i<j_{i+1}  }} \xvar_{j_1}\dotsm \xvar_{j_n}
\,.
\end{align*}
Alternatively, given a composition $\alpha$ of $n$, we sometimes write $\gessel_\alpha\coloneqq \gessel_{n,\compset{\alpha}}$.
The expansion of fundamental quasisymmetric functions into monomial quasisymmetric functions is given by
\begin{equation}\label{eq:GesselintoMonomial}
\gessel_{\alpha}(\xvec)
=
\sum_{\beta \leq \alpha} \qmonom_\beta(\xvec)
\,.
\end{equation}
Given a composition $\alpha$ set
$z_\alpha
\coloneqq\prod_{i\geq1}
i^{m_i} m_i!$,
where $m_i$ denotes the number of parts of $\alpha$ that are equal to $i$.
The following definitions appear in \cite{BallantineDaughertyHicksMason2017}.
Given compositions $\alpha \leq \beta$, let
\begin{align}\label{eq:hookValues}
\pi(\alpha)
\coloneqq
\prod_{i=1}^{\ell(\alpha)}
(\alpha_1+\alpha_2+\dotsb + \alpha_i)
&&
\text{and}
&&
\pi(\alpha,\beta)
\coloneqq
\prod_{i=1}^{\ell(\beta)} \pi\big( \alpha^{(i)} \big)
\,,
\end{align}
where $\alpha^{(i)}$ is the composition of $\beta_i$ that consists of the parts $\alpha_j$ with $\compb{j}{\alpha}\subseteq\compb{i}{\beta}$.
The \emph{quasisymmetric power sum} $\Psi_\alpha$ is defined as
\begin{equation}\label{eq:PsiintoMonomial}
\Psi_\alpha(\xvec) \coloneqq z_\alpha \sum_{\beta \geq \alpha } \frac{1}{\pi(\alpha,\beta)} \qmonom_\beta(\xvec).
\end{equation}
For example,
\[
\Psi_{231} = 
\frac{1}{10}\qmonom_6 + 
\frac{1}{4}\qmonom_{24}+
\frac{3}{5}\qmonom_{51}+\qmonom_{231}.
\]
The quasisymmetric power sums refine the power sum symmetric functions as 
\begin{equation}\label{eq:powersum_Psi_expansion}
\psumP_\lambda(\xvec) = \sum_{\alpha \sim \lambda} \Psi_\alpha(\xvec)
\,,
\end{equation}
where the sum ranges over all compositions $\alpha$ whose parts rearrange to $\lambda$.
This is shown in \cite[Thm.~3.11]{BallantineDaughertyHicksMason2017} and we also give an alternative proof in \cref{sec:p}.

Let $\omega$ be the automorphism on quasisymmetric functions defined by $\omega(\gessel_S) = \gessel_{[n-1]\setminus (n-S)}$.
In particular, on the classical symmetric functions $\omega$ acts as
$\omega (\elementaryE_\lambda) = \completeH_\lambda$, $\quad \omega (\schurS_\lambda) = \schurS_{\lambda'}$, and
$\omega (\psumP_\lambda) = (-1)^{|\lambda|-\length(\lambda)}\psumP_\lambda$.
%
We remark that
\begin{equation}
\omega\left( \Psi_\alpha \right) = (-1)^{|\alpha|-\length(\alpha)}\Psi_{\alpha^r}
\,,
\end{equation}
where $\alpha^r$ denotes the reverse of $\alpha$, see \cite[Sec.~4]{BallantineDaughertyHicksMason2017}.

The following expansion of the fundamental basis into quasisymmetric power sums is the main result of this section.

\begin{theorem}\label{thm:gesselInPSum}
Let $n\in\setN$ and $S \subseteq [n-1]$.
Then
\begin{equation}\label{eq:gesselInPSum}
\gessel_{n,S}(\xvec)
= \sum_{\alpha} \frac{\Psi_{\alpha}(\xvec)}{z_{\alpha}} (-1)^{|S\setminus \compset{\alpha}|}
\,,
\end{equation}
where the sum ranges over all compositions $\alpha$ of $n$ such that $S$ is $\alpha$-unimodal.
\end{theorem}
\begin{proof}
Expanding both sides of \eqref{eq:gesselInPSum} in the monomial basis according to~\eqref{eq:GesselintoMonomial} and~\eqref{eq:PsiintoMonomial} we obtain
\[
\sum_{\beta \leq \gamma}
\qmonom{_\beta}(\xvec)
= 
\sum_{\substack{ \alpha \vDash n \\ 
\compset{\gamma}\in U_{\alpha}}}
\sum_{\beta \geq \alpha }
\frac{1}{\pi(\alpha,\beta)} \qmonom{_\beta}(\xvec)\,  (-1)^{|\compset{\gamma}\setminus \compset{\alpha}|}
\,.
\]
Comparing coefficients of $\qmonom_{\beta}$, it suffices to prove that for all compositions $\beta$ and $\gamma$ of $n$, we have
\begin{equation}\label{eq:Mbeta_coeff}
\sum_{\alpha\in R(\beta,\gamma)}
\frac{1}{\pi(\alpha,\beta)}(-1)^{|\compset{\gamma}\setminus \compset{\alpha}|} =
\begin{cases}
1 &\quad\text{if }\beta \leq \gamma,\\
0 &\quad\text{otherwise,}
\end{cases}
\end{equation}
where $R(\beta,\gamma)$ denotes the set of all compositions $\alpha\leq\beta$ such that $\compset{\gamma}$ is $\alpha$-unimodal.

It is shown in~\cite[Lemma 3.7]{BallantineDaughertyHicksMason2017} that $n! = |\CONS{\alpha}{\beta}|\cdot \pi(\alpha,\beta)$. 
We give a short alternative proof of this identity in \cref{prop:treeHook} below using the hook-length formula for forests.
After multiplying both sides of \refq{Mbeta_coeff} by $n!$ the claim follows from \cref{thm:CONS}.
\end{proof}

\begin{proposition}\label{prop:treeHook}
Let $n\in\setN$ and $\alpha \leq \beta$ be compositions of $n$.
Then 
\[
|\CONS{\alpha}{\beta}|\cdot \pi(\alpha,\beta) = n!
\,.
\]
\end{proposition}
\begin{proof}
To the pair $(\alpha,\beta)$ we associate a (labeled) rooted forest on the vertices $[n]$.
For $i\in[\ell(\alpha)]$ set $s_i\coloneqq\alpha_{1}+\dotsm+\alpha_{i}$.
For all $i\in[\ell(\alpha)]$ and all $j\in\compb{i}{\alpha}\setminus\{s_i\}$ add an edge from $s_i$ to $j$.
Moreover if $s_i,s_{i+1}\in\compb{k}{\beta}$ for some $k\in[\ell(\beta)]$ then add an edge from $s_{i+1}$ to $s_i$.
For example, if $\alpha=2312$ and $\beta=62$ then the forest is shown in \cref{fig:hookTree}.

\begin{figure}[b]
\begin{tikzpicture}[scale=.8,baseline=-10mm,
circ/.style={circle,draw,inner sep=0.8pt, minimum width=12pt,font=\footnotesize}]
\begin{scope}
\node[circ] (1)  at (0,0) {$2$};
\node[circ] (11) at (0,-1) {$1$};
\node[circ] (2)  at (1.5,1) {$5$};
\node[circ] (21) at (1,0) {$3$};
\node[circ] (22) at (2,0) {$4$};
\node[circ] (3)  at (3,2) {$6$};
\node[circ] (4)  at (4.5,1) {$8$};
\node[circ] (41)  at (4.5,0) {$7$};
\draw[black,thick,->] (1)--(11);
\draw[black,thick,->] (2)--(21);
\draw[black,thick,->] (2)--(22);
\draw[black,thick,->] (4)--(41);
\draw[black,thick,->] (3)--(2);
\draw[black,thick,->] (2)--(1);
\end{scope}
\begin{scope}[xshift=8cm]
\node[circ] (1)  at (0,0) {$2$};
\node[circ] (11) at (0,-1) {$1$};
\node[circ] (2)  at (1.5,1) {$5$};
\node[circ] (21) at (1,0) {$1$};
\node[circ] (22) at (2,0) {$1$};
\node[circ] (3)  at (3,2) {$6$};
\node[circ] (4)  at (4.5,1) {$2$};
\node[circ] (41)  at (4.5,0) {$1$};
\draw[black,thick,->] (1)--(11);
\draw[black,thick,->] (2)--(21);
\draw[black,thick,->] (2)--(22);
\draw[black,thick,->] (4)--(41);
\draw[black,thick,->] (3)--(2);
\draw[black,thick,->] (2)--(1);
\end{scope}
\end{tikzpicture}
\caption{Left: The forest associated with the compositions $\alpha=2312$ and $\beta=62$.
Right: The hook lengths of the vertices.
}\label{fig:hookTree}
\end{figure}
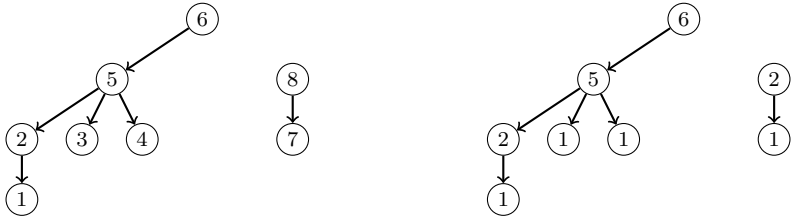

The \emph{hook length} of a vertex in the forest is defined as the size of the subtree rooted at that vertex, see \cref{fig:hookTree}.
From the definitions in \eqref{eq:hookValues} it is straightforward to see that the product of all hook lengths in the forest is $\pi(\alpha,\beta)$.
Furthermore it is immediate from the definition that the linear extensions of the forest can be identified with $\CONS{\alpha}{\beta}$.
The hook-length formula for counting linear extensions of forests by D.~Knuth \cite[Chap.~5.1.4, Ex.~20]{Knuth1998ArtOfProgramming} then implies that $|\CONS{\alpha}{\beta}|\cdot \pi(\alpha,\beta) = n!$ as claimed.
\end{proof}

\cref{thm:gesselInPSum} immediately implies the following result.

\begin{corollary}\label[corollary]{cor:qsym_Athanasiadis}
Let $X$ be a quasisymmetric function with
\begin{equation*}
X(\xvec)
=\sum_{S \subseteq [n-1]} c_S \gessel_{n,S}(\xvec).
\end{equation*}
Then
\begin{equation*}
X(\xvec) = \sum_{\alpha \vDash n} \frac{\Psi_{\alpha}(\xvec)}{z_{\alpha}} 
\sum_{S\in U_\alpha} (-1)^{|S\setminus \compset{\alpha}|} c_S
\,.
\end{equation*}
\end{corollary}

From \cref{cor:qsym_Athanasiadis} and \eqref{eq:powersum_Psi_expansion} we obtain a new proof of the following result due to C.~Athanasiadis for symmetric functions, which has served as inspiration for \cref{thm:gesselInPSum}.

\begin{corollary}
\label{prop:fundToPowerSumLift}
\textnormal{(\cite[Prop.~3.2]{Athanasiadis15}, see also \cite[Prop.~9.3]{AdinRoichman2015})}
Let $X$ be a symmetric function with
\begin{equation*}
X(\xvec) = \sum_{S \subseteq [n-1]} c_S \gessel_{n,S}(\xvec).
\end{equation*}
Then
\begin{equation*}
X(\xvec) = \sum_{\lambda \vdash n} \frac{\psumP_{\lambda}(\xvec)}{z_\lambda}
\sum_{S\in U_\lambda} (-1)^{|S\setminus \compset{\lambda}|} c_S.
\end{equation*}
\end{corollary}

\section{Reverse \texorpdfstring{$P$}{P}-partitions}\label[section]{sec:revPP}

In this section it is shown that the generating function of reverse $P$-partitions expands positively into quasisymmetric power sums for all posets $P$.

The theory of $P$-partitions was developed by R.~Stanley~\cite{Stanley1972}, and has numerous applications in the world of quasisymmetric functions.
Let $(P,w)$ be a labeled poset.
A function $f:P\to\setP$ is called \emph{reverse $(P,w)$-partition} 
if it satisfies the following two properties for all $\poseta,\posetb\in P$:
\begin{enumerate}[(i)]
\item If $\poseta\leq_Py$ then $f(\poseta)\leq f(\posetb)$.
\item If $\poseta<_Py$ and $w(\poseta)>w(\posetb)$ then $f(\poseta)<f(\posetb)$.
\end{enumerate}
Let $\rpp(P,w)$ be the set of reverse $(P,w)$-partitions, 
and denote the generating function of reverse $(P,w)$-partitions by
\begin{align}
K_{P,w}(\xvec)
\coloneqq
\sum_{f\in\rpp(P,w)} 
\prod_{\poseta\in P}\xvar_{f(\poseta)}
.
\end{align}
If $(P,w)$ is naturally labeled, then reverse $(P,w)$-partitions are just order-preserving maps $f:P\to\setP$, which are also called \emph{reverse $P$-partitions}.
Denote the set of reverse $P$-partitions by $\rpp(P)$ and the corresponding generating function by
\[
K_{P}(\xvec)
=
\sum_{\substack{f:P \to \setP \\ 
\poseta <_P \posetb \Rightarrow f(\poseta) \leq f(\posetb)}}
\prod_{\poseta\in P}\xvar_{f(\poseta)}
.
\]
If instead $w$ is order-reversing then reverse $(P,w)$-partitions are \emph{strict reverse $P$-partitions}.

The expansion of $K_{P,w}$ into fundamental quasisymmetric functions is a well-known result.

\begin{lemma}[{\cite[Cor.~7.19.5]{StanleyEC2}}]
\label{lem:fExpOfKP}
Let $(P,w)$ be a labeled poset with $n$ elements.
Then the series $K_{P,w}(\xvec)$ is quasisymmetric and its expansion into fundamental quasisymmetric functions is given by
\[
K_{P,w}(\xvec) = \sum_{\sigma \in \ILinExt(P,w) }
\gessel_{n,\DES(\sigma)}(\xvec)
\,.
\]
\end{lemma}

The following theorem is the main result of this section.
\begin{theorem}\label{thm:KP_positive}
Let $(P,w)$ be a naturally labeled poset with $n$ elements.
Then
\begin{equation}\label{eq:as_sum}
K_{P}(\xvec)
=\sum_{\alpha\vDash n} \frac{\Psi_{\alpha}(\xvec)}{z_{\alpha}} |\ILinExt^\ast_{\alpha}(P,w)|
\,.
\end{equation}
In particular, the quasisymmetric function $K_P$ is $\Psi$-positive.
\end{theorem}

\begin{proof}
By \cref{lem:fExpOfKP}
\[
K_{P}(\xvec)
=K_{P,w}(\xvec)
=\sum_{\sigma \in \ILinExt(P,w) }\gessel_{n,\DES(\sigma)}(\xvec)
\,.
\]
Expanding the right hand side into quasisymmetric power sums according to \cref{thm:gesselInPSum}, we obtain
\begin{align*}
K_P(\xvec)
&=\sum_{\sigma\in\ILinExt(P,w)}
\sum_{\substack{\alpha\vDash n \\ \DES(\sigma)\in U_\alpha}}
\frac{\Psi_{\alpha}(\xvec)}{z_{\alpha}} (-1)^{|\DES(\sigma)\setminus \compset{\alpha}|}
\\
&=
\sum_{\alpha\vDash n}
\frac{\Psi_{\alpha}(\xvec)}{z_{\alpha}}
\sum_{\sigma\in\ILinExt_{\alpha}(P,w)}
(-1)^{|\DES(\sigma)\setminus \compset{\alpha}|}
\,.
\end{align*}
The claim follows directly from \cref{thm:ILinExt}.
\end{proof}

The following is an immediate consequence of \cref{thm:KP_positive} and \eqref{eq:powersum_Psi_expansion}.

\begin{corollary}\label{cor:symmetric_ppositive}
Let $\posets$ be a finite set of posets.
If the sum
\[
\sum_{P\in\posets} K_P(\xvec)
\]
is symmetric, then it is $p$-positive.
\end{corollary}

We obtain analogous results for the generating functions of strict reverse $P$-partitions.
\begin{lemma}[{See also \cite[Chap.~3.3.2]{LuotoEtAl2013IntroQSymSchur}}]
Let $(P,w)$ be a labeled poset with $n$ elements. 
Then the automorphism $\omega$ acts in an order-reversing manner as
\[
\omega K_{P,w}(\xvec) = K_{P,n+1-w}(\xvec)
\]
where $(n+1-w)(\poseta)\coloneqq n+1-w(\poseta)$.
\end{lemma}
\begin{proof}
This follows from the fact that
\[
\sigma^{-1} \in \ILinExt(P,w) \Longleftrightarrow \tau^{-1} \in \ILinExt(P,n+1-w)
\,,
\]
where $\tau_i=n+1-\sigma_i$, and that 
\[
S = \DES(\sigma^{-1}) \Longleftrightarrow \DES(\tau^{-1}) = [n-1]\setminus (n-S)
\,.\qedhere
\]
\end{proof}
In particular, the generating function of \emph{strict} reverse $P$-partitions is given by $\omega K_P(\xvec)$.

\begin{corollary}\label{cor:strictOmegaPsiPos}
Let $(P,w)$ be a labeled poset such that $w$ is order-reversing.
Then $\omega K_{P,w}$ is $\Psi$-positive.
\end{corollary}

\begin{corollary}\label{cor:strictOmegaPos}
Let $\posets$ be a finite family of labeled posets $(P,w)$, all of which are equipped with an order-reversing map $w$.
If
\[
X(\xvec)=\sum_{(P,w)\in\posets}K_{P,w}(\xvec)
\]
is symmetric, then $\omega X$ is $p$-positive.
\end{corollary}

Note that $K_{P,w}$ is not in general $\Psi$-positive if $w$ is not a natural labeling.
Moreover positive linear combinations of quasisymmetric functions $K_{P,w}$ are not in general $p$-positive whenever they are symmetric.

For example, \cref{cor:strictOmegaPos} shows that $K_{P,w}$ is not in general $\Psi$-positive if $w$ is order-reversing, because the automorphism $\omega$ can introduce signs.

Another example is provided by Schur functions $s_{\lambda}$, which are shown to be special cases of the functions $K_{P,w}$ in \cite[Sec.~7.19]{StanleyEC2}.
The expansion of Schur functions into power sum symmetric functions is given by the celebrated Murnaghan--Nakayama rule, which is not positive, see \cite[Cor.~7.17.5]{StanleyEC2}.

It is conjectured \cite[p. 81]{Stanley1972} that $K_{P,w}$ is symmetric if and only if it is a skew Schur function.
Among these, $w$ is order-preserving only if $K_{P,w}$ is equal to a complete 
homogeneous symmetric function $\completeH_{\lambda}$.
It is an interesting question whether the expansion of the quasisymmetric 
functions $K_{P,w}$ into quasisymmetric power sums can be used to obtain new insights in this regard.

\section{Order-preserving surjections}
\label[section]{sec:opsurj}

The purpose of this section is to present a different characterisation of the set $\ILinExt_{\alpha}^{\ast}(P,w)$ associated to a naturally labeled poset $(P,w)$.
In the process we eschew $\alpha$-unimodal linear extensions in exchange for order-preserving surjections.
This new point of view is then used to compute three examples that are related to matroid quasisymmetric functions, chromatic quasisymmetric functions and Eulerian quasisymmetric functions.
Moreover in \cref{thm:KP_opsurj} below we formulate an equivalent version of \cref{thm:KP_positive} using order-preserving surjections.

Let $P$ be a poset with $n$ elements.
Denote by $\opsurj(P)$ the set of order-preserving surjections $f:P\to[k]$ for some $k\in\setN$.
Define the \emph{type} of a surjection $f:P\to[k]$ as
\begin{equation*}
\alpha(f)\coloneqq(\abs{f^{-1}(1)},\dotsc,\abs{f^{-1}(k)})
.
\end{equation*}
Thus $\alpha(f)$ is a composition of $n=\abs{P}$ with $k$ parts.
Denote by $\opsurj_{\alpha}(P)$ the set of order-preserving surjections of $P$ with type $\alpha$.

Furthermore, let $\opsurj^{\ast}(P)$ denote the set of order-preserving surjections $f\in\opsurj(P)$ such that $f^{-1}(i)$ contains a unique minimal element for all $i\in[\ell(\alpha(f))]$.
That is, for all $\posetb,\posetc\in P$ with $f(\posetb)=f(\posetc)$ there exists $\poseta\in P$ with $\poseta\leq_P\posetb$ and $\poseta\leq_P\posetc$ and $f(\poseta)=f(\posetb)$.
Set $\opsurj_{\alpha}^{\ast}(P)\coloneqq\opsurj_{\alpha}(P)\cap\opsurj^{\ast}(P)$.

Note that $\opsurj_{(1^n)}(P)=\opsurj_{(1^n)}^{\ast}(P)$ is just the set of linear extensions of $P$.
Compared to linear extensions, order-preserving surjections onto chains have not received much explicit attention.
Nevertheless some combinatorial objects can be regarded as order-preserving surjections in disguise.
For example, $\opsurj_{(2^m)}^{\ast}(\lambda)$, where $\lambda$ denotes the Young diagram of a partition of $n=2m$, is just the set of domino tableaux of shape $\lambda$, see \cite{Leeuwen2000}.
The set $\opsurj_{(2^m)}^{\ast}(P)$ can be taken as the definition of $P$-domino tableaux \cite[Sec.~4]{Stanley2005}.
Similarly, ribbon tableaux which appear in the study of Schur functions can be seen as certain order-preserving surjections on Young diagrams.
We do not pursue this direction any further in this paper.

Let $(P,w)$ be a naturally labeled poset with $n$ elements, and let $\alpha$ be a composition of $n$ with $\ell$ parts.
Given a permutation $\sigma\in\ILinExt_{\alpha}^{\ast}(P,w)$ define the map $f_{\sigma}:P\to[\ell]$ by $f_{\sigma}(\poseta)\coloneqq i$ for all $\poseta\in\compp{i}{\alpha}(\sigma)$ and all $i\in[\ell]$.
Here we use the notation from \cref{eq:subposet}.

The following proposition is the key result of this section and relates the set $\ILinExt_{\alpha}^{\ast}(P,w)$ to order-preserving surjections onto chains.

\begin{proposition}\label[proposition]{prop:opsurj}
Let $(P,w)$ be a naturally labeled poset with $n$ elements, and let $\alpha$ be a composition of $n$.
Then the correspondence $\sigma\mapsto f_{\sigma}$ defines a bijection $\varphi:\ILinExt_{\alpha}^{\ast}(P,w)\to\opsurj_{\alpha}^{\ast}(P)$.
\end{proposition}

\begin{proof}
Let $\ell$ denote the number of parts of $\alpha$.
It is not difficult to see that the map $\varphi:\ILinExt_{\alpha}^{\ast}(P,w)\to\opsurj_{\alpha}^{\ast}(P)$ is well-defined.

We next show that $\varphi$ is injective.
Suppose $\sigma,\tau\in\ILinExt_{\alpha}^{\ast}(P,w)$ with $f_{\sigma}=f_{\tau}$.
Then $\compp{i}{\alpha}(\sigma)=\compp{i}{\alpha}(\tau)$ for all $i\in[\ell]$.
Consequently $\sigma^{-1}\circ w(\compp{i}{\alpha}(\sigma))=\compb{i}{\alpha}=\tau^{-1}\circ w(\compp{i}{\alpha}(\sigma))$, and therefore $\sigma(\compb{i}{\alpha})=\tau(\compb{i}{\alpha})$ for all $i\in[\ell]$.
\cref{lemma:no_descents} implies $\sigma=\tau$.

To see that $\varphi$ is surjective let $f\in\opsurj_{\alpha}^{\ast}(P)$.
Define $\sigma\in\symS_n$ as the unique permutation such that $\sigma(\compb{i}{\alpha})=w(f^{-1}(i))$ for all $i\in[\ell]$ and $\DES(\sigma)\subseteq\compset{\alpha}$.
That is, the word $\sigma_1\dotsm\sigma_n$ is obtained by first listing the numbers in $w(f^{-1}(1))$ in increasing order, then the numbers in $w(f^{-1}(2))$ in increasing order, and so on.

Clearly $\sigma$ is $\alpha$-unimodal, and each subposet $\compp{i}{\alpha}(\sigma)=f^{-1}(i)$ contains a unique minimal element. 
It remains to show that $\sigma^{-1}\circ w:P\to[n]$ is a linear extension.

Let $\poseta,\posetb\in P$ with $\poseta<_P\posetb$.
If there exists $i\in[\ell]$ with $\poseta,\posetb\in f^{-1}(i)$ then $w(\poseta),w(\posetb)\in\sigma(\compb{i}{\alpha})$ and $\sigma^{-1}\circ w(\poseta)<\sigma^{-1}\circ w(\posetb)$ since $w(\poseta)<w(\posetb)$ and $\sigma$ restricted to the set $\compb{i}{\alpha}$ has no descents.

Otherwise $\poseta\in f^{-1}(i)$ and $\posetb\in f^{-1}(j)$ for some $i,j\in[\ell]$ with $i<j$ because $f$ is order-preserving.
We conclude $\sigma^{-1}\circ w(\poseta)\in\compb{i}{\alpha}$ and $\sigma^{-1}\circ w(\posetb)\in\compb{j}{\alpha}$ where $i<j$, and therefore $\sigma^{-1}\circ w(\poseta)<\sigma^{-1}\circ w(\posetb)$.

Hence $\sigma\in\ILinExt_{\alpha}^{\ast}(P,w)$ and $f_\sigma=f$.
\end{proof}

\begin{example}*\label[example]{ex:path8}
Let $\alpha=2213$ and consider the following labeled poset.
\begin{align*}
\begin{tikzpicture}[xscale=0.5,yscale=0.5,baseline=12pt,
circ/.style={circle,draw,inner sep=0.8pt, minimum width=12pt,font=\footnotesize}
]
\draw(0,3)node{$P$};
\node[circ] (1) at (0,0) {$v_1$};
\node[circ] (2) at (1,1) {$v_2$};
\node[circ] (3) at (2,2) {$v_3$};
\node[circ] (4) at (3,1) {$v_4$};
\node[circ] (5) at (4,2) {$v_5$};
\node[circ] (6) at (5,1) {$v_6$};
\node[circ] (7) at (6,2) {$v_7$};
\node[circ] (8) at (7,3) {$v_8$};
\draw[black,thick] (1)--(2)--(3)--(4)--(5)--(6)--(7)--(8);
\end{tikzpicture}
&&
\begin{tikzpicture}[xscale=0.5,yscale=0.5,baseline=12pt,
circ/.style={circle,draw,inner sep=0.8pt, minimum width=12pt,font=\footnotesize}
]
\draw(0,3)node{$(P,w)$};
\node[circ] (1) at (0,0) {$1$}; 
\node[circ] (2) at (1,1) {$2$}; 
\node[circ] (3) at (2,2) {$5$}; 
\node[circ] (4) at (3,1) {$3$}; 
\node[circ] (5) at (4,2) {$6$}; 
\node[circ] (6) at (5,1) {$4$}; 
\node[circ] (7) at (6,2) {$7$}; 
\node[circ] (8) at (7,3) {$8$}; 
\draw[black,thick] (1)--(2)--(3)--(4)--(5)--(6)--(7)--(8);
\end{tikzpicture}
\end{align*}
Let $\sigma=12|47|8|356\in\ILinExt_{\alpha}^{\ast}(P,w)$.
Then $\compp{1}{\alpha}(\sigma)=\{v_1,v_2\}$, $\compp{4}{\alpha}(\sigma)=\{v_3,v_4,v_5\}$, $\compp{2}{\alpha}(\sigma)=\{v_6,v_7\}$ and $\compp{3}{\alpha}(\sigma)=\{v_8\}$.
Hence $f_{\sigma}:P\to\{1,2,3,4\}$ is as shown below.
\begin{align*}
\begin{tikzpicture}[xscale=0.5,yscale=0.5,
circ/.style={circle,draw,inner sep=0.8pt, minimum width=12pt,font=\footnotesize}
]
\draw(0,3)node{$(P,\sigma^{-1}\circ w)$};
\node[circ] (1) at (0,0) {$1$}; 
\node[circ] (2) at (1,1) {$2$}; 
\node[circ] (3) at (2,2) {$7$}; 
\node[circ] (4) at (3,1) {$6$}; 
\node[circ] (5) at (4,2) {$8$}; 
\node[circ] (6) at (5,1) {$3$}; 
\node[circ] (7) at (6,2) {$4$}; 
\node[circ] (8) at (7,3) {$5$}; 
\draw[black,very thick] (1)--(2);
\draw[dotted,thick] (2)--(3);
\draw[black,very thick] (3)--(4);
\draw[black,very thick] (4)--(5);
\draw[dotted,thick] (5)--(6);
\draw[black,very thick] (6)--(7);
\draw[dotted,thick] (7)--(8);
\end{tikzpicture}
&&
\begin{tikzpicture}[xscale=0.5,yscale=0.5,
circ/.style={circle,draw,inner sep=0.8pt, minimum width=12pt,font=\footnotesize}
]
\draw(0,3)node{$(P,f_{\sigma})$};
\node[circ] (1) at (0,0) {$1$}; 
\node[circ] (2) at (1,1) {$1$}; 
\node[circ] (3) at (2,2) {$4$}; 
\node[circ] (4) at (3,1) {$4$}; 
\node[circ] (5) at (4,2) {$4$}; 
\node[circ] (6) at (5,1) {$2$}; 
\node[circ] (7) at (6,2) {$2$}; 
\node[circ] (8) at (7,3) {$3$}; 
\draw[black,very thick] (1)--(2);
\draw[dotted,thick] (2)--(3);
\draw[black,very thick] (3)--(4);
\draw[black,very thick] (4)--(5);
\draw[dotted,thick] (5)--(6);
\draw[black,very thick] (6)--(7);
\draw[dotted,thick] (7)--(8);
\end{tikzpicture}
\qeddefhere
\end{align*}
\end{example}

An immediate consequence of \cref{prop:opsurj} is the following.

\begin{corollary}\label[corollary]{cor:card_ILinExt}
Let $(P,w)$ be a naturally labeled poset with $n$ elements, and let $\alpha$ be a composition of $n$.
Then the cardinality $\abs{\ILinExt_{\alpha}^{\ast}(P,w)}$ is independent of $w$.
That is, if $w':P\to[n]$ is another natural labeling, then $\abs{\ILinExt_{\alpha}^{\ast}(P,w)}=\abs{\ILinExt_{\alpha}^{\ast}(P,w')}$.
\end{corollary}

Note that \cref{cor:card_ILinExt} also follows from \cref{thm:KP_positive}.
\cref{prop:opsurj} even provides a bijection between $\ILinExt_{\alpha}^{\ast}(P,w)$ and $\ILinExt_{\alpha}^{\ast}(P,w')$.

Another advantage of order-preserving surjections over unimodal linear extensions is that the sets $\opsurj_{\alpha}(P,w)$ for $\alpha\vDash n$ are disjoint whereas the sets $\ILinExt_{\alpha}^{\ast}(P,w)$ intersect.
It is therefore much more convenient to work, say, with the set $\opsurj^{\ast}(P)$ than with its pendant in the world of $\alpha$-unimodal linear extensions.

The main result of the previous section, \cref{thm:KP_positive}, has an equivalent formulation in terms of order-preserving surjections.

\begin{theorem}\label[theorem]{thm:KP_opsurj}
Let $P$ be poset with $n$ elements.
Then
\begin{equation*}
K_P(\xvec)
=
\sum_{f\in\opsurj^{\ast}(P)}
\frac{\Psi_{\alpha(f)}(\xvec)}{z_{\alpha(f)}}
=
\sum_{\alpha\vDash n}
\frac{\Psi_{\alpha}(\xvec)}{z_{\alpha}}\abs{\opsurj_{\alpha}^{\ast}(P)}
\,.
\end{equation*}
\end{theorem}

\cref{thm:KP_opsurj} yields the intriguing identity
\begin{equation}
\label{eq:monomialO}
K_P(\xvec)
=
\sum_{f\in\opsurj(P)}
\qmonom_{\alpha(f)}
=
\sum_{f\in\opsurj^{\ast}(P)}
\frac{\Psi_{\alpha(f)}(\xvec)}{z_{\alpha(f)}}.
\end{equation}
Multiplying \eqref{eq:monomialO} by $n!$ and taking coefficients of $\qmonom_{\beta}$ using~\eqref{eq:PsiintoMonomial} and~\cref{prop:treeHook}, leads to the following open problem.

\begin{problem}\label[problem]{pr:findbij}
Let $P$ be a poset with $n$ elements and $\beta$ a composition of $n$.
Find a bijection
\begin{equation*}
\varphi:
\symS_n\times\opsurj_{\beta}(P)
\to
\bigcup_{\alpha\leq\beta}\CONS{\alpha}{\beta}\times\opsurj_{\alpha}^{\ast}(P)
\,.
\end{equation*}
\end{problem}
We are currently unable to provide such bijections, although we suspect it might not be hopelessly difficult to find them.
Note that the first part of the proof of \cref{thm:CONS} solves \cref{pr:findbij} in the case where $P$ is a chain.
A full solution to \cref{pr:findbij} should give an independent proof of Theorems~\ref{thm:KP_positive} and~\ref{thm:KP_opsurj}.\footnote{
It woud be particularly appealing to the combinatorialist if such a proof were cancellation free and avoided sign-reversing involutions.}
We also suspect that guessing \eqref{eq:monomialO} without proving \cref{thm:KP_positive} first would have been very difficult.

\medskip

We now compute the numbers $\abs{\opsurj_{\alpha}^{\ast}(P)}$ for three examples.
In later sections we use these examples and apply \cref{thm:KP_opsurj} to three families of quasisymmetric functions.
The first example is related to the quasisymmetric functions of uniform matroids, see \cref{sec:matroid}.

\begin{example}*[The complete bipartite graph]
\label[example]{ex:complete_bipartite_graph}
Let $P$ be the poset with ground set $\{\poseti{1},\dotsc,\poseti{r},\posetj{1},\dotsc,\posetj{m}\}$ and cover relations $\poseti{i}<\posetj{j}$ for all $i\in[r]$ and $j\in[m]$.
Thus the Hasse diagram of $P$ is the complete bipartite graph $K_{r,m}$.

Suppose $f\in\opsurj^{\ast}(P)$.
Then $f$ restricts to a bijection $f:\{\poseti{i}:i\in[r]\}\to[r]$.
Moreover, there exists a subset $S\subseteq[m]$ such that $f(\posetj{j})=r$ for all $j\in S$ and $f$ restricts to a bijection
\begin{equation*}
f:\{\posetj{j}:j\in[m]\setminus S\}\to\{r+s:s\in[m-k]\}
,
\end{equation*}
where $k\coloneqq\abs{S}$.
We conclude that
\[
\abs{\opsurj_{\alpha}^{\ast}(P)}
=
\begin{cases}
{r!m!}/{k!}&\quad\text{if }\alpha=(1^{r-1},k+1,1^{m-k}),\\
0&\quad\text{otherwise.}
\end{cases}
\qeddefhere
\] 
\end{example}

The next example is connected to Eulerian quasisymmetric functions and chromatic quasisymmetric functions of paths.
This is explained in \cref{sec:eulerian}.

\begin{example}[The path]
\label[example]{ex:path}
Let $\alpha$ be a composition of $n$ with $\ell$ parts.
We want to compute
\[
\sum_{S\subseteq[n-1]}q^{\abs{S}}\abs{\opsurj_{\alpha}^{\ast}(P_{S})}
\,,
\]
where $P_S$ denotes the poset on $\{\poseti{1},\dots,\poseti{n}\}$ with cover relations $\poseti{i}<\poseti{i+1}$ for $i\in[n-1]\setminus S$ and $\poseti{i}>\poseti{i+1}$ for $i\in S$.
Thus the Hasse diagram of $P_S$ viewed as an (undirected) graph is a path.
Arranging the vertices $\poseti{1},\dotsc,\poseti{n}$ from left to right, we can think of the Hasse diagram of $P_S$ as a word $W_S\in\{\up,\down\}^{n-1}$ with \emph{up-} and \emph{down-steps}, where the down-steps correspond to elements of $S$.

In this way the poset $P$ in \cref{ex:path8} corresponds to the set $S=\{3,5\}$ and the word $\up\up\down\up\down\up\up$.

Let $S\subseteq[n-1]$ and fix some $f\in\opsurj_{\alpha}^{\ast}(P_S)$.

For each $i\in[\ell]$ the subgraph of the Hasse diagram of $P_S$ induced by $f^{-1}(i)$ can be identified with a word $W(i)$ that consists of $r_i$ down-steps followed by $\alpha_i-r_i-1$ up-steps.
This follows from the fact that $f^{-1}(i)$ has a unique minimal element.
Moreover, reading the values of $f$ from left to right without repetitions we obtain a permutation $\pi$.

For instance, the order-preserving surjection $f_{\sigma}$ from \cref{ex:path8} yields $r_1=0$, $r_2=0$, $r_3=0$, $r_4=1$ and $\pi=1423$.

We claim that the map $\psi$ defined by $f\mapsto(\pi,r_1,\dotsc,r_\ell)$ is a bijection
\[
\psi:\bigsqcup_{S\subseteq[n-1]}\opsurj_{\alpha}^{\ast}(P_{S})
\to\symS_{\ell}\times\prod_{i=1}^{\ell}\{0,\dotsc,\alpha_i-1\}
\,.
\]
To see this we construct the inverse of $\psi$.
Given $(\pi,r_1,\dotsc,r_{\ell})$ first form words $W(i)$ consisting of $r_i$ down-steps followed by $\alpha_i-r_i-1$ up-steps.
We recover $S$ from the identity
\[
W_S=
W(\pi_1)\step{1}W(\pi_2)\dotsm W(\pi_{\ell-1})\step{\ell-1}W(\pi_{\ell})
\,,
\]
where $\step{i}=\down$ if $i\in\DES(\pi)$ and $\step{i}=\up$ otherwise.
Once $\alpha,S$ and $\pi$ are known, it is easy to recover $f$.

From
$\abs{S}=\des(\pi)+\sum_{i=1}^{\ell}r_i$
it follows that
\begin{align*}
\sum_{S\subseteq[n-1]}q^{\abs{S}}\abs{\opsurj_{\alpha}^{\ast}(P_{S})}
=
A_{\ell}(q)\prod_{i=1}^{\ell}[\alpha_i]_q
\,,
\end{align*}
where
$A_{k}(q)\coloneqq\sum_{\sigma\in\symS_{k}}q^{\des(\sigma)}$
is the \emph{Eulerian polynomial} and
$[a]_q\coloneqq\frac{1-q^a}{1-q}$
is the commonly used \emph{$q$-integer}.
\end{example}

Our final example is related to chromatic quasisymmetric functions of cycles and cycle Eulerian quasisymmetric functions, see \cref{sec:eulerian}.
The same computation was previously done by B.~Ellzey with different notation.

\begin{example}*[The cycle, {\cite[Thm.~4.4]{Ellzey2016}}]
\label[example]{ex:cycle}
Let $\alpha$ be a composition of $n$ with $\ell$ parts.
We want to compute
\[
\sum_{\substack{S\subseteq[n]\\ 0<\abs{S}<n}}q^{\abs{S}}\abs{\opsurj_{\alpha}^{\ast}(P_{S})}
\,,
\]
where $P_S$ denotes the poset on $\{\poseti{1},\dots,\poseti{n}\}$ with cover relations $\poseti{i}<\poseti{i+1}$ for $i\in[n]\setminus S$ and $\poseti{i}>\poseti{i+1}$ for $i\in S$.
Here indices are to be understood modulo $n$, thus the Hasse diagram of $P_S$ viewed as an (undirected) graph is a cycle.
Arranging the vertices $\poseti{1},\dotsc,\poseti{n}$ from left to right, we can think of the Hasse diagram of $P_S$ as a word $W_S\in\{\up,\down\}^{n}$, where the last letter determines the relation between $\poseti{1}$ and $\poseti{n}$.
Note that we have to exclude the cases $S=\emptyset$ and $S=[n]$ since they do not give rise to partial orders.

For example, the following poset corresponds to the set $S=\{2,3,4,6,7\}$ and the word $\up\down\down\down\up\down\down$.
\begin{align*}
\begin{tikzpicture}[xscale=0.5,yscale=0.5,baseline=12pt,
circ/.style={circle,draw,inner sep=0.8pt, minimum width=12pt,font=\footnotesize}
]
\draw(0,0)node{$P$};
\node[circ] (1) at (0,2) {$v_1$};
\node[circ] (2) at (1,3) {$v_2$};
\node[circ] (3) at (2,2) {$v_3$};
\node[circ] (4) at (3,1) {$v_4$};
\node[circ] (5) at (4,0) {$v_5$};
\node[circ] (6) at (5,1) {$v_6$};
\node[circ] (7) at (6,0) {$v_7$};
\draw[black,thick] (1)--(2)--(3)--(4)--(5)--(6)--(7);
\draw[black,thick,->](-.8,2.8)--(1);
\draw[black,thick,->](7)--(6.8,-.8);
\end{tikzpicture}
&&
\begin{tikzpicture}[xscale=0.5,yscale=0.5,baseline=12pt,
circ/.style={circle,draw,inner sep=0.8pt, minimum width=12pt,font=\footnotesize}
]
\draw(0,0)node{$(P,w)$};
\node[circ] (1) at (0,2) {$1$};
\node[circ] (2) at (1,3) {$7$};
\node[circ] (3) at (2,2) {$5$};
\node[circ] (4) at (3,1) {$3$};
\node[circ] (5) at (4,0) {$2$};
\node[circ] (6) at (5,1) {$6$};
\node[circ] (7) at (6,0) {$4$};
\draw[black,thick] (1)--(2)--(3)--(4)--(5)--(6)--(7);
\draw[black,thick,->](-.8,2.8)--(1);
\draw[black,thick,->](7)--(6.8,-.8);
\end{tikzpicture}
\end{align*}

Let $S\subseteq[n]$ with $0<\abs{S}<n$.
We have to distinguish two cases.

\noindent
\textbf{Case 1.} First assume that $\ell\geq2$.

Fix some $f\in\opsurj_{\alpha}^{\ast}(P_{S})$.
For each $i\in[\ell]$ the subgraph of the Hasse diagram of $P_S$ induced by $f^{-1}(i)$ is a path and can be identified with a word $W(i)$ that consists of $r_i$ down-steps followed by $\alpha_i-r_i-1$ up-steps.
Again this follows from the fact that $f^{-1}(i)$ has a unique minimal element.
Reading the values of $f$ from left to right without repetitions we obtain a long cycle $(\pi_1,\dotsc,\pi_{\ell-1},\ell)\in\symS_{\ell}$ which we identify with the permutation $\pi_1\dotsm\pi_{\ell-1}\in\symS_{\ell-1}$.
Moreover, let $k$ be the position of the minimal element of $f^{-1}(1)$.

For instance, consider the composition $\alpha=331$ and the labeled poset $(P,w)$ above, and let $\sigma=235|147|6\in\ILinExt_{\alpha}^{\ast}(P,w)$.
Then $\compp{2}{\alpha}(\sigma)=\{v_1,v_2,v_7\}$, $\compp{1}{\alpha}(\sigma)=\{v_3,v_4,v_5\}$ and $\compp{3}{\alpha}(\sigma)=\{v_6\}$.
Hence $f_{\sigma}$ is as shown below.
\begin{align*}
\begin{tikzpicture}[xscale=0.5,yscale=0.5,
circ/.style={circle,draw,inner sep=0.8pt, minimum width=12pt,font=\footnotesize}
]
\draw(0,0)node{$(P,\sigma^{-1}\circ w)$};
\node[circ] (1) at (0,2) {$4$};
\node[circ] (2) at (1,3) {$6$};
\node[circ] (3) at (2,2) {$3$};
\node[circ] (4) at (3,1) {$2$};
\node[circ] (5) at (4,0) {$1$};
\node[circ] (6) at (5,1) {$7$};
\node[circ] (7) at (6,0) {$5$};
\draw[black,very thick,->](-.8,2.8)--(1);
\draw[black,very thick,->](7)--(6.8,-.8);
\draw[black,very thick] (1)--(2);
\draw[dotted,thick] (2)--(3);
\draw[black,very thick] (3)--(4)--(5);
\draw[dotted,thick] (5)--(6)--(7);
\end{tikzpicture}
&&
\begin{tikzpicture}[xscale=0.5,yscale=0.5,
circ/.style={circle,draw,inner sep=0.8pt, minimum width=12pt,font=\footnotesize}
]
\draw(0,0)node{$(P,f_{\sigma})$};
\node[circ] (1) at (0,2) {$2$};
\node[circ] (2) at (1,3) {$2$};
\node[circ] (3) at (2,2) {$1$};
\node[circ] (4) at (3,1) {$1$};
\node[circ] (5) at (4,0) {$1$};
\node[circ] (6) at (5,1) {$3$};
\node[circ] (7) at (6,0) {$2$};
\draw[black,very thick,->](-.8,2.8)--(1);
\draw[black,very thick,->](7)--(6.8,-.8);
\draw[black,very thick] (1)--(2);
\draw[dotted,thick] (2)--(3);
\draw[black,very thick] (3)--(4)--(5);
\draw[dotted,thick] (5)--(6)--(7);
\end{tikzpicture}
\end{align*}
This yields $r_1=2$, $r_2=1$, $r_3=0$, $\pi=21$, and $k=5$.

We claim that the map $\psi$ defined by $f\mapsto(k,\pi,r_1,\dotsc,r_\ell)$ is a bijection
\[
\psi:\bigsqcup_{\substack{S\subseteq[n]\\ 0<\abs{S}<n}}\opsurj_{\alpha}^{\ast}(P_{S})
\to[n]\times\symS_{\ell-1}\times\prod_{i=1}^{\ell}\{0,\dotsc,\alpha_i-1\}
\,.
\]
To see this we construct the inverse of $\psi$ in a similar fashion as in \cref{ex:path}.
Given $(k,\pi,r_1,\dotsc,r_{\ell})$ first form words $W(i)$ consisting of $r_i$ down-steps followed by $\alpha_i-r_i-1$ up-steps.
Then $W_S$ is a cyclic shift of the word
\[
W(\pi_1)\step{1}W(\pi_2)\dotsm\step{\ell-2}W(\pi_{\ell-1})\up W(\ell)\down
\,,
\]
where $\step{i}=\down$ if $i\in\DES(\pi)$ and $\step{i}=\up$ otherwise.
The number $k$ contains precisely the information needed to recover $S$.
As before, $f$ can be obtained once $\alpha,S$ and $\pi$ are known.

For example, suppose we are given the data $(k,\pi,r_1,r_2,r_3)=(5,21,2,1,0)$.
We first form the posets corresponding to the words $W(i)$.
\begin{align*}
\begin{tikzpicture}[xscale=0.5,yscale=0.5,baseline=-5mm,
circ/.style={circle,draw,inner sep=0.8pt, minimum width=12pt,font=\footnotesize}
]
\draw(-1,0)node{$W(1)=\down\down$};
\node[circ] (A) at (0,2) {$1$};
\node[circ] (B) at (1,1) {$1$};
\node[circ] (C) at (2,0) {$1$};
\draw[black,thick] (A)--(B)--(C);
\end{tikzpicture}
&&
\begin{tikzpicture}[xscale=0.5,yscale=0.5,baseline=-5mm,
circ/.style={circle,draw,inner sep=0.8pt, minimum width=12pt,font=\footnotesize}
]
\draw(1,2)node{$W(2)=\down\up$};
\node[circ] (A) at (0,1) {$2$};
\node[circ] (B) at (1,0) {$2$};
\node[circ] (C) at (2,1) {$2$};
\draw[black,thick] (A)--(B)--(C);
\end{tikzpicture}
&&
\begin{tikzpicture}[xscale=0.5,yscale=0.5,baseline=-5mm,
circ/.style={circle,draw,inner sep=0.8pt, minimum width=12pt,font=\footnotesize}
]
\draw(0,2)node{$W(3)=\emptyset$};
\node[circ] (A) at (0,0) {$3$};
\end{tikzpicture}
\end{align*}
Next join the words $W(i)$ according to $\pi$ (with $W(3)$ in last position).
\begin{align*}
\begin{tikzpicture}[xscale=0.5,yscale=0.5,
circ/.style={circle,draw,inner sep=0.8pt, minimum width=12pt,font=\footnotesize}
]
\draw(0,0)node{$W(2)\down W(1)\up W(3)\down$};
\node[circ] (1) at (0,3) {$2$};
\node[circ] (2) at (1,2) {$2$};
\node[circ] (3) at (2,3) {$2$};
\node[circ] (4) at (3,2) {$1$};
\node[circ] (5) at (4,1) {$1$};
\node[circ] (6) at (5,0) {$1$};
\node[circ] (7) at (6,1) {$3$};
\draw[dotted,thick,->](-.8,3.8)--(1);
\draw[dotted,thick,->](7)--(6.8,.2);
\draw[black,very thick] (1)--(2)--(3);
\draw[dotted,thick] (3)--(4);
\draw[black,very thick] (4)--(5)--(6);
\draw[dotted,thick] (6)--(7);
\end{tikzpicture}
&&
\begin{tikzpicture}[xscale=0.5,yscale=0.5,
circ/.style={circle,draw,inner sep=0.8pt, minimum width=12pt,font=\footnotesize}
]
\draw(0,0)node{$W_S=\up\down\down\down\up\down\down$};
\node[circ] (1) at (0,2) {$2$};
\node[circ] (2) at (1,3) {$2$};
\node[circ] (3) at (2,2) {$1$};
\node[circ] (4) at (3,1) {$1$};
\node[circ] (5) at (4,0) {$1$};
\node[circ] (6) at (5,1) {$3$};
\node[circ] (7) at (6,0) {$2$};
\draw[black,very thick,->](-.8,2.8)--(1);
\draw[black,very thick,->](7)--(6.8,-.8);
\draw[black,very thick] (1)--(2);
\draw[dotted,thick] (2)--(3);
\draw[black,very thick] (3)--(4)--(5);
\draw[dotted,thick] (5)--(6)--(7);
\end{tikzpicture}
\end{align*}
Since the minimal element of $f^{-1}(1)$ is currently in sixth position, we have to cyclically shift until the position is $k=5$ to obtain $W_S$.

From
$\abs{S}=1+\des(\pi)+\sum_{i=1}^{\ell}r_i$
it follows that
\begin{align*}
\sum_{S\subseteq[n-1]}q^{\abs{S}}\abs{\opsurj_{\alpha}^{\ast}(P_{S})}
=
nqA_{\ell-1}(q)\prod_{i=1}^{\ell}[\alpha_i]_q
\,,
\end{align*}
where $A_{k}(q)$ and $[a]_q$ are defined as in \cref{ex:path}.

\noindent
\textbf{Case 2.} Now suppose $\alpha=(n)$.

Note that $\opsurj_{(n)}^{\ast}(P_S)$ contains a unique element if $P_S$ has a unique minimal element, and $\opsurj_{(n)}^{\ast}(P_S)$ is empty otherwise.
It is not difficult to see that
\[
\sum_{\substack{S\subseteq[n]\\ 0<\abs{S}<n}}q^{\abs{S}}|\opsurj_{(n)}^{\ast}(P_S)|
=nq[n-1]_q
\,.\qeddefhere
\]
\end{example}

\section{Reverse \texorpdfstring{$P$}{P}-partitions with forced equalities}
\label[section]{sec:revPPeq}

In this section we consider the generating function of reverse $P$-partitions that assign equal colors to certain given elements.
In \cref{thm:KPE} below we show that these functions are $\Psi$-positive, thus extending \cref{thm:KP_positive}.
We hope this will enable (or simplify) future applications of our results to families of quasisymmetric functions that are not covered by \cref{sec:revPP}.
A simple example of such a family are the power sum symmetric functions $\psumP_{\lambda}$, which we discuss in \cref{sec:p}.
Reverse $P$-partitions with forced equalities have not received much attention to date.
The results of this section might be an indication that they are worth investigating further.

A \emph{partitioned poset} $(P,w,E)$ is a naturally labeled poset $(P,w)$ endowed with an equivalence relation $E$ on $P$.
Given a partitioned poset $(P,w,E)$ with $n$ elements, let
\begin{equation*}
\rpp(P,E)
\coloneqq
\big\{
f\in\rpp(P):f(\poseta)=f(\posetb)\text{ for all }\poseta,\posetb\in P\text{ with }\poseta\overset{E}{\sim}\posetb
\big\}
\end{equation*}
and define the generating function
\[
K_{P,E}(\xvec)
\coloneqq
\sum_{\substack{f\in\rpp(P)\\
\poseta\overset{E}{\sim}\posetb\Rightarrow f(\poseta)=f(\posetb)}}
\prod_{\poseta\in P}\xvec_{f(\poseta)}
\,.
\]
This is a homogeneous quasisymmetric function of degree $n$.
For example, if all equivalence classes of $E$ are singletons then $K_{P,E}=K_{P}$.
On the other hand if $\poseta\overset{E}{\sim}\posetb$ for all $\poseta,\posetb\in P$, then $K_{P,E}$ equals the power sum symmetric function $\psumP_{n}$.

In the above definition the equivalence relation $E$ is completely arbitrary.
We now define a very special kind of equivalence relation.

Let $P$ be a poset.
An equivalence relation $E$ on $P$ is a \emph{chain congruence} on $P$ if the following two conditions are satisfied:
\begin{enumerate}[(i)]
\item The equivalence class $[\poseta]_E$ is a chain in $P$ for all $\poseta\in P$.
\item For all $\poseta,\posetb\in P$ with $\poseta<_P\posetb$ and $\poseta\overset{E}\nsim\posetb$ we have $\max{[\poseta]_E}<_P\min{[\posetb]_E}$.
\end{enumerate}

In order to prove $\Psi$-positivity of the functions $K_{P,E}$ it suffices to consider chain congruences.
This is the content of the next lemma.

\begin{lemma}\label[lemma]{lem:chain_congruence}
Let $(P,w,E)$ be a partitioned poset.
Then there exists a partitioned poset $(P',w',E')$ such that
\[
K_{P,E}(\xvec)=K_{P',E'}(\xvec)
\,,
\]
the poset $P'$ is obtained from $P$ by adding order relations, $E'$ is obtained from $E$ by joining equivalence classes, and $E'$ is a chain congruence on $P'$.
\end{lemma}
\begin{proof}
Define a relation $E'$ on $P$ as follows.
For $\poseta,\posetb\in P$ let $\poseta\overset{E'}{\sim}\posetb$ if and only if the following two (symmetric) conditions are satisfied:
\begin{enumerate}[(i)]
\item
There exist $k\in\setN$ and $\poseti{1},\posetj{1},\dotsc,\poseti{k},\posetj{k}\in P$ such that $\poseta=\poseti{1}$, $\posetb=\posetj{k}$, $\poseti{i}\leq_P\posetj{i}$ for all $i\in[k]$, and $\posetj{i}\overset{E}{\sim}\poseti{i+1}$ for all $i\in[k-1]$. 
\item
There exist $k\in\setN$ and $\poseti{1},\posetj{1},\dotsc,\poseti{k},\posetj{k}\in P$ such that $\posetb=\poseti{1}$, $\poseta=\posetj{k}$, $\poseti{i}\leq_P\posetj{i}$ for all $i\in[k]$, and $\posetj{i}\overset{E}{\sim}\poseti{i+1}$ for all $i\in[k-1]$. 
\end{enumerate}
See \cref{fig:chaincongruence} for an example.

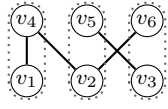
\begin{figure}
\begin{tikzpicture}[xscale=0.8,yscale=0.8,
circ/.style={circle,draw,inner sep=0.8pt, minimum width=12pt,font=\footnotesize},equiv/.style={black!60,thick,dotted,rounded corners=6pt}
]
\begin{scope}
\node[circ] (1) at (0,0) {$v_1$};
\node[circ] (4) at (0,1) {$v_4$};
\node[circ] (2) at (1,0) {$v_2$};
\node[circ] (5) at (1,1) {$v_5$};
\node[circ] (3) at (2,0) {$v_3$};
\node[circ] (6) at (2,1) {$v_6$};
\draw[black, thick](1)--(4)(2)--(4)(2)--(6)(3)--(5);
\draw[equiv](-.3,-.3)rectangle(.3,1.3);
\draw[equiv,xshift=1cm](-.3,-.3)rectangle(.3,1.3);
\draw[equiv,xshift=2cm](-.3,-.3)rectangle(.3,1.3);
\end{scope}
\end{tikzpicture}
\caption{A poset $P$ with an equivalence relation $E$ given by $v_1\overset{E}{\sim}v_4$ and $v_2\overset{E}{\sim}v_5$ and $v_3\overset{E}{\sim}v_6$.
The equivalence relation $E'$ has two classes $\{v_1,v_4\}$ and $\{v_2,v_3,v_5,v_6\}$.
Suppose $w:P\to[6]$ sends $v_i$ to $i$.
Then $v_2\leq_{P'}v_3\leq_{P'}v_5\leq_{P'}v_6\leq_{P'}v_1\leq_{P'}v_4$.
}
\label{fig:chaincongruence}
\end{figure}

Clearly $E'$ is an equivalence relation and $\poseta\overset{E}{\sim}\posetb$ implies $\poseta\overset{E'}{\sim}\posetb$.
We claim that $K_{P,E}=K_{P,E'}$.
To see this note that $\rpp(P,E')\subseteq\rpp(P,E)$.
Thus let $f\in\rpp(P,E)$ and $\poseta,\posetb\in P$ with $\poseta\overset{E'}{\sim}\posetb$.
The definition of $E'$ yields $f(\poseta)\leq f(\posetb)\leq f(\poseta)$ and therefore $f\in\rpp(P,E')$.

Next define a relation $\leq_{P'}$ on $P$ as follows.
For all $\poseta,\posetb\in P$ let $\poseta\leq_{P'}\posetb$ if and only if one of the following two mutually exclusive conditions is satisfied:

\begin{enumerate}[(i)]
\item
We have $\poseta\overset{E'}{\sim}\posetb$ and $w(\poseta)\leq w(\posetb)$.
\item
We have $\poseta\overset{E'}{\nsim}\posetb$ and there exist $k\in\setN$ and $\poseti{1},\posetj{1},\dotsc,\poseti{k},\posetj{k}\in P$ such that $\poseta=\poseti{1}$, $\posetb=\posetj{k}$, $\poseti{i}\leq_P\posetj{i}$ for all $i\in[k]$, and $\posetj{i}\overset{E'}{\sim}\poseti{i+1}$ for all $i\in[k-1]$. 
\end{enumerate}

Then $P'$ is a poset with the same ground set as $P$, and $\poseta\leq_P\posetb$ implies $\poseta\leq_{P'}\posetb$ for all $\poseta,\posetb\in P$.
Let $w'$ be an arbitrary natural labeling of $P'$.

We claim that $K_{P',E'}=K_{P,E'}$.
To see this note that $\rpp(P',E')\subseteq\rpp(P,E')$.
Thus let $f\in\rpp(P,E')$ and $\poseta,\posetb\in P$ with $\poseta\leq_{P'}\posetb$.
The definition of $P'$ implies $f(\poseta)\leq f(\posetb)$ and therefore $f\in\rpp(P',E')$.

It remains to show that $E'$ is a chain congruence on $P'$.
Clearly $[\poseta]_{E'}$ is a chain by definition of $P'$ for all $\poseta\in P'$.
Secondly suppose $\poseta<_{P}\posetb$ and $\poseta\overset{E'}{\nsim}\posetb$ for some $\poseta,\posetb\in P'$.
Then
\begin{equation*}
\max{[\poseta]_{E'}}
\leq_{P}\max{[\poseta]_{E'}}
\overset{E'}{\sim}\poseta
\leq_{P}\posetb
\overset{E'}{\sim}\min{[\posetb]_{E'}}
\leq_{P}\min{[\posetb]_{E'}}
.
\end{equation*}
Hence $\max{[\poseta]_{E'}}\leq_{P'}\min{[\posetb]_{E'}}$and the proof is complete.
\end{proof}

Let $(P,w,E)$ be a partitioned poset.
If we are only interested in $K_{P,E}$, then by \cref{lem:chain_congruence} we may assume that $E$ is a chain congruence on $P$.
If this is the case we can form the \emph{quotient poset} $P/E$, that is, the partial order on the set of equivalence classes $\{[\poseta]_E:\poseta\in P\}$ defined by $[\poseta]_E\leq[\posetb]_E$ if and only if $\poseta\leq_P\posetb$.
Then
\[
K_{P,E}(\xvec)
=\sum_{f\in\rpp(P/E)}\prod_{\poseta\in P}\xvar_{f([\poseta]_E)}
=
\sum_{f\in\rpp(P/E)}\prod_{C\in P/E}\xvar_{f(C)}^{|C|}
\,.
\]
This leads to a second, equivalent, way of thinking about partitioned posets.

A \emph{weighted poset} $(P,w,d)$ consists of a naturally labeled poset $(P,w)$ and a vector $d=(d_{\poseta})_{\poseta\in P}$ of positive integers --- a weight on $P$.
Define the weighted generating function
\begin{equation*}
K^d_{P}(\xvec) \coloneqq \sum_{f\in\rpp(P)} \prod_{\poseta\in P} \xvar^{d_{\poseta}}_{f(\poseta)}
\,.
\end{equation*}
This is a homogeneous quasisymmetric function of degree $|d|$.
For example, if $P$ is a poset with one element then $K^d_{P}(\xvec)$ is 
equal to the power sum symmetric function $\psumP_{\abs{d}}(\xvec)$.

The following theorem is the main result of this section and generalizes \cref{thm:KP_positive}.

\begin{theorem}\label[theorem]{thm:KPE}
Let $(P,w)$ be a naturally labeled poset with $n$ elements, and let $E$ be a chain congruence on $P$.
Then
\begin{equation}\label{eq:KPE}
K_{P,E}(\xvec)
=\sum_{\alpha\vDash n}
\frac{\Psi_{\alpha}(\xvec)}{z_{\alpha}}
\sum_{\sigma\in\ILinExt_{\alpha}^{\ast}(P,w,E)}
\prod_{i=1}^{\ell(\alpha)}\big|[\min\compp{i}{\alpha}(\sigma)]_E\big|
\,,
\end{equation} 
where $\ILinExt_{\alpha}^{\ast}(P,w,E)$ denotes the set of permutations $\sigma\in\ILinExt_{\alpha}^{\ast}(P,w)$ such that for each $\poseta\in P$ there exists $i\in[\ell(\alpha)]$ with $[\poseta]_E\subseteq\compp{i}{\alpha}(\sigma)$.
Note that $\min\compp{i}{\alpha}(\sigma)$ is well-defined for all $\sigma\in\ILinExt_{\alpha}^{\ast}(P,w)$.

Equivalently,
\begin{equation}\label{eq:KPE_opsurj}
K_{P,E}(\xvec)
=\sum_{\alpha\vDash n}
\frac{\Psi_{\alpha}(\xvec)}{z_{\alpha}}
\sum_{f\in\opsurj_{\alpha}^{\ast}(P,E)}
\prod_{i=1}^{\ell(\alpha)}\big|[\min f^{-1}(i)]_E\big|
\,,
\end{equation} 
where $\opsurj_{\alpha}^{\ast}(P,E)$ denotes the set of order-preserving surjections $f\in\opsurj_{\alpha}^{\ast}(P)$ that satisfy $f(\poseta)=f(\posetb)$ for all $\poseta,\posetb\in P$ with $\poseta\overset{E}{\sim}\posetb$.
Note that $\min f^{-1}(i)$ is well-defined for all $f\in\opsurj_{\alpha}^{\ast}(P)$.

In particular, $K_{P,E}$ is $\Psi$-positive.
\end{theorem}

\cref{thm:KPE} has an equivalent formulation in terms of weighted posets.

\begin{theorem}\label[theorem]{thm:KPd}
Let $(P,w,d)$ be a weighted poset with $n$ elements.
Then
\begin{equation}\label{eq:KPd}
K_P^d(\xvec)=
\sum_{\alpha\vDash n}
\sum_{\sigma\in\ILinExt_{\alpha}^{\ast}(P,w)}
\frac{\Psi_{\beta}(\xvec)}{z_{\beta}}
\prod_{i=1}^{\ell(\alpha)}d_{\min\compp{i}{\alpha}(\sigma)}
\,,
\end{equation}
where $\beta=\beta(d,\alpha,\sigma)$ is the composition of $\abs{d}$ defined by
\[
\beta_i
=\sum_{\poseta\in\compp{i}{\alpha}(\sigma)}d_{\poseta}
\]
for all $i\in[\ell(\alpha)]$.

Equivalently,
\begin{equation}\label{eq:KPd_opsurj}
K_P^d(\xvec)=
\sum_{\alpha\vDash n}
\sum_{f\in\opsurj_{\alpha}^{\ast}(P)}
\frac{\Psi_{\beta}(\xvec)}{z_{\beta}}
\prod_{j=1}^{\ell(\alpha)}d_{\min f^{-1}(j)}
\,,
\end{equation}
where $\beta=\beta(d,\alpha,f)$ is the composition of $\abs{d}$ defined by
\begin{equation}\label{eq:KPd_opsurj_beta}
\beta_i
=\sum_{\poseta\in f^{-1}(i)}d_{\poseta}
\end{equation}
for all $i\in[\ell(\alpha)]$.

In particular, $K_{P}^d$ is $\Psi$-positive.
\end{theorem}

\begin{proof}[Proof of equivalence of the statements in Theorems~\ref{thm:KPE} and~\ref{thm:KPd}]
To see that \eqref{eq:KPE} and \eqref{eq:KPE_opsurj} are equivalent, note that the bijection $\sigma\mapsto f_{\sigma}$ from \cref{prop:opsurj} restricts to a bijection from $\ILinExt_{\alpha}^{\ast}(P,w,E)$ to $\opsurj_{\alpha}^{\ast}(P,E)$.
The claimed equivalence follows from $\compp{i}{\alpha}(\sigma)=f_{\sigma}^{-1}(i)$.

The equivalence of \eqref{eq:KPd} and \eqref{eq:KPd_opsurj} follows from \cref{prop:opsurj} as above.

To see that \cref{thm:KPd} implies \cref{thm:KPE} let $E$ be a chain congruence on a finite poset $P$, and let $P/E$ be the associated quotient poset.
Every order-preserving surjection $f\in\opsurj_{\beta}^{\ast}(P,E)$ yields an order-preserving surjection $\bar f\in\opsurj_{\alpha}^{\ast}(P/E)$ where the type $\beta$ of $f$ is related to the type $\alpha$ of $\bar f$ by \eqref{eq:KPd_opsurj_beta}.
Indeed, this correspondence defines a bijection from $\opsurj_{\beta}^{\ast}(P,E)$ to $\opsurj_{\alpha}^{\ast}(P/E)$.
Define a weight $d$ on $P/E$ by assigning to each equivalence class $[\poseta]_E$ its cardinality.
Then the formula \eqref{eq:KPd_opsurj} for $K_{P/E}^d$ implies \eqref{eq:KPE_opsurj}.

Conversely, let $(P,w,d)$ be a weighted poset.
Define a poset $Q$ by replacing each element $\poseta\in P$ by a chain $C_{\poseta}$ with $d_{\poseta}$ elements.
That is, $Q=\bigsqcup_{\poseta\in P}C_{\poseta}$ and for $\posetj{1}\in C_{\poseti{1}}$ and $\posetj{2}\in C_{\poseti{2}}$, where $\poseti{1}\neq\poseti{2}$, we have $\posetj{1}\leq_Q\posetj{2}$ if and only if $\poseti{1}\leq_P\poseti{2}$.
The poset $Q$ comes equipped with a chain congruence $E$, namely $\posetj{1}\overset{E}{\sim}\posetj{2}$ if and only if $\posetj{1},\posetj{2}\in C_{\poseta}$ for some element $\poseta\in P$.
Given an order-preserving surjection $f\in\opsurj_{\alpha}^{\ast}(P)$ define an order-preserving surjection $\hat f\in\opsurj_{\beta}^{\ast}(Q,E)$ by $\hat f(\posetb)=f(\poseta)$ if $\posetb\in C_{\poseta}$.
This defines a bijection from $\opsurj_{\alpha}^{\ast}(P)$ to $\opsurj_{\beta}^{\ast}(Q,E)$ where $\alpha$ and $\beta$ are related by \eqref{eq:KPd_opsurj_beta}.
Hence the formula \eqref{eq:KPE_opsurj} for $K_{Q,E}$ implies \eqref{eq:KPd_opsurj}, and \cref{thm:KPd} follows from \cref{thm:KPE}.
\end{proof}

\begin{proof}[Proof of \cref{thm:KPE}~\eqref{eq:KPE_opsurj}]
Given an equivalence relation $R$ let $\classes(R)$ denote the set of its equivalence classes, and let $c(R)=\abs{\classes(R)}$ denote the number of equivalence classes.

We prove the claim by induction on $n-c(E)$.

Clearly the case $c(E)=n$ reduces to \cref{thm:KP_positive}, so we may assume that there exists $C\in\classes(E)$ with $\abs{C}\geq2$, and that the claim holds for all chain congruences on $P$ with more equivalence classes.
Let $C=\{\poseti{0},\dots,\poseti{k}\}$ for some $k\in[n-1]$, so that $\poseti{0}=\min C$ and $\poseti{k}=\max C$.

Define two new equivalence relations on $P$ by
\begin{equation*}
\begin{split}
\classes(E')
&\coloneqq\classes(E)\setminus\{C\}\cup\{\{\poseti{0}\},C\setminus\{\poseti{0}\}\}
\,,
\\
\classes(E'')
&\coloneqq\classes(E)\setminus\{C\}\cup\{\{\poseti{k}\},C\setminus\{\poseti{k}\}\}
\,.
\end{split}
\end{equation*}
Moreover define a new naturally labeled poset $(P',w)$ by removing all order relations $\poseti{i}<\poseti{k}$ for $i\in\{0,\dots,k-1\}$, see \cref{fig:KPE_recursion}.

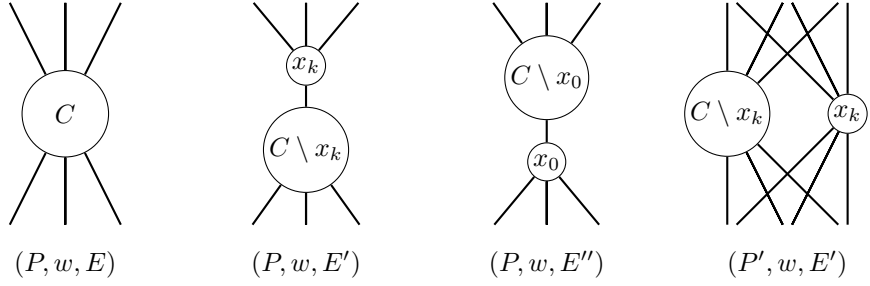
\begin{figure}[t]
\centering
\begin{tikzpicture}[xscale=0.8,yscale=0.8,baseline=40pt,
circ/.style={circle,draw,inner sep=1pt, minimum width=14pt}
]
\begin{scope}
\node (1) at (-1, 0) {$ $};
\node (2) at ( 0, 0) {$ $};
\node (3) at ( 1, 0) {$ $};
\node[circ,minimum width=7.6ex] (4) at ( 0, 2) {$C$};
\node (5) at ( -1, 4) {$ $};
\node (6) at ( 0, 4) {$ $};
\node (7) at ( 1, 4) {$ $};
\node (8) at ( 0, -.5) {$(P,w,E)$};
\draw[black,thick] (1)--(4)--(2)--(4)--(3);
\draw[black,thick] (5)--(4)--(6)--(4)--(7);
\end{scope}
\begin{scope}[xshift=4cm]
\node (1) at (-1, 0) {$ $};
\node (2) at ( 0, 0) {$ $};
\node (3) at ( 1, 0) {$ $};
\node[circ] (4a) at ( 0, 1.4) {$C\setminus\poseti{k}$};
\node[circ] (4b) at ( 0, 2.8) {$\poseti{k}$};
\node (5) at ( -1, 4) {$ $};
\node (6) at ( 0, 4) {$ $};
\node (7) at ( 1, 4) {$ $};
\node (8) at ( 0, -.5) {$(P,w,E')$};
\draw[black,thick] (4a)--(4b);
\draw[black,thick] (1)--(4a)--(2)--(4a)--(3);
\draw[black,thick] (5)--(4b)--(6)--(4b)--(7);
\end{scope}
\begin{scope}[xshift=8cm]
\node (1) at (-1, 0) {$ $};
\node (2) at ( 0, 0) {$ $};
\node (3) at ( 1, 0) {$ $};
\node[circ] (4a) at ( 0, 1.2) {$\poseti{0}$};
\node[circ] (4b) at ( 0, 2.6) {$C\setminus\poseti{0}$};
\node (5) at ( -1, 4) {$ $};
\node (6) at ( 0, 4) {$ $};
\node (7) at ( 1, 4) {$ $};
\node (8) at ( 0, -.5) {$(P,w,E'')$};
\draw[black,thick] (4a)--(4b);
\draw[black,thick] (1)--(4a)--(2)--(4a)--(3);
\draw[black,thick] (5)--(4b)--(6)--(4b)--(7);
\end{scope}
\begin{scope}[xshift=12cm]
\node (1) at (-1, 0) {$ $};
\node (2) at ( 0, 0) {$ $};
\node (3) at ( 1, 0) {$ $};
\node[circ] (4a) at ( -1, 2) {$C\setminus\poseti{k}$};
\node[circ] (4b) at ( 1, 2) {$\poseti{k}$};
\node (5) at ( -1, 4) {$ $};
\node (6) at ( 0, 4) {$ $};
\node (7) at ( 1, 4) {$ $};
\node (8) at ( 0, -.5) {$(P',w,E')$};
\draw[black,thick] (1)--(4a)--(2)--(4a)--(3);
\draw[black,thick] (5)--(4b)--(6)--(4b)--(7);
\draw[black,thick] (1)--(4b)--(2)--(4b)--(3);
\draw[black,thick] (5)--(4a)--(6)--(4a)--(7);
\end{scope}
\end{tikzpicture}
\caption{The four partitioned posets from the proof of \cref{thm:KPE}.}
\label[figure]{fig:KPE_recursion}
\end{figure}

Note that $E'$ and $E''$ are chain congruences on $P$, and that $E'$ is a chain congruence on $P'$.
Thus $K_{P,E'}$, $K_{P,E''}$ and $K_{P',E'}$ satisfy \eqref{eq:KPE_opsurj} by the induction hypothesis.

The function $K_{P,E}$ satisfies the recursion
\begin{equation}\label{eq:KPE_recursion}
K_{P,E}(\xvec)
=K_{P,E'}(\xvec)+K_{P,E''}(\xvec)-K_{P',E'}(\xvec)
.
\end{equation}
This identity follows directly from the definition of $K_{P,E}$ and is analogous to the fact that
\begin{equation*}
\begin{split}
\abs{\{(i,j)\in[n]^2:i=j\}}
&=
\abs{\{(i,j)\in[n]^2:i\leq j\}}
+\abs{\{(i,j)\in[n]^2:i\geq j\}}
\\
&\quad
-\abs{\{(i,j)\in[n]^2\}}
\,.
\end{split}
\end{equation*}
Now fix a composition $\alpha$ of $n$ with $\ell$ parts and a surjective map $f:P\to[\ell]$ of type $\alpha$.
We need to show that the contribution of $f$ predicted by \eqref{eq:KPE_opsurj} is the same on both sides of \eqref{eq:KPE_recursion}.

First assume that $f\in\opsurj_{\alpha}^{\ast}(P,E)$.
Then there exists $i\in[\ell]$ with $f(\poseta)=i$ for all $\poseta\in C$.
It follows that
\begin{equation*}
f\in
\opsurj_{\alpha}^{\ast}(P,E')
\cap\opsurj_{\alpha}^{\ast}(P,E'')
\end{equation*}
and that $f\in\opsurj_{\alpha}(P')$ is an order-preserving surjection of type $\alpha$.

If $\poseti{0}=\min f^{-1}(i)$ then $f\notin\opsurj_{\alpha}^{\ast}(P',E')$ because $f^{-1}(i)$ has two minimal elements in $P'$.
The contribution of $f$ on both sides of \eqref{eq:KPE_recursion} is
\begin{equation*}
\abs{C}\cdot\prod_{j\neq i}\big|[\min f^{-1}(j)]_E\big|
=
((\abs{C}-1)+1-0)\cdot\prod_{j\neq i}\big|[\min f^{-1}(j)]_E\big|
\,.
\end{equation*}
On the other hand, if $\min f^{-1}(i)<_P\poseti{0}$ then $f\in\opsurj_{\alpha}^{\ast}(P',E')$.
The contribution of $f$ on both sides of \eqref{eq:KPE_recursion} is
\begin{equation*}
\prod_{j=1}^{\ell}\big|[\min f^{-1}(j)]_E\big|
=
(1+1-1)\cdot\prod_{j=1}^{\ell}\big|[\min f^{-1}(j)]_E\big|
\,.
\end{equation*}

To finish the proof we need to show that the contributions of surjective maps $f:P\to[\ell]$ of type $\alpha$ that do not lie in $\opsurj_{\alpha}^{\ast}(P,E)$ on the right hand side of \eqref{eq:KPE_recursion} cancel.

Let
\begin{equation*}
\opsurj_{\alpha}^+(P',E')
\coloneqq
\{f\in\opsurj_{\alpha}^{\ast}(P',E'):f(\poseti{0})<f(\poseti{k})
\}
\,,
\end{equation*}
and similarly
\begin{equation*}
\opsurj_{\alpha}^-(P',E')
\coloneqq
\{f\in\opsurj_{\alpha}^{\ast}(P',E'):f(\poseti{0})>f(\poseti{k})
\}
\,.
\end{equation*}

If $f\in\opsurj_{\alpha}^{\ast}(P,E')$ then 
\begin{equation*}
f\notin\opsurj_{\alpha}^{\ast}(P,E)
\quad\Longleftrightarrow\quad
f(\poseti{k-1})<f(\poseti{k})
.
\end{equation*}
In this case $f\in\opsurj_{\alpha}^+(P',E')$, and the terms corresponding to $f$ cancel.

If $f\in\opsurj_{\alpha}^{\ast}(P,E'')$ then 
\begin{equation*}
f\notin\opsurj_{\alpha}^{\ast}(P,E)
\quad\Longleftrightarrow\quad
f(\poseti{0})<f(\poseti{1})
.
\end{equation*}
If this is the case define $f':P\to[\ell]$ by $f'(\poseti{0})=f(\poseti{k})$, $f'(\poseti{k})=f(\poseti{0})$ and $f'(\poseta)=f(\poseta)$ for all $\poseta\in P\setminus\{\poseti{0},\poseti{k}\}$.
It is not difficult to see that $f'\in\opsurj_{\alpha}^{-}(P',E')$, and that the terms coming from $f$ and $f'$ cancel.

Finally, if $f\in\opsurj_{\alpha}^{\ast}(P',E')$ then 
\begin{equation*}
f\notin\opsurj_{\alpha}(P,E)
\quad\Longleftrightarrow\quad
f(\poseti{0})\neq f(\poseti{k})
.
\end{equation*}
Hence we are in one of the cases above, and the proof is complete.
\end{proof}

An immediate consequence of \cref{lem:chain_congruence} and \cref{thm:KPE} is the following result.

\begin{corollary}\label[corollary]{cor:KPE_positive}
Let $(P,w,E)$ be a partitioned poset.
Then $K_{P,E}$ is $\Psi$-positive.
\end{corollary}

Note that \cref{thm:KPE} does not apply unless $E$ is a chain congruence.
Furthermore, \cref{thm:KPd} requires that the weight $d$ is indeed a vector of \emph{positive} integers.
For example, $K_P^d$ given by the poset $\poseta < \posetb < \posetc$ 
with weights $d_\poseta=1$, $d_\posetb=0$ and $d_\posetc=2$
is not $\Psi$-positive.

\section{Applications}
\label[section]{sec:applications}

In this section we apply Theorems~\ref{thm:KP_opsurj} and~\ref{thm:KPd} to derive the expansions of various families of quasisymmetric functions into quasisymmetric power sums.

\subsection{Complete homogeneous symmetric functions}
\label[section]{sec:h}

As a warm up we now derive the expansion of the complete homogeneous symmetric function $\completeH_{\lambda}$ into power sum symmetric functions.

Let $P_\lambda$ be the poset consisting of $\ell$ disjoint chains of lengths $\lambda_i$.
Then $\completeH_\lambda = K_{P_{\lambda}}$ and
\[
\completeH_{\lambda}(\xvec)
=
\sum_{\mu \vdash n} \frac{\psumP_\mu(\xvec)}{z_{\mu}} |\opsurj^\ast_{\mu}(P_\lambda)|.
\]
by \cref{thm:KP_opsurj}.

Another interpretation of these coefficients is given by so called ordered $\mu$-bricks,
see \cite{Egecioglu1991}.
An \emph{ordered $\mu$-brick tabloid} of shape $\lambda$
is a Young diagram of shape $\lambda$ 
filled with labeled \emph{bricks} of sizes given by $\mu$.
The bricks are placed in the diagram such that the 
bricks in each row are sorted with increasing label.
Let $OB_{\mu \lambda}$ be the set of such ordered $\mu$-brick tabloids of shape $\lambda$.

It is easy to see that $\opsurj^\ast_{\mu}(P_\lambda)$ is in bijection with $OB_{\mu\lambda}$.
For instance, let $\lambda=532$ and $\mu=322111$.
Then $\opsurj^\ast_{\mu}(P_\lambda)$
contains the following order-preserving surjection, which is in natural correspondence with the shown ordered brick tabloid.
\begin{align*}
\begin{tikzpicture}[scale=0.45,
circ/.style={circle,draw,inner sep=0.8pt, minimum width=12pt},entry/.style={xshift=5mm,yshift=5mm,font= \small},thickLine/.style={very thick,line join=round}
]
\begin{scope}[yscale=1.5,yshift=-2cm]
\node[circ] (1) at ( 0, 1) {$1$};
\node[circ] (2) at ( 0, 2) {$1$};
\node[circ] (3) at ( 0, 3) {$1$};
\node[circ] (4) at ( 0, 4) {$4$};
\node[circ] (5) at ( 0, 5) {$6$};
\begin{scope}[xshift=1cm]
\node[circ] (6) at ( 1, 1) {$3$};
\node[circ] (7) at ( 1, 2) {$3$};
\node[circ] (8) at ( 1, 3) {$5$};
\end{scope}
\begin{scope}[xshift=2cm]
\node[circ] (9) at ( 2, 1) {$2$};
\node[circ] (10)at ( 2, 2) {$2$};
\end{scope}
\draw[black,thick] (3)--(4)--(5);
\draw[black,thick] (7)--(8);
\draw[black,very thick] (1)--(2)--(3);
\draw[black,very thick] (6)--(7);
\draw[black,very thick] (9)--(10);
\end{scope}
%
\begin{scope}[xshift=10cm]
\draw[thick](0,1)--(2,1)(0,2)--(3,2)(2,1)--(2,2)(3,2)--(3,3)(4,2)--(4,3);
\draw[thickLine](0,0)--(2,0)--(2,1)--(3,1)--(3,2)--(5,2)--(5,3)--(0,3)--cycle;
\draw[entry](1,2)node{$1$}(3,2)node{$4$}(4,2)node{$6$}(.5,1)node{$3$}(2,1)node{$5$}(.5,0)node{$2$}(1,0);
\end{scope}
\end{tikzpicture}
\end{align*}

\subsection{Power sum symmetric functions}
\label[section]{sec:p}

We have already stated in \eqref{eq:powersum_Psi_expansion} the expansion of power sum symmetric functions $\psumP_{\lambda}$ into quasisymmetric power sums $\Psi_{\alpha}$.
We now give an independent proof of this result using the tools we have developed.

Let $\lambda$ be a partition of $n$ with $\ell$ parts, and let $A$ denote the antichain with elements $1,\dotsc,\ell$.
Then $\lambda$ defines a weight on $A$.
By \cref{thm:KPd} we have
\[
\psumP_{\lambda}(\xvec)
=K_{A}^{\lambda}(\xvec)
=\sum_{\sigma\in\symS_{\ell}}
\frac{\Psi_{\sigma(\lambda)}(\xvec)}{z_{\sigma(\lambda)}}\prod_{i=1}^{\ell}\lambda_i
=z_{\lambda}
\sum_{\alpha\sim\lambda}
\frac{\Psi_{\alpha}(\xvec)}{z_{\alpha}}
=\sum_{\alpha\sim\lambda}
\Psi_{\alpha}(\xvec)
\,.
\]
Here $\sigma(\lambda)$ denotes the composition $(\lambda_{\sigma^{-1}(1)},\dotsc,\lambda_{\sigma^{-1}({\ell})})$, and the last two sums are taken over all compositions $\alpha$ whose parts rearrange to $\lambda$.

\subsection{Schur functions}
\label[section]{sec:schur}

While Schur functions are not $p$-positive, as is evident from the famous Murnaghan--Nakayama rule, 
we still obtain an expansion of Schur functions into quasisymmetric power sums from the results in \cref{sec:gessel}.
The well-known formula below expresses the Schur functions in 
the fundamental quasisymmetric functions (see for example \cite[Thm. 7.19.7]{StanleyEC2}):
\begin{align}\label{eq:SchurintoGessel}
\schurS_\lambda(\xvec) = \sum_{T \in \SYT(\lambda)} \gessel_{n,\DES(T)}(\xvec).
\end{align}
The set $\SYT(\lambda)$ is the set of standard Young tableaux of shape $\lambda \vdash n$,
and the descent set of such a standard Young tableau is the set of entries $i$ such that $i+1$ appears in a row with a higher index.
The following result due to Y.~Roichman is an immediate consequence of \eqref{eq:SchurintoGessel} and \cref{prop:fundToPowerSumLift}.

\begin{theorem}[{\cite[Thm. 4]{Roichman1997}}]\label{thm:schurRoichmanPexp}
Let $\lambda$ be a partition of $n$.
Then the expansion of the Schur function $\schurS_{\lambda}$ into power sum symmetric functions is given by
\begin{align}\label{eq:schurRoichmanPexp}
\schurS_\lambda(\xvec)
=
\sum_{\mu \vdash n}
\frac{\psumP_\mu(\xvec)}{z_\mu}
\sum_{\substack{ T \in \SYT(\lambda) \\  \DES(T) \in U_\mu }}
(-1)^{\DES(T) \setminus S_\mu}
\,.
\end{align}
\end{theorem}

\begin{example}
Let $\lambda = (3,3)$. Then $\SYT(\lambda)$ is the following set of standard Young tableaux,
where the descents have been marked bold.
\begin{align*}
\begin{tikzpicture}[scale=.4,thinLine/.style={line width=\lineThickness pt},thickLine/.style={line width=2*\lineThickness pt,line join=round},entry/.style={xshift=5mm,yshift=5mm,font= \small}]
\pgfmathsetmacro{\lineThickness}{.5}
\begin{scope}
\draw[thinLine](0,1)--(3,1)(1,0)--(1,2)(2,0)--(2,2);
\draw[thickLine](0,0)--(3,0)--(3,2)--(0,2)--cycle;
\draw[entry](0,1)node{$1$}(1,1)node{$2$}(2,1)node{$\boldsymbol{3}$}(0,0)node{$4$}(1,0)node{$5$}(2,0)node{$6$};
\end{scope}
\begin{scope}[xshift=4cm]
\draw[thinLine](0,1)--(3,1)(1,0)--(1,2)(2,0)--(2,2);
\draw[thickLine](0,0)--(3,0)--(3,2)--(0,2)--cycle;
\draw[entry](0,1)node{$1$}(1,1)node{$\boldsymbol{2}$}(2,1)node{$\boldsymbol{4}$}(0,0)node{$3$}(1,0)node{$5$}(2,0)node{$6$};
\end{scope}
\begin{scope}[xshift=8cm]
\draw[thinLine](0,1)--(3,1)(1,0)--(1,2)(2,0)--(2,2);
\draw[thickLine](0,0)--(3,0)--(3,2)--(0,2)--cycle;
\draw[entry](0,1)node{$\boldsymbol{1}$}(1,1)node{$3$}(2,1)node{$\boldsymbol{4}$}(0,0)node{$2$}(1,0)node{$5$}(2,0)node{$6$};
\end{scope}
\begin{scope}[xshift=12cm]
\draw[thinLine](0,1)--(3,1)(1,0)--(1,2)(2,0)--(2,2);
\draw[thickLine](0,0)--(3,0)--(3,2)--(0,2)--cycle;
\draw[entry](0,1)node{$1$}(1,1)node{$\boldsymbol{2}$}(2,1)node{$\boldsymbol{5}$}(0,0)node{$3$}(1,0)node{$4$}(2,0)node{$6$};
\end{scope}
\begin{scope}[xshift=16cm]
\draw[thinLine](0,1)--(3,1)(1,0)--(1,2)(2,0)--(2,2);
\draw[thickLine](0,0)--(3,0)--(3,2)--(0,2)--cycle;
\draw[entry](0,1)node{{$\boldsymbol{1}$}}(1,1)node{$\boldsymbol{3}$}(2,1)node{$\boldsymbol{5}$}(0,0)node{$2$}(1,0)node{$4$}(2,0)node{$6$};
\end{scope}
\end{tikzpicture}
\end{align*}
If $\mu = (2,2,2)$ then $S_\mu = \{2,4\}$.
We can check that the descent sets of all five tableaux
are $\mu$-unimodal, and contribute the signs $-1$, $+1$, $-1$, $-1$ and $-1$ respectively.
Furthermore, $U_\mu$ is the set of all subsets of $[5]$.
The formula in \eqref{eq:schurRoichmanPexp}
then gives 
\[
\schurS_{33}(\xvec) = \dotsb+ (-3)\frac{\psumP_{222}(\xvec)}{z_{222}} + \dotsb.
\]
In contrast, the Murnaghan--Nakayama rule --- given as a sum over so called rim-hook tableaux --- is cancellation-free
for this choice of $\lambda$ and $\mu$, and is given as a sum over exactly three rim-hook tableaux.

\end{example}

\subsection{Chromatic quasisymmetric functions}
\label[section]{sec:chromatic}

In 1995, R.~Stanley introduced a symmetric function generalization of the 
chromatic polynomal, \cite{Stanley95Chromatic}. 
This definition was later refined by J.~Shareshian and M.~Wachs in~\cite{ShareshianWachs2016},
where a $q$-parameter was introduced.

\begin{definition}*\label{def:qsymChromatic}
Let $G$ be a directed graph (no loops, but multiple edges are allowed) on the vertex set $[n]$. A \emph{coloring} of $G$
is an assignment of colors in $\setP$ to the vertices. A coloring is \emph{proper}
if vertices connected by an edge are assigned different colors.
An \emph{ascent}\footnote{In the case of multiple edges $(i,j)$, the contribution to $\asc$ is the multiplicity.} 
of a coloring $\coloring$ is a directed edge $(i,j)$ of $G$ 
such that $\coloring(i) < \coloring(j)$. 
The number of ascents of a coloring is denoted $\asc(\coloring)$.

The \emph{chromatic quasisymmetric function} of $G$ is defined in \cite{ShareshianWachs2016}\footnote{
The definition in \cite{ShareshianWachs2016} is slightly less general, and uses the acyclic orientation of 
the edges determined by the labeling of the graph.
The more general definition was introduced in \cite{Ellzey2016}.
} as
\[
\chrom_G(\xvec;q) = \sum_{\substack{\coloring: G \to \setP \\ \coloring \text{ proper}}} 
\xvar_{\coloring(1)} \dotsm \xvar_{\coloring(n)} q^{\asc(\coloring)}
.
\qeddefhere
\]
\end{definition}
When $q=1$ we obtain the \emph{chromatic symmetric function} $\chrom_G(\xvec)$ in
\cite{Stanley95Chromatic}. 
The function $\chrom_G(\xvec)$ is a symmetric function, and does not depend on the orientation of $G$.
R.~Stanley proves that $\omega \chrom_G(\xvec)$ is $p$-positive for any (undirected) graph $G$.

For some choices of directed graphs $G$, $\chrom_G(\xvec;q)$ is a symmetric function. 
A class of such graphs is characterized by B.~Ellzey in \cite[Thm.~5.5]{Ellzey2016}.
In particular, if $G$ is the incomparability 
graph of a $3+1$ and $2+2$-avoiding poset (together with a certain associated orientation), then
the function $\chrom_G(\xvec;q)$ is symmetric --- these graphs are referred to as \emph{unit-interval graphs}.

B.~Ellzey also proves that $\chrom_G(\xvec;q)$
has the property that $\omega \chrom_G(\xvec;q)$ expands positively 
in the power-sum symmetric functions whenever it is symmetric.
The only part of her proof that requires the restriction to symmetric functions is the application of \cref{prop:fundToPowerSumLift}.
We can now prove the following generalization of 
the main result in \cite[Thm. 4.1]{Ellzey2016}:

\begin{theorem}\label{thm:chromPsiPos}
Let $G$ be a directed graph and consider the 
expansion into quasisymmetric power sums
\begin{equation}\label{eq:chromPsiPos}
\omega \chrom_G(\xvec;q) = \sum_{\alpha} c^G_{\alpha}(q) \frac{\Psi_\alpha(\xvec)}{z_\alpha}.
\end{equation}
Then $c^G_{\alpha}(q) \in \setN[q]$ for all compositions $\alpha$.
\end{theorem}
\begin{proof}
Let $AO(G)$ denote the set of acyclic orientations of the graph $G$ viewed as an undirected graph.
For an orientation $\theta$, we let $\asc(\theta)$ be the number of edges oriented in the same direction as in $G$.
It is straightforward to prove (see \cite{AlexanderssonPanova2016,Ellzey2016}) that
\[
\chrom_G(\xvec;q) = \sum_{\theta \in AO(G)} q^{\asc(\theta)}  K_{P(\theta),w}(\xvec)
\]
where $P(\theta)$ is the poset obtained from $\theta$ by taking the transitive closure of the 
directed edges, and $w=w(\theta)$ is order-reversing.
The statement now follows from \cref{cor:strictOmegaPsiPos}.
\end{proof}

Note that different types of combinatorial interpretations for the coefficients $c_{\alpha}^G(q)$ in \eqref{thm:chromPsiPos} are known in certain special cases, see~\cite{Athanasiadis15,Ellzey2016}.
Our approach yields a combinatorial interpretation in the more general setting, which is similar to~\cite{Ellzey2016}.

It was conjectured in \cite[Conj. 7.6]{ShareshianWachs2016} that the coefficients $c^G_{\alpha}(q)$
are unimodal for unit-interval graphs. This conjecture is still open.
However, this conjecture does not extend to the general quasisymmetric setting.
\begin{example}
Consider the following directed graphs $G$ and $H$. 
\begin{align*}
\begin{tikzpicture}[scale=.8,baseline=-10mm,
circ/.style={circle,draw,inner sep=0.8pt, minimum width=12pt,font=\footnotesize}]
\begin{scope}
\node[] (G) at (0,2) {$G$};
\node[circ] (1) at (1,0) {$v_1$};
\node[circ] (2) at (0,1) {$v_2$};
\node[circ] (3) at (2,1) {$v_3$};
\node[circ] (4) at (1,2) {$v_4$};
\draw[black,thick,->] (1)--(2);
\draw[black,thick,->] (1)--(3);
\draw[black,thick,->] (4)--(2);
\draw[black,thick,->] (4)--(3);
\end{scope}
\begin{scope}[xshift=4cm,yshift=1cm]
\node[] (H) at (1,1) {$H$};
\node[circ] (1) at (0,0) {$v_1$};
\node[circ] (2) at (1,0) {$v_2$};
\node[circ] (3) at (2,0) {$v_3$};
\node[circ] (4) at (3,0) {$v_4$};
\node[circ] (5) at (4,0) {$v_5$};
\draw[black,thick,->] (1)--(2);
\draw[black,thick,->] (3)--(2);
\draw[black,thick,->] (4)--(3);
\draw[black,thick,->] (4)--(5);
\end{scope}
\end{tikzpicture}
\end{align*}
Then 
\begin{align*}
\omega \chrom_G(\xvec;q) &= 
(4 q+4 q^2+4 q^3) \frac{\Psi_4}{z_4}+
(2+4 q+4 q^3+2 q^4) \frac{\Psi_{13}}{z_{13}} \\
& +(4 q+8 q^2+4 q^3) \frac{\Psi_{22}}{z_{22}}+
(4 q+4 q^2+4 q^3) \frac{\Psi_{31}}{z_{31}} \\
&+
(4 q+8 q^2+4 q^3) \frac{\Psi_{112}}{z_{112}}+
(4+4 q+4 q^3+4 q^4)\frac{\Psi_{121}}{z_{121}} \\
&+
(4 q+8 q^2+4 q^3) \frac{\Psi_{211}}{z_{211}}+
(4+4 q+8 q^2+4 q^3+4 q^4) \frac{\Psi_{1111}}{z_{1111}}
\,.
\end{align*}
Note that the coefficient of $\Psi_{121}$ is not unimodal.

For the graph $H$ we have
\begin{align*}
\omega \chrom_H(\xvec;q) &= \dotsb +
(2+5 q+4 q^2+5 q^3+2 q^4) \frac{\Psi_{131}}{z_{131}} + \dotsb
\,,
\end{align*}
where the coefficient of $\Psi_{131}$ is not unimodal.
\end{example}

\subsection{\texorpdfstring{$k$}{k}-balanced chromatic quasisymmetric functions}

B.~Humpert introduces another quasisymmetric generalization of chromatic symmetric functions in \cite{Humpert2011}.

\begin{definition}*
Let $G$ be an oriented graph (no loops or multiple edges) on the vertex set $[n]$ and let $k \in \setP$.
An orientation $\theta$ of $G$ is said to be \emph{$k$-balanced} if for every undirected cycle in $G$,
walking along the cycle one traverses at least $k$ edges forward, and at least $k$ edges backwards.
Thus an orientation is acyclic if and only if it is $1$-balanced.

A proper coloring $\coloring$ of $G$ induces an acyclic orientation $\theta(\coloring)$ of $G$ 
by orienting edges towards the vertex with larger color.

The \emph{$k$-balanced chromatic quasisymmetric function} (\cite{Humpert2011} defined this only for $q=1$) is defined as
\[
\chrom^k_G(\xvec;q) = \sum_{\substack{\coloring: G \to \setP \\ \coloring \text{ proper} \\ \theta(\coloring) \text{ is $k$-balanced}}} 
\xvar_{\coloring(1)} \dotsm \xvar_{\coloring(n)} q^{\asc(\coloring)}.
\qeddefhere
\]
\end{definition}
Note that for $k=1$ we recover the quasisymmetric function in \cref{def:qsymChromatic} as $\chrom^1_G(\xvec;q) = \chrom_G(\xvec;q)$.

\begin{proposition}[{\cite[Thm. 3.4]{Humpert2011}}]
\label[proposition]{prop:Humpert}
The $k$-balanced chromatic quasisymmetric function of an oriented graph $G$ has the expansion 
\[
\chrom^k_G(\xvec;q) = \sum_{\substack{\theta \in O(G) \\ \theta \text{ is $k$-balanced }}} q^{\asc(\theta)} K_{P(\theta),w}(\xvec)
\]
where the sum taken is over all $k$-balanced orientations of $G$, $P(\theta)$ is the transitive closure of the directed edges 
and $w=w(\theta)$ is order-reversing.
\end{proposition}

From \cref{prop:Humpert} and \cref{cor:strictOmegaPsiPos} we obtain the following consequence.

\begin{corollary}
Let $G$ be any oriented graph and consider the expansion into quasisymmetric power sums
\[
\omega \chrom^k_G(\xvec;q) = \sum_{\alpha} c^G_{\alpha}(q) \frac{\Psi_\alpha(\xvec)}{z_\alpha}.
\]
Then $c^G_{\alpha}(q) \in \setN[q]$ for all compositions $\alpha$.
\end{corollary}

\subsection{LLT polynomials}

The LLT polynomials were introduced by A.~Lascoux, B.~Leclerc and J.-Y.~Thibon in \cite{Lascoux97ribbontableaux} using \emph{ribbon tableaux}.
The LLT polynomials can be seen as $q$-deformations of products of Schur functions and there are several open problems regarding LLT polynomials.
A different combinatorial model for the LLT 
polynomials was considered in \cite{Haglund2005Macdonald},
where each $k$-tuple of skew shapes index an LLT polynomial.
When each such skew shape is a skew Young diagram with a single box,
we say that the LLT polynomial is \emph{unicellular}.
The unicellular LLT polynomials have a central role in the work 
of E.~Carlsson and A.~Mellit~\cite{CarlssonMellit2017}, in which they introduced a combinatorial model for the unicellular LLT polynomials using Dyck paths.
In \cite{AlexanderssonPanova2016} this Dyck path model was extended to certain directed graphs.

By modifying the definition of the chromatic symmetric functions slightly,
we recover the unicellular LLT polynomials considered in \cite{AlexanderssonPanova2016}:
\begin{definition}
Let $G$ be a directed graph (no loops, but multiple edges are allowed) on the vertex set $[n]$.
The \emph{unicellular graph LLT polynomial} is defined as
\[
\LLT_G(\xvec;q) = \sum_{\substack{\coloring: G \to \setP }} 
\xvar_{\coloring(1)} \dotsm \xvar_{\coloring(n)} q^{\asc(\coloring)}.
\]
Note that we now sum over all colorings.
\end{definition}
The $\LLT_G(\xvec;q)$ are in general only quasisymmetric, 
but for certain choices of $G$ (the same choices as for the chromatic quasisymmetric functions)
they turn out to be symmetric and contain the family of unicellular LLT polynomials,
see \cite{AlexanderssonPanova2016}.

It was observed in \cite{AlexanderssonPanova2016,HaglundWilson2017} 
that $\omega \LLT_G(\xvec;q+1)$ is $p$-positive whenever $\LLT_G(\xvec;q)$
is a unicellular LLT polynomial.
We can now give a proof of the following much stronger statement.
\begin{theorem}\label{thm:lltPsiPos}
Let $G$ be a directed graph and consider the expansion into quasisymmetric power sums
\[
\omega \LLT_G(\xvec;q+1) = \sum_{\alpha} c^G_{\alpha}(q) \frac{\Psi_\alpha(\xvec)}{z_\alpha}.
\]
Then $c^G_{\alpha}(q) \in \setN[q]$ for all compositions $\alpha$.
\end{theorem}
\begin{proof}
Let $O(G)$ denote the set of orientations of the graph $G$ viewed as an undirected graph.
For $\theta \in O(G)$ we let $\asc(\theta)$ be the number of edges oriented in the same direction as in $G$.
Similar as in the proof of \cref{thm:chromPsiPos}, we have
\[
\LLT_G(\xvec;q+1) = \sum_{\theta \in O(G)} q^{\asc(\theta)}  K_{P(\theta),w}(\xvec)
\]
where $P(\theta)$ is the transitive closure of \emph{only} the edges of $\theta$
oriented in the same manner as in $G$ and $w=w(\theta)$ is order-reversing.
Note that we let $K_{P(\theta),w}(\xvec) \coloneqq 0$ when $P(\theta)$ has a cycle --- this can only happen if $G$
has a directed cycle.

Again the result follows from \cref{cor:strictOmegaPsiPos}.
\end{proof}

\bigskip 

We can enlarge the family of unicellular graph LLT polynomials.
\begin{definition}
Let $G$ be a directed graph on the vertex set $[n]$ and let $S$ be a subset of 
the edges of $G$.
The \emph{vertical strip graph LLT polynomial} is defined as
\[
\LLT_{G,S}(\xvec;q)
=
\sum_{
\substack{\coloring: G \to \setP \\
(i,j) \in S \Rightarrow \coloring(i) < \coloring(j)}
} 
\xvar_{\coloring(1)} \dotsm \xvar_{\coloring(n)} q^{\asc(\coloring)-|S|},
\]
where we sum over all colorings such that $\coloring(i) < \coloring(j)$ whenever $(i, j)$ is a (directed) edge in $S$.
\end{definition}
The name ``vertical strip graph LLT polynomials'' is motivated as follows.
For some choices of $G$ and $S$ we recover the family of LLT polynomials that are,
in the model introduced in \cite{Haglund2005Macdonald},
indexed by $k$-tuples of 
vertical strips.
Vertical strip LLT polynomials occur naturally in the study of 
the delta operator and diagonal harmonics.
The family of vertical strip LLT polynomials contains (a version of) modified Hall--Littlewood polynomials.
See \cite{AlexanderssonPanova2016} for an explicit construction of the correspondence between the above model
and the model in \cite{Haglund2005Macdonald}.

\begin{theorem}\label{thm:lltVertPsiPos}
Let $G$ be a directed graph, $S$ a subset of the edges of $G$, and let
\begin{align}\label{eq:vertlltPsiExp}
\omega \LLT_{G,S}(\xvec;q+1) = \sum_{\alpha} c^{G,S}_{\alpha}(q) \frac{\Psi_\alpha(\xvec)}{z_\alpha} 
\end{align}
be the expansion into quasisymmetric power sums.
Then $c^G_{\alpha}(q)\in\setN[q]$ for all compositions $\alpha$.
\end{theorem}
\begin{proof}
The same technique as above (also in \cite{AlexanderssonPanova2016}) shows that 
\[
\LLT_{G,S}(\xvec;q+1) = \sum_{\theta \in O_S(G)} q^{\asc(\theta)-|S|}  K_{P(\theta),w}(\xvec)
\]
where $O_S(G)$ is now the subset of orientations of $G$ such that edges in $S$ are oriented as in $G$.
\end{proof}
A special case of \cref{thm:lltVertPsiPos} was proved in \cite{AlexanderssonPanova2016}.

It is conjectured by P. Alexandersson and G. Panova in \cite{AlexanderssonPanova2016} that 
the coefficients $c^{G}_{\alpha}(q)$ in \cref{thm:lltPsiPos} are unimodal whenever $G$ is a unit interval graph.
Computer experiments suggests that this conjecture extends to the more general setting in \cref{thm:lltPsiPos}.

\begin{conjecture}\label[conjecture]{conj:unimodal}
Let $G$ be an oriented graph (no loops or multiple edges).
Then the coefficients $c_{\alpha}^G(q)\in\setN[q]$ in the expansion
\begin{equation*}
\omega\LLT_G(\xvec;q+1)
=
\sum_{\alpha}c_{\alpha}^G(q)\frac{\Psi_{\alpha}(\xvec)}{z_{\alpha}}
\end{equation*}
are unimodal for all compositions $\alpha$.
\end{conjecture}

\cref{conj:unimodal} has been verified for all oriented graphs with six or fewer vertices.\footnote{
See \cite[\oeis{A001174}]{OEIS} for the number of such graphs.}
In contrast, we note that the coefficients $c^{G,S}_{\alpha}(q)$ in \eqref{eq:vertlltPsiExp}
are \emph{not} unimodal in general. For example,
\[
G = \{1\to 2, 1\to 3, 2 \to 4, 2\to 5\}, \qquad S = \{1\to 2, 1\to 3\}
\]
gives
\[
\LLT_{G,S}(\xvec;q+1) = \dotsb + (1+q^2)\frac{\Psi_{12}}{z_{12}} + \dotsb.
\]

\bigskip 
It is possible to refine Theorems~\ref{thm:chromPsiPos}, \ref{thm:lltPsiPos} and~\ref{thm:lltVertPsiPos} by
assigning a different $q$-weight to each edge of $G$, so that for a coloring $\coloring$
we let
\[
\qvec^{\asc(\coloring)}
\coloneqq
\prod_{\substack{(i,j)\in E(G)\\
\coloring(i)<\coloring(j)}}
q_{i,j}
\,.
\]
The resulting functions are again quasisymmetric and the 
analogues of the above theorems can be proved in the same manner. 
We leave out the details.

\subsection{Tutte quasisymmetric functions}
\label[section]{sec:tutte}

In \cite[Definition 3.1]{Stanley98Chromatic} R.~Stanley defines the \emph{multivariate Tutte polynomial}
of a graph $G$ with vertex set $[n]$ as
\[
\mathrm{Tutte}_G(\xvec;q) \coloneqq \sum_{\coloring:G \to \setP} \xvar_{\coloring(1)} \dotsm \xvar_{\coloring(n)} (1+q)^{m(\coloring)}
\,,
\]
where the sum ranges over all vertex colorings of $G$ and $m(\coloring)$ denotes 
the number of monochromatic edges --- edges $\{i,j\}$ such that $\coloring(i)=\coloring(j)$.
It is evident that this is a symmetric function and it is straightforward to prove (see \cite{Stanley98Chromatic}) that
\begin{equation}\label{eq:Tutte}
\mathrm{Tutte}_G(\xvec;q) = \sum_{S\subseteq E(G)} q^{|S|}\psumP_{\lambda(S)}(\xvec)
\end{equation}
where the sum ranges over all subsets of the edges of $G$, and $\lambda(S)$ is the partition whose parts are the sizes of the connected components of the subgraph of $G$ spanned by the edges in $S$.
\medskip

J.~Awan and O.~Bernardi \cite{AwanBernardi2016} define a quasisymmetric generalization of the Tutte polynomial,
which they call the $B$-polynomial.
\begin{definition}*
Let $G$ be a directed graph on the vertices $[n]$. Let 
\begin{align}
B_{G}(\xvec;y,z) \coloneqq \sum_{\coloring : G \to \setP} \xvar_{\coloring(1)} \dotsm \xvar_{\coloring(n)} y^{\asc(\coloring)} z^{\inv(\coloring)}
\end{align}
where $\asc(\coloring)$ and $\inv(\coloring)$ are defined as
\begin{align*}
\asc(\coloring) = |\{ (i,j) : \coloring(i) < \coloring(j) \}|
&&\text{and}&& 
\inv(\coloring) = |\{ (i,j) : \coloring(i) > \coloring(j) \}|.
\qeddefhere
\end{align*}
\end{definition}
Notice that the chromatic quasisymmetric function can 
be obtained as $\chrom_G(\xvec;q) = [z^{n}] B_{G}(\xvec;qz,z)$,
and that the unicellular graph LLT polynomials can be obtained as  
$\LLT_G(\xvec;q) = B_{G}(\xvec;q,0)$.
Furthermore, for any directed graph $G$, the Tutte polynomials satisfies the relationship
\[
y^{|E(G)|}\mathrm{Tutte}_{\underline{G}}(\xvec;\tfrac{1}{y}-1) = B_{G}(\xvec;y,y).
\]
where $\underline{G}$ denotes the undirected version of $G$.

\begin{theorem}
Let $G$ be a directed graph and consider the expansion
\[
\omega B_{G}(\xvec;y+1,z+1) = \sum_{\alpha} c^G_{\alpha}(y,z) \frac{\Psi_\alpha(\xvec)}{z_\alpha}.
\]
Then $c^G_{\alpha}(y,z) \in \setN[y,z]$ for all compositions $\alpha$.
\end{theorem}
\begin{proof}
Let $E(G)$ be the set of directed edges of $G$. 
Then we have that
\[
B_{G}(\xvec;y+1,z+1) = \sum_{\substack{ A, I \subseteq E(G) \\ A \cap I = \emptyset }} y^{|A|} z^{|I|} K_{P(A,I),w}(\xvec)
\]
where $P(A,I)$ is the transitive closure of the directed edges 
\begin{align}\label{eq:qsymtutteEdges}
A\cup \{(j,i) : (i,j) \in I \}
\end{align}
and $w$ is an order-reversing labeling of $P(A,I)$. Here we let $K_{P(A,I),w}\coloneqq0$ if 
some edges in \eqref{eq:qsymtutteEdges} form a cycle.
By \cref{cor:strictOmegaPsiPos}, the statement follows.
\end{proof}

\subsection{Matroid quasisymmetric functions}
\label[section]{sec:matroid}

In 2009 L.~Billera, N.~Jia and V.~Reiner introduced a quasisymmetric 
function associated to matroids as a new matroid invariant, see \cite{BilleraJiaReiner2009}.
The definition is as follows:
\begin{definition}*
Let $M$ be a matroid with ground set $E$ and bases $\matBasisB(M)$. A map $f:E \to \setP$
is said to be $M$-\emph{generic} if the sum $f(B) \coloneqq \sum_{e \in B} f(e)$
is minimized by a unique $B \in \matBasisB(M)$.
An $M$-generic function $f$ must also be injective.

The \emph{matroid quasisymmetric function} is then defined as
\begin{align}
F(M,\xvec) \coloneqq \sum_{\text{$f$ $M$-generic}} \prod_{e\in E} \xvar_{f(e)}.
\qeddefhere
\end{align}
\end{definition}

In \cite{BilleraJiaReiner2009} it is proved that $F(M,\xvec)$ is indeed a quasisymmetric function.

\medskip

Let $M = (E,\matBasisB(M))$ be a matroid. 
Given a basis $B \in \matBasisB(M)$ let $B^* \coloneqq E\setminus B$ and
define the poset $P_B$ on the vertex set $E = B \sqcup B^*$ such that
\[
e \prec e' \text{ if and only if } e \in B \text{ and } (B\setminus\{e\}) \cup \{e'\} \text{ is in } \matBasisB(M).
\]

\begin{theorem}[{\cite[Thm. 5.2]{BilleraJiaReiner2009}}]
\label[theorem]{thm:matroid}
Let $M = (E,\matBasisB(M))$ be a matroid. Then 
\[
F(M,\xvec) = \sum_{B \in \matBasisB(M)} K_{P_B,w}(\xvec)
\]
where $w$ is any order-reversing labeling of $P_B$.
\end{theorem}

Using \cref{cor:strictOmegaPsiPos} we get the following corollary.

\begin{corollary}
Let $M$ be a matroid.
Then $\omega F(M,\xvec)$ is $\Psi$-positive.
\end{corollary}

We compute the $\Psi$-expansion of the matroid quasisymmetric function $F(M,\xvec)$ 
explicitly in the case where $M$ is the uniform matroid.
\begin{example}*[Uniform matroid]
The uniform matroid $U=U_n^r$ has ground set $E=[n]$ and every $r$-element subset of $E$ constitutes a basis.
That is, $\matBasisB(U)=\binom{[n]}{r}$.
In this case the poset $P_B$ is given by $e\prec e'$ for all $e\in B$ and $e'\in B^*$.
In particular all posets $P_B$ for $B\in\matBasisB(U)$ are isomorphic.
The Hasse diagram of $P_B$ is the complete bipartite graph $K_{r,m}$, where $m=n-r$.

Fix a basis $B\in\matBasisB(U)$ and a natural labeling $w$ of $P_B$.
It follows from \cref{thm:matroid}, \cref{thm:KP_opsurj} and \cref{ex:complete_bipartite_graph} that
\begin{align*}
\omega F(U,\xvec)
&=\abs{\matBasisB(U)}\sum_{\alpha\vDash n}\frac{\Psi_{\alpha}(\xvec)}{z_{\alpha}}\abs{\opsurj_{\alpha}^{\ast}({P_B})}
\\
&=\binom{n}{r}\sum_{k=0}^m
\frac{\Psi_{(1^{r-1},k+1,1^{m-k})}(\xvec)}{(k+1)(n-k-1)!}
\cdot\frac{r!m!}{k!}
\\
&=\sum_{k=0}^m\binom{n}{k+1}\Psi_{(1^{r-1},k+1,1^{m-k})}(\xvec)
\,.\qeddefhere
\end{align*}
\end{example}

\subsection{Eulerian quasisymmetric functions}
\label[section]{sec:eulerian}

The aim of this section is to explain how the tools developed in this paper can be used to prove known $p$-expansions of the Eulerian quasisymmetric functions and the cycle Eulerian quasisymmetric functions of J.~Shareshian and M.~Wachs.

The \emph{Eulerian polynomials} are defined as the descent generating functions of the symmetric group
\[
A_n(q)
\coloneqq\sum_{\sigma\in\symS_n}q^{\des(\sigma)}
\,,
\]
where $\des(\sigma)\coloneqq\abs{\DES(\sigma)}$.
By convention $A_0(q)\coloneqq1$.

In \cite{ShareshianWachs2010} J.~Shareshian and M.~Wachs generalize the classical identity for the exponential generating function of Eulerian polynomials
\begin{equation}\label{eq:euler}
\sum_{n\geq 0}A_n(q)\frac{z^n}{n!}
=\frac{1-q}{\exp(z(1-q))-q}
\,,
\end{equation}
and introduced Eulerian quasisymmetric functions.

Let $\sigma\in\symS_n$ be a permutation.
Define the set of \emph{exceedences} of $\sigma$ as
\[
\EXC(\sigma)
\coloneqq\{i\in[n-1]:\sigma_i>i\}
\,,
\]
and set $\exc(\sigma)\coloneqq\abs{\EXC(\sigma)}$.
For $n\in\setN$ let $[\bar{n}]\coloneqq\{\bar{1},\dotsc,\bar{n}\}$ be a disjoint copy of the set $[n]$.
Define a total order on the alphabet $[\bar{n}]\cup[n]$ by
\begin{align}\label{eq:bar_order}
\bar{1}<\dotsm<\bar{n}<1<\dotsm<n
.
\end{align}
Given a permutation $\sigma=\sigma_1\dotsm\sigma_n$, define the word $\bar{\sigma}$ in the alphabet
$[\bar{n}]\cup[n]$ by replacing $\sigma_i$ with $\bar{\sigma_i}$ whenever $i\in\EXC(\sigma)$ is an exceedence.
Let
\[
\DEX(\sigma)\coloneqq\DES(\bar{\sigma}),
\]
where descents of $\bar{\sigma}$ are computed with respect to the order in \eqref{eq:bar_order}.
For example, let $\sigma=613542$.
Then $\EXC(\sigma)=\{1,4\}$, $\bar{\sigma}=\bar{6}13\bar{5}42$ and $\DEX(\sigma)=\{3,5\}$.

The \emph{Eulerian quasisymmetric functions} are defined as
\[
Q_{n,j}(\xvec)\coloneqq
\sum_{\substack{\sigma\in\symS_n \\ \exc(\sigma)=j}}\gessel_{n,\DEX(\sigma)}(\xvec)
\,.
\]
By definition $Q_{n,j}$ is quasisymmetric.
It turns out that $Q_{n,j}$ is in fact symmetric.

In \cite[Thm.~1.2]{ShareshianWachs2010} Eulerian quasisymmetric functions are shown to have the following generating function that specializes to \eqref{eq:euler}.
\begin{equation}\label{eq:euler_H}
\sum_{n,j\geq 1}
Q_{n,j}(\xvec)q^jz^n
=\frac{(1-q)H(\xvec;z)}{H(\xvec;qz)-qH(\xvec;z)}
\end{equation}
Here $H(\xvec;z)\coloneqq\sum_{n\geq0}\completeH_n(\xvec)z^n$ denotes the generating function of complete homogeneous symmetric functions.

From \eqref{eq:euler_H} J.~Shareshian and M.~Wachs deduce the following expansion of Eulerian quasisymmetric functions into power sum symmetric functions.

\begin{proposition}[{\cite[Prop.~6.6]{ShareshianWachs2010}}]
\label[proposition]{prop:Eulerian_p_expansion}
Let $n\in\setN$.
Then
\[
\sum_{j=0}^{n-1}q^jQ_{n,j}(\xvec)
=\sum_{\lambda\vdash n}
\frac{\psumP_{\lambda}(\xvec)}{z_{\lambda}}
A_{\ell(\lambda)}(q)\prod_{i=1}^{\ell(\lambda)}[\lambda_i]_q
\,,
\]
where $[a]_q\coloneqq\frac{1-q^a}{1-q}
$ denotes the usual $q$-integer.
\end{proposition}

We present an alternative proof of \cref{prop:Eulerian_p_expansion} using the theory of order-preserving surjections and an interpretation of Eulerian quasisymmetric functions as generating functions of 
banners, which was obtained in \cite[Sec.~3.2]{ShareshianWachs2010}.
Note that this also offers a different route to proving that $Q_{n,j}$ is symmetric and satisfies \eqref{eq:euler_H}.

Let $X$ and $\bar{X}$ denote disjoint copies of the positive integers, that is,
\begin{equation}\label{eq:Xbar}
X\coloneqq\{1,2,3,\dotsc\}
\qquad\text{and}\qquad
\bar{X}\coloneqq\{\bar{1},\bar{2},\bar{3},\dotsc\}
\,.
\end{equation}
Moreover define $\abs{\,\cdot\,}:X\cup\bar{X}\to\setN$ by $\abs{a}=\abs{\bar{a}}=a$.
A \emph{banner} of length $n$ is a word $b=b_1\dotsm b_n$ in the alphabet $X\cup\bar{X}$ 
such that the following three conditions are satisfied:
\begin{enumerate}[(i)]
\item If $b_i\in\bar{X}$ then $\abs{b_i}\geq\abs{b_{i+1}}$ for all $i\in[n-1]$.
\item If $b_i\in X$ then $\abs{b_i}\leq\abs{b_{i+1}}$ for all $i\in[n-1]$.
\item We have $b_n\in X$.
\end{enumerate}
Let $\banners_{n,j}$ denote the set of banners of length $n$ that contain exactly $j$ barred letters.
Given a banner $b\in\banners_{n,j}$ define its \emph{weight} as $\xvec^b\coloneqq \xvar_{\abs{b_1}}\dotsm \xvar_{\abs{b_n}}$.

It can be shown \cite[Thm.~3.6]{ShareshianWachs2010} that
\begin{equation}\label{eq:banners}
Q_{n,j}(\xvec)
=\sum_{b\in\banners_{n,j}}\xvec^b
\,.
\end{equation}
As was observed by R.~Stanley (see~\cite[Thm.~7.2]{ShareshianWachs2010} and the remarks thereafter), 
it is an immediate consequence of \eqref{eq:banners} that Eulerian quasisymmetric functions 
are related to reverse $P$-partitions and chromatic symmetric functions.

For $S\subseteq[n-1]$ let $\banners_{n,S}$ denote the set of banners $b$ such that $b_i\in\bar{X}$ if and only if $i\in S$.
Then
\[
\sum_{b\in\banners_{n,S}}\xvec^b
=K_{P_S}(\xvec)
\,,
\]
where $P_S$ is the poset defined in \cref{ex:path}.
Consequently
\begin{equation}\label{eq:Eulerian_KP}
\sum_{j=0}^{n-1}q^jQ_{n,j}(\xvec)
=\sum_{S\subseteq[n-1]}
q^{\abs{S}}K_{P_S}(\xvec)
=\omega X_{G}(\xvec;q)
\,,
\end{equation}
where $G$ denotes the directed path of length $n-1$.

\begin{proof}[Proof of \cref{prop:Eulerian_p_expansion}]
Let $S\subseteq[n-1]$.
By \cref{thm:KP_opsurj}
\[
K_{P_S}(\xvec)
=\sum_{\alpha\vDash n}
\frac{\Psi_{\alpha}(\xvec)}{z_{\alpha}}
\abs{\opsurj_{\alpha}^{\ast}(P_S)}
\,.
\]
In combination with \eqref{eq:Eulerian_KP} 
and \cref{ex:path} we obtain
\begin{equation*}\label{eq:Eulerian_Psi_expansion}
\begin{split}
\sum_{j=0}^{n-1}q^{j}Q_{n,j}(\xvec)
&=\sum_{\alpha\vDash n}
\frac{\Psi_{\alpha}(\xvec)}{z_{\alpha}}
\sum_{S\subseteq[n-1]}
q^{\abs{S}}
\abs{\opsurj_{\alpha}^{\ast}(P_S)}
\\
&=\sum_{\alpha\vDash n}
\frac{\Psi_{\alpha}(\xvec)}{z_{\alpha}}
A_{\ell(\alpha)}(q)\prod_{i=1}^{\ell(\alpha)}[\alpha_i]_q
\,.
\end{split}
\end{equation*}
The claim follows from \eqref{eq:powersum_Psi_expansion}.
\end{proof}

Using a few tricks one can also deal with the more challenging \emph{cycle Eulerian quasisymmetric functions} $Q_{(n),j}$, which are defined as
\[
Q_{(n),j}(\xvec)
\coloneqq
\sum_{\substack{\sigma\in\symS_n\text{ is a long cycle}\\ \exc(\sigma)=j}}\gessel_{n,\DEX(\sigma)}(\xvec)
\,.
\]
The Eulerian quasisymmetric functions $Q_{n,j}$ can be expressed in terms of cycle Eulerian quasisymmetric functions $Q_{(n),j}$ (and vice versa) via plethysm.
J.~Shareshian and M.~Wachs conjectured and later proved together with B.~Sagan the following expansion into power sum symmetric functions.

\begin{theorem}[{\cite[Conj.~6.5]{ShareshianWachs2010}, \cite[Thm.~4.1]{SaganShareshianWachs2011}}]
\label[theorem]{thm:cyclic_Eulerian_p_expansion}
Let $n\in\setN$.
Then
\begin{equation}
\label{eq:cyclic_Eulerian_p_expansion}
\sum_{j=0}^{n-1}q^jQ_{(n),j}(\xvec)
=\sum_{\lambda\vdash n}
\frac{\psumP_{\lambda}(\xvec)}{z_{\lambda}}
\sum_{d\mid\gcd(\lambda)}
\mu(d)\,d^{\ell(\lambda)-1}\,q^d\,A_{\ell(\lambda)-1}(q^d)
\prod_{i=1}^{\ell(\lambda)}
\left[\frac{\lambda_i}{d}\right]_{q^d}
\,,
\end{equation}
where $\mu$ denotes the number theoretic M{\"o}bius function, and $\gcd(\alpha)\coloneqq\gcd(\alpha_1,\dotsc,\alpha_{\ell})$ for all compositions $\alpha$ with $\ell$ parts.
\end{theorem} 

Note that our statement of \cref{thm:cyclic_Eulerian_p_expansion} differs slightly from the given references.
See \cite[Lem.~4.2]{SaganShareshianWachs2011} for a proof that the statements are equivalent.

In~\cite{SaganShareshianWachs2011} \cref{thm:cyclic_Eulerian_p_expansion} is derived from \cref{prop:Eulerian_p_expansion} by the use of plethystic calculus and the manipulation of formal power series.
We present a proof based on the theory of order-preserving surjections, M{\"o}bius inversion and an interpretation of the cycle Eulerian quasisymmetric functions as the generating functions of primitive necklaces due to J.~Shareshian and M.~Wachs.
The application of M{\"o}bius inversion to problems related to primitive necklaces is very classical 
and dates back (at least) to~\cite{MetropolisRota1983}.
In this way the M{\"o}bius function and the power $q^d$ on the right 
hand side of \eqref{eq:cyclic_Eulerian_p_expansion} appear naturally.

Let $X,\bar{X}$ and $\abs{\,\cdot\,}$ be as in \eqref{eq:Xbar}.
A \emph{bicolored necklace} of length $n$ is a \emph{circular word} $o_1\dotsm o_n$ in the alphabet $X\cup\bar{X}$ that satisfies the three conditions below.
Circular means that we do not distinguish between $o_1\dotsm o_n$ and $o_2\dotsm o_no_1$.
\begin{enumerate}[(i)]
\item If $o_i\in X$ then $\abs{o_i}\leq\abs{o_{i+1}}$ for all $i\in[n]$, where indices are viewed modulo $n$.
\item If $o_i\in\bar{X}$ then $\abs{o_i}\geq\abs{o_{i+1}}$ for all $i\in[n]$, where indices are viewed modulo $n$.
\item If $n=1$ then $o_1\in X$.
\end{enumerate}
A bicolored necklace $o$ is \emph{primitive} if the words $o_k\dotsm o_{k+n-1}$ are distinct for all $k\in[n]$, where indices are once more viewed modulo $n$.
Let $\necklaces_{(n),j}$ denote the set of primitive bicolored necklaces of length $n$ that contain exactly $j$ barred letters.
For example, the primitive bicolored necklaces in $\necklaces_{(3),1}$ with letters $1,2,\bar1,\bar2$ are
\[
11\bar1,\quad
22\bar2,\quad
12\bar2,\quad
11\bar2.
\]

Given a bicolored necklace $o_1\dotsm o_n$, define its \emph{weight} as $\xvec^o\coloneqq \xvar_{\abs{o_1}}\dotsm \xvar_{\abs{o_n}}$.
In \cite[Sec.~3.1]{ShareshianWachs2010} J.~Shareshian and M.~Wachs show that cycle Eulerian quasisymmetric functions are the generating functions of primitive bicolored necklaces.
\[
Q_{(n),j}(\xvec)
=\sum_{o\in\necklaces_{(n),j}}\xvec^o
\]
The quasisymmetric functions $Q_{(n),j}$ are not positive linear combinations of reverse $P$-partition enumerators $K_P$ for some simple set of posets $P$.
Instead, let $\banners_{n,j}'$ be the set of words
$b_1\dotsm b_n$ in the alphabet $X\cup\bar{X}$ that satisfy the following three properties:
\begin{enumerate}[(i)]
\item If $b_i\in X$ then $\abs{b_i}\leq\abs{b_{i+1}}$ for all $i\in[n]$, where indices are viewed modulo $n$.
\item If $b_i\in\bar{X}$ then $\abs{b_i}\geq\abs{b_{i+1}}$ for all $i\in[n]$, where indices are viewed modulo $n$.
\item The word $b$ contains exactly $j$ barred letters.
\end{enumerate}
Denote the generating function of $\banners_{n,j}'$ by
\begin{equation*}
F_{n,j}(\xvec)
\coloneqq
\sum_{b\in\banners_{n,j}'}\xvec^b
\,.
\end{equation*}
Note that every word $b_1\dotsm b_n$ can be written uniquely as $(a_1\dotsm a_d)^{n/d}$ where $d$ divides $n$ and $a_1\dotsm a_d$ is a primitive word of length $d$.
Moreover each primitive circular word $o_1\dotsm o_d$ gives rise to $d$ distinct primitive words.
Therefore the quasisymmetric functions $F_{n,j}$ are related to the cycle Eulerian quasisymmetric functions by
\begin{equation}\label{eq:Fnj_Qnj_expansion}
\sum_{j=0}^{n-1}q^jF_{n,j}(\xvec)
=
\sum_{d\mid n}\sum_{j=0}^{d-1}dq^{jn/d}Q_{(d),j}(\xvec^{n/d})
\,,
\end{equation}
where $\xvec^k$ denotes the variables $\xvar_1^k,\xvar_2^k,\dotsc$.

Contrary to the cycle Eulerian quasisymmetric functions $Q_{(n),j}$, the generating functions $F_{n,j}$ are immediately related to reverse $P$-partitions and chromatic quasisymmetric functions of cycles.
More precisely,
\begin{equation}\label{eq:Fnj_KPS_expansion}
\sum_{j=0}^{n-1}q^jF_{n,j}(\xvec)
=\psumP_n(\xvec)+\sum_{\substack{S\subseteq[n] \\0<\abs{S}<n}}q^{\abs{S}}K_{P_S}(\xvec)
=\psumP_n(\xvec)+\omega X_G(\xvec;q)
\,,
\end{equation}
where $P_S$ is defined as in \cref{ex:cycle}, and $G$ denotes the directed cycle of length $n$.

\begin{proposition}[{\cite[Thm.~4.4]{Ellzey2016}}]
\label{prop:necklaces_Psi_expansion}
Let $n\in\setN$.
Then
\[
\sum_{j=0}^{n-1}q^jF_{n,j}(\xvec)
=\psumP_{n}(\xvec)[n]_q
+\sum_{\substack{\lambda\vdash n\\ \ell(\lambda)\geq2}}
\frac{\psumP_{\lambda}(\xvec)}{z_{\lambda}}
\,nq\,A_{\ell(\lambda)-1}(q)\prod_{i=1}^{\ell(\lambda)}[\lambda_i]_{q}
\,.
\]
\end{proposition}

\begin{proof}
By \eqref{eq:Fnj_KPS_expansion} and \cref{thm:KP_opsurj} we obtain
\begin{equation}\label{eq:Fnj_Psi_expansion}
\sum_{j=0}^{n-1}q^jF_{n,j}(\xvec)
=\psumP_{n}(\xvec)
+\sum_{\alpha\vDash n}
\frac{\Psi_{\alpha}(\xvec)}{z_{\alpha}}
\sum_{\substack{S\subseteq[n]\\ 0<\abs{S}<n}}q^{\abs{S}}\abs{\opsurj_{\alpha}^{\ast}(P_S)}.
\end{equation}
By 
\cref{ex:cycle} the right hand side of \eqref{eq:Fnj_Psi_expansion} is equal to
\begin{equation*}
\psumP_{n}(\xvec)(1+q[n-1]_q)
+\sum_{\substack{\alpha\vDash n\\ \ell(\alpha)\geq2}}
\frac{\Psi_{\alpha}(\xvec)}{z_{\alpha}}
\,nq\,A_{\ell(\alpha)-1}(q)\prod_{i=1}^{\ell(\alpha)}[\alpha_i]_q
\,.
\end{equation*}
The claim follows from \eqref{eq:powersum_Psi_expansion}.
\end{proof}

Note that in particular \cref{prop:necklaces_Psi_expansion} implies that $F_{n,j}(\xvec)$ is symmetric.

\begin{proof}[Proof of \cref{thm:cyclic_Eulerian_p_expansion}]
We can generalize \eqref{eq:Fnj_Qnj_expansion} as follows:
For all $k\in[n]$ with $k\mid n$ we have
\[
\sum_{j=0}^{k-1}q^{jn/k}F_{k,j}(\xvec^{n/k})
=\sum_{d\mid k}\sum_{j=0}^{d-1}dq^{jn/d}Q_{(d),j}(\xvec^{n/d})
\,.
\]
The M{\"o}bius inversion formula for the interval $[1,n]$ in the divisor lattice yields
\begin{equation}\label{eq:Qnj_Fnj_expansion}
\sum_{j=0}^{n-1}q^jQ_{(n),j}(\xvec)
=\frac{1}{n}\sum_{d\mid n}\mu(d)\sum_{j=0}^{n/d-1}q^{dj}F_{n/d,j}(\xvec^d)
\,.
\end{equation}
Given a partition $\lambda$ of length $\ell$, let $d\lambda\coloneqq(d\lambda_1,\dotsc,d\lambda_{\ell})$.
By \cref{prop:necklaces_Psi_expansion} the right hand side of \eqref{eq:Qnj_Fnj_expansion} is equal to
\begin{equation*}
\begin{split}
&\sum_{d\mid n}\frac{\mu(d)}{d}
\Bigg(
\frac{\psumP_{n}(\xvec)}{z_{(n/d)}}[n/d]_{q^d}
+\sum_{\substack{\lambda\vdash n/d \\ \ell(\lambda)\geq2}}\frac{\psumP_{d\lambda}(\xvec)}{z_{\lambda}}\,q^d\,A_{\ell(\lambda)-1}(q^d)\prod_{i=1}^{\ell(\lambda)}[\lambda_i]_{q^d}
\Bigg)
\\
&=\sum_{\lambda\vdash n}\frac{\psumP_{\lambda}(\xvec)}{z_{\lambda}}
\sum_{d\mid\gcd(\lambda)}\mu(d)\,d^{\ell(\lambda)-1}\,B_{\ell(\lambda)-1}(q^d)\prod_{i=1}^{\ell(\lambda)}\left[\frac{\lambda_i}{d}\right]_{q^d}
\,,
\end{split}
\end{equation*}
where $B_k(q)\coloneqq qA_k(q)$ if $k>0$, and $B_0(q)\coloneqq1$.
The claim follows from the amusing identity
\begin{equation*}
\sum_{d\mid n}\mu(d)[n/d]_{q^d}
=\sum_{d\mid n}\mu(d)q^d[n/d]_{q^d}
\,.\qedhere
\end{equation*}
\end{proof}

\section{Further directions}
\label[section]{sec:directions}

\subsection{Symmetry}

There is an almost trivial necessary and sufficient condition for when a linear combination of quasisymmetric power sums is a symmetric function that  follows from \eqref{eq:powersum_Psi_expansion}.

\begin{proposition}\label[proposition]{prop:symmetry}
Let $X$ be a quasisymmetric function with
\begin{equation*}
X(\xvec)=\sum_{\alpha}c_{\alpha}\Psi_{\alpha}(\xvec).
\end{equation*}
Then $X$ is symmetric if and only if $c_{\alpha}=c_{\beta}$ for all compositions $\alpha$ and $\beta$ such that $\beta$ can be obtained by permuting the parts of $\alpha$.
\end{proposition}

\cref{prop:symmetry} offers a method for proving (or disproving) that a given quasisymmetric function is symmetric.
As mentioned at the end of \cref{sec:revPP}, it would be particularly interesting if this idea was applicable to the generating functions $K_{P,w}$.

Furthermore by \cref{prop:symmetry} any symmetric function for which the expansion into quasisymmetric power sums is known, immediately gives rise to a set of symmetries on its coefficients $c_{\alpha}$.
These symmetries might not at all be obvious from a purely combinatorial point of view.
For example, consider the expansion of Schur functions into quasisymmetric power sums given in \cref{sec:schur}.
It follows from the symmetry of Schur functions that the sum
\begin{equation*}
\sum_{\substack{
T\in\SYT(\lambda) \\
\DES(T)\in U_{\alpha}
}}
(-1)^{\DES(T)\setminus S_{\alpha}}
\end{equation*}
is invariant under the permutation of the parts of $\alpha$.
The authors are unaware of a proof of this fact that does not appeal to the theory of symmetric functions.
We expect that many potentially interesting combinatorial problems can be obtained in similar fashion.

Lastly, if a combinatorial statistic $c$ on compositions appears to satisfy the symmetry properties of \cref{prop:symmetry}, then it might be worth a try to investigate the quasisymmetric function
\begin{equation*}
X_n(\xvec)
\coloneqq
\sum_{\alpha\vDash n}
c(\alpha)\frac{\Psi_{\alpha}}{z_{\alpha}}
\,.
\end{equation*}
Proving that $X_n$ is symmetric, for example by deriving its expansion into the fundamental or monomial bases, will also prove the symmetry of $c$.

\subsection{Schur-positivity and \texorpdfstring{$h$}{h}-positivity}

There are open problems regarding a combinatorial proof of the Schur-positivity of LLT polynomials \cite{Haglund2005Macdonald},
as well as proving $\elementaryE$-positivity of chromatic symmetric functions \cite{StanleyStembridge1993}.

Since both families of polynomials are related to the $K_P(\xvec)$,
it is natural to ask if every symmetric positive linear combination of such functions is $h$-positive (or weaker, Schur-positive).
This is not the case.
A computer search gave us the following counterexamples:
\begin{align*}
\begin{tikzpicture}[scale=0.8,baseline=-5mm,
circ/.style={circle,draw,inner sep=0.8pt, minimum width=12pt,font=\footnotesize}
]
\draw(1,2)node{$A$};
\node[circ] (A) at (0,0) {$v_1$}; 
\node[circ] (B) at (1,0) {$v_2$};
\node[circ] (C) at (2,0) {$v_3$};
\node[circ] (D) at (1,1) {$v_4$};
\draw[black,thick] (A)--(D)--(B)(C)--(D);
\end{tikzpicture}
&&
\begin{tikzpicture}[scale=0.8,baseline=-5mm,
circ/.style={circle,draw,inner sep=0.8pt, minimum width=12pt,font=\footnotesize}
]
\draw(.5,2)node{$B$};
\node[circ] (A) at (0,0) {$v_1$};
\node[circ] (B) at (1,0) {$v_2$};
\node[circ] (C) at (0,1) {$v_3$};
\node[circ] (D) at (1,1) {$v_4$};
\draw[black,thick] (A)--(C)--(B)--(D)--(A);
\end{tikzpicture}
&&
\begin{tikzpicture}[scale=0.8,baseline=-5mm,
circ/.style={circle,draw,inner sep=0.8pt, minimum width=12pt,font=\footnotesize}
]
\draw(1,2)node{$C$};
\node[circ] (A) at (1,0) {$v_1$};
\node[circ] (B) at (0,1) {$v_2$};
\node[circ] (C) at (1,1) {$v_3$};
\node[circ] (D) at (2,1) {$v_4$};
\draw[black,thick] (B)--(A)--(C)(A)--(D);
\end{tikzpicture}
&&
\begin{tikzpicture}[scale=0.8,baseline=-5mm,
circ/.style={circle,draw,inner sep=0.8pt, minimum width=12pt,font=\footnotesize}
]
\draw(-1,2)node{$D$};
\node[circ] (A) at (1,0) {$v_1$};
\node[circ] (B) at (1,1) {$v_2$};
\node[circ] (C) at (0,2) {$v_3$};
\node[circ] (D) at (2,2) {$v_4$};
\draw[black,thick] (A)--(B)--(C)(B)--(D);
\end{tikzpicture}
\end{align*}
For the above posets
\begin{align*}
2K_A(\xvec)
+3K_B(\xvec)
+2K_C(\xvec)
&=
7\schurS_{4}
+7\schurS_{31}
+\schurS_{22}
+2\schurS_{211}
\\&=
2h_4
+4h_{31}
-h_{22}
+2h_{1111}
\end{align*}
is Schur-positive but not $h$-positive, and
\begin{align*}
K_A(\xvec)
+3K_B(\xvec)
+K_C(\xvec)
+3K_D(\xvec)
&=
8\schurS_{4}
+5\schurS_{31}
-\schurS_{22}
+\schurS_{211}
\end{align*}
is not even Schur-positive.

\subsection{Murnaghan--Nakayama-type formula}

There are several quasisymmetric analogues of Schur functions. 
Here is an incomplete list of functions that can be viewed as such:

\begin{itemize}
 \item The \emph{fundamental quasisymmetric functions}. 
 
 \item The \emph{quasisymmetric Schur functions}, $\QS_\alpha(\xvec)$, introduced by J.~Haglund, K.~Luoto, S.~Mason and S.~van~ Willgienburg in \cite{HaglundEtAl2011}.

 \item The \emph{row-strict Schur functions}, $\RS_\alpha(\xvec)$ by S.~Mason and J.~Remmel which are 
 related to the quasisymmetric Schur functions as $\RS_\alpha(\xvec) = \omega(\QS_\alpha(\xvec))$.
 
 \item The \emph{Young quasisymmetric Schur functions}, which are also closely related to the two above quasisymmetric variants of Schur functions, see \cite{LuotoEtAl2013IntroQSymSchur}.
 
 \item The \emph{dual immaculate quasisymmetric functions}, which expands positively into the Young quasisymmetric Schur functions, see \cite{AllenHallamMason2018}.
 
 \item The \emph{extended Schur functions}, by S.~Assaf and D.~Searles, see \cite{AssafSearles2017}.  This family includes the usual Schur functions.
\end{itemize}

All of the above families expand positively into the fundamental quasisymmetric
functions, and are bases for the space of quasisymmetric functions.
Using \cref{thm:gesselInPSum} one can give analogues of Roichman's formula, see \cref{thm:schurRoichmanPexp}, for these bases.

The classical Murnaghan--Nakayama rule \cite{Murnaghan1937,Nakayama1940} states that
\[
\psumP_r(\xvec)\,\schurS_\lambda(\xvec)
=
\sum_\mu
(-1)^{\mathrm{ht}(\mu/\lambda)}\,\schurS_\mu(\xvec)
\,,
\]
where the sum is over all $\mu$ such that $\mu/\lambda$ is a ribbon of size $r$.
A natural future direction is then to seek quasisymmetric refinements or analogues of the
Murnaghan--Nakayama rule.
Let $\{X_\alpha\}$ be any of the above families of quaissymmetric functions indexed by partitions. 
The problem is then to find a rule that 
gives the coefficients $\chi_{\alpha\beta}^\gamma$ in the expansion
\[
\Psi_\alpha(\xvec)\,X_\beta(\xvec)
=
\sum_\gamma \chi_{\alpha\beta}^\gamma\,X_\gamma(\xvec)
\,.
\]
Possible research in this direction is also discussed in \cite[Sec.~7.1]{BallantineDaughertyHicksMason2017}.

\subsection{Poset invariants}

Whenever a class of combinatorial objects has a nontrivial isomorphism problem 
(such as posets, graphs or knots) it immediately becomes an interesting task to find 
invariants that might be used to distinguish such objects.

In~\cite{McNamaraWard2014} P.~McNamara and R.~Ward ask for a 
necessary and sufficient condition that two labeled posets $(P,w)$ and $(Q,w')$ have the 
same $P$-partition generating function, that is, $K_{P,w}=K_{Q,w'}$.

For naturally labeled posets \cref{thm:KP_opsurj} yields that $K_{P}=K_{Q}$ if and 
only if $\abs{\opsurj_{\alpha}^{\ast}(P)}=\abs{\opsurj_{\alpha}^{\ast}(Q)}$ for all compositions $\alpha$.
That is, the numbers of certain order-preserving surjections onto chains agree.
Note that this includes (in the case of naturally labeled posets) 
the observation~\cite[Prop.~3.2]{McNamaraWard2014} that $K_{P}=K_{Q}$ 
implies $\abs{\ILinExt(P,w)}=\abs{\ILinExt(Q,w')}$.
The judgment whether our answer is more useful than the trivial 
answer, ``$K_{P}=K_{Q}$ if and only if the multisets of descent 
sets $\{\DES(\sigma):\sigma\in\ILinExt(P,w)\}$ and $\{\DES(\sigma):\sigma\in\ILinExt(Q,w')\}$ agree'', is left to the reader.

P.~McNamara and R.~Ward
pose several other problems in this direction, many of which where recently solved for naturally labeled posets by R.~Liu and M.~Weselcouch~\cite{LiuWeselcouch2018}.
It could be worth investigating whether the $\Psi$-expansion of $K_P$ has applications in this regard.

\medskip
Another open question that was first raised by R.~Stanley in \cite[p. 170]{Stanley95Chromatic} is whether the chromatic symmetric function $\chrom_G(\xvec)$ defined in \cref{sec:chromatic} distinguish trees.
This was investigated, for instance, by J.~Martin, M.~Morin and J.~Wagner in ~\cite{MartinMorinWagner2008}.

Let $T$ be a rooted tree.
We interpret $T$ as a poset by declaring the edges to be cover relations, and the root to be the unique minimal element.
It was recently shown by T.~Hasebe and S.~Tsujie that the generating function of reverse $P$-partitions distinguishes rooted trees~\cite{HasebeTsujie2017}.
The following is a straightforward consequence of this result.

\begin{proposition}\label{prop:rooted trees}
The chromatic quasisymmetric function $\chrom_G(\xvec;q)$ distinguishes rooted trees with all edges directed away from the root.
\end{proposition}

\begin{proof}
Let $T$ be a rooted tree on $n$ vertices.
It follows from the proof of \cref{thm:chromPsiPos} that the coefficient of $q^{n-1}$ in $\chrom_T(\xvec;q)$ is just $K_{T,w}$, where $T$ is viewed as a poset as above and $w$ is an order-reversing labeling.
Thus the claim is a consequence of~\cite[Thm.~1.3]{HasebeTsujie2017}.
\end{proof}

It is an open problem whether $\chrom_G(\xvec;q)$ distinguishes all oriented trees.

\medskip
Let $P$ and $Q$ be (disjoint) posets.
Let $P+Q$ denote the \emph{direct sum} of $P$ and $Q$, that is, the partial order on $P\sqcup Q$ defined by $\poseta<\posetb$ if $\poseta<_P\posetb$ or $\poseta<_Q\posetb$.
Let $P\oplus Q$ denote the \emph{ordinal sum} of $P$ and $Q$, that is, the partial order on $P\sqcup Q$ defined by $\poseta<\posetb$ if (i) $\poseta<_P\posetb$, or (ii) $\poseta<_Q\posetb$, or (iii) $\poseta\in P$ and $\posetb\in Q$.
Clearly all rooted trees can be obtained by successively taking direct sums and adding a minimal element, that is, forming the poset $\{\hat{0}\}\oplus P$.
A poset is called \emph{series parallel} if it can be built from sigletons using only the two operations direct and ordinal sum.

A very natural question raised by T.~Hasebe and S.~Tsujie is whether $K_P$ distinguishes series parallel posets.
The ideas in this paper might offer a new angle to attack this problem, since it is not too difficult to compute the expansions of $K_{P+Q}$ respectively $K_{P\oplus Q}$ into quasisymmetric power sums recursively using \cref{thm:KP_opsurj}.

Another possible question is whether other constructions on posets 
have a simple interpretation in the basis of quasisymmetric power sums.

\subsection{Type \texorpdfstring{$B_n$}{B} analogues}

There is a type $B$ analogue of \cref{thm:schurRoichmanPexp} given in \cite{AdinEtAl2017}. 
This setting uses Schur functions, power sum symmetric functions and fundamental quasisymmetric functions in two sets of variables $\xvec$ and $\yvec$.
It is natural to ask if this result extends to an analogue of \cref{thm:gesselInPSum} that uses some kind of quasisymmetric power sums in two sets of variables.

\bibliographystyle{alphaurl}
\bibliography{bibliography}

\newcommand{\etalchar}[1]{$^{#1}$}
\begin{thebibliography}{HLMvW11}

\bibitem[AAER17]{AdinEtAl2017}
Ron~M. Adin, Christos~A. Athanasiadis, Sergi Elizalde, and Yuval Roichman.
\newblock Character formulas and descents for the hyperoctahedral group.
\newblock {\em Advances in Applied Mathematics}, 87:128--169, jun 2017.
\newblock \href {http://dx.doi.org/10.1016/j.aam.2017.01.004}
  {\path{doi:10.1016/j.aam.2017.01.004}}.

\bibitem[AB16]{AwanBernardi2016}
Jordan Awan and Olivier Bernardi.
\newblock Tutte polynomials for directed graphs.
\newblock {\em ArXiv e-prints}, 2016.
\newblock \href {http://arxiv.org/abs/1610.01839} {\path{arXiv:1610.01839}}.

\bibitem[AHM18]{AllenHallamMason2018}
Edward~E. Allen, Joshua Hallam, and Sarah~K. Mason.
\newblock Dual immaculate quasisymmetric functions expand positively into
  {Y}oung quasisymmetric {S}chur functions.
\newblock {\em Journal of Combinatorial Theory, Series A}, 157:70--108, jul
  2018.
\newblock \href {http://dx.doi.org/10.1016/j.jcta.2018.01.006}
  {\path{doi:10.1016/j.jcta.2018.01.006}}.

\bibitem[AP17]{AlexanderssonPanova2016}
Per Alexandersson and Greta Panova.
\newblock {LLT} polynomials, chromatic quasisymmetric functions and graphs with
  cycles.
\newblock {\em ArXiv e-prints}, 2017.
\newblock \href {http://arxiv.org/abs/1705.10353} {\path{arXiv:1705.10353}}.

\bibitem[AR15]{AdinRoichman2015}
Ron~M. Adin and Yuval Roichman.
\newblock Matrices, characters and descents.
\newblock {\em Linear Algebra and its Applications}, 469:381--418, mar 2015.
\newblock \href {http://dx.doi.org/10.1016/j.laa.2014.11.028}
  {\path{doi:10.1016/j.laa.2014.11.028}}.

\bibitem[AS17]{AssafSearles2017}
Sami Assaf and Dominic Searles.
\newblock Kohnert polynomials.
\newblock {\em ArXiv e-prints}, 2017.
\newblock \href {http://arxiv.org/abs/1711.09498} {\path{arXiv:1711.09498}}.

\bibitem[Ath15]{Athanasiadis15}
Christos~A. Athanasiadis.
\newblock Power sum expansion of chromatic quasisymmetric functions.
\newblock {\em Electronic Journal of Combinatorics}, 22(2):1--9, 2015.

\bibitem[BDH{\etalchar{+}}17]{BallantineDaughertyHicksMason2017}
Cristina Ballantine, Zajj Daugherty, Angela Hicks, Sarah Mason, and Elizabeth
  Niese.
\newblock Quasisymmetric {P}ower {S}ums.
\newblock {\em ArXiv e-prints}, 2017.
\newblock \href {http://arxiv.org/abs/1710.11613} {\path{arXiv:1710.11613}}.

\bibitem[BJR09]{BilleraJiaReiner2009}
Louis~J. Billera, Ning Jia, and Victor Reiner.
\newblock A quasisymmetric function for matroids.
\newblock {\em European Journal of Combinatorics}, 30(8):1727--1757, nov 2009.
\newblock \href {http://dx.doi.org/10.1016/j.ejc.2008.12.007}
  {\path{doi:10.1016/j.ejc.2008.12.007}}.

\bibitem[CM17]{CarlssonMellit2017}
Erik Carlsson and Anton Mellit.
\newblock A proof of the shuffle conjecture.
\newblock {\em Journal of the American Mathematical Society}, 31(3):661--697,
  nov 2017.
\newblock \href {http://dx.doi.org/10.1090/jams/893}
  {\path{doi:10.1090/jams/893}}.

\bibitem[Dé83]{Desarmenien1983}
Jacques Désarménien.
\newblock Fonctions symétriques associées à des suites classiques de
  nombres.
\newblock {\em Ann. Scient. Ecole Normale Supérieure}, 16:271--304, 1983.

\bibitem[Ell16]{Ellzey2016}
Brittney Ellzey.
\newblock Chromatic quasisymmetric functions of directed graphs.
\newblock {\em ArXiv e-prints}, 2016.
\newblock \href {http://arxiv.org/abs/1612.04786} {\path{arXiv:1612.04786}}.

\bibitem[ER91]{Egecioglu1991}
{\"{O}}mer E{\u{g}}ecio{\u{g}}lu and Jeffrey~B. Remmel.
\newblock Brick tabloids and the connection matrices between bases of symmetric
  functions.
\newblock {\em Discrete Applied Mathematics}, 34(1-3):107--120, nov 1991.
\newblock \href {http://dx.doi.org/10.1016/0166-218x(91)90081-7}
  {\path{doi:10.1016/0166-218x(91)90081-7}}.

\bibitem[ER13]{ElizaldeRoichman2013}
Sergi Elizalde and Yuval Roichman.
\newblock Arc permutations.
\newblock {\em Journal of Algebraic Combinatorics}, 39(2):301--334, may 2013.
\newblock \href {http://dx.doi.org/10.1007/s10801-013-0449-6}
  {\path{doi:10.1007/s10801-013-0449-6}}.

\bibitem[GKL{\etalchar{+}}95]{GelfandKrobLascouxLeclercRetakhThibon1995}
Israel~M. Gelfand, Daniel Krob, Alain Lascoux, Bernard Leclerc, Vladimir~S.
  Retakh, and Jean-Yves Thibon.
\newblock Noncommutative symmetrical functions.
\newblock {\em Advances in Mathematics}, 112(2):218--348, may 1995.
\newblock \href {http://dx.doi.org/10.1006/aima.1995.1032}
  {\path{doi:10.1006/aima.1995.1032}}.

\bibitem[HHL05]{Haglund2005Macdonald}
James Haglund, Mark Haiman, and Nicholas Loehr.
\newblock A combinatorial formula for {M}acdonald polynomials.
\newblock {\em J. Amer. Math. Soc.}, 18(03):735--762, jul 2005.
\newblock \href {http://dx.doi.org/10.1090/s0894-0347-05-00485-6}
  {\path{doi:10.1090/s0894-0347-05-00485-6}}.

\bibitem[HLMvW11]{HaglundEtAl2011}
James Haglund, Kurt~W. Luoto, Sarah~K. Mason, and Stephanie van Willigenburg.
\newblock Quasisymmetric {S}chur functions.
\newblock {\em Journal of Combinatorial Theory, Series A}, 118(2):463--490,
  2011.
\newblock \href {http://dx.doi.org/10.1016/j.jcta.2009.11.002}
  {\path{doi:10.1016/j.jcta.2009.11.002}}.

\bibitem[HT17]{HasebeTsujie2017}
Takahiro Hasebe and Shuhei Tsujie.
\newblock Order quasisymmetric functions distinguish rooted trees.
\newblock {\em Journal of Algebraic Combinatorics}, 46(3-4):499--515, may 2017.
\newblock \href {http://dx.doi.org/10.1007/s10801-017-0761-7}
  {\path{doi:10.1007/s10801-017-0761-7}}.

\bibitem[Hum11]{Humpert2011}
Brandon Humpert.
\newblock A quasisymmetric function generalization of the chromatic symmetric
  function.
\newblock {\em Electronic Journal of Combinatorics}, 18(1):1--13, 2011.

\bibitem[HW17]{HaglundWilson2017}
James Haglund and Andrew~T. Wilson.
\newblock Macdonald polynomials and chromatic quasisymmetric functions.
\newblock {\em ArXiv e-prints}, 2017.
\newblock \href {http://arxiv.org/abs/1701.05622} {\path{arXiv:1701.05622}}.

\bibitem[Knu98]{Knuth1998ArtOfProgramming}
Donald~E. Knuth.
\newblock {\em The Art of Computer Programming, Volume 3: (2nd Ed.) Sorting and
  Searching}.
\newblock Addison Wesley Longman Publishing Co., Inc., Redwood City, CA, USA,
  1998.

\bibitem[LLT97]{Lascoux97ribbontableaux}
Alain Lascoux, Bernard Leclerc, and Jean-Yves Thibon.
\newblock {R}ibbon {T}ableaux, {H}all-{L}ittlewood {F}unctions, {Q}uantum
  {A}ffine {A}lgebras {A}nd {U}nipotent {V}arieties.
\newblock {\em J. Math. Phys}, 38:1041--1068, 1997.

\bibitem[LMvW13]{LuotoEtAl2013IntroQSymSchur}
Kurt Luoto, Stefan Mykytiuk, and Stephanie van Willigenburg.
\newblock {\em An Introduction to Quasisymmetric {S}chur Functions: {H}opf
  Algebras, Quasisymmetric Functions, and {Y}oung Composition Tableaux
  (SpringerBriefs in Mathematics)}.
\newblock Springer, 2013.
\newblock URL:
  \url{https://www.math.ubc.ca/\%7esteph/papers/QuasiSchurBook.pdf}.

\bibitem[LR10]{LoehrRemmel2010}
Nicholas~A. Loehr and Jeffrey~B. Remmel.
\newblock A computational and combinatorial expos{\'{e}}
  of~plethystic~calculus.
\newblock {\em Journal of Algebraic Combinatorics}, 33(2):163--198, jun 2010.
\newblock \href {http://dx.doi.org/10.1007/s10801-010-0238-4}
  {\path{doi:10.1007/s10801-010-0238-4}}.

\bibitem[LW18]{LiuWeselcouch2018}
Ricky~I. Liu and Michael Weselcouch.
\newblock $p$-partition generating function equivalence of naturally labeled
  posets.
\newblock {\em ArXiv e-prints}, 2018.
\newblock \href {http://arxiv.org/abs/1807.02865} {\path{arXiv:1807.02865}}.

\bibitem[Mac79]{Macdonald79symmetric}
Ian~G. Macdonald.
\newblock {\em Symmetric {F}unctions and {H}all {P}olynomials}.
\newblock Oxford University Press, 1979.

\bibitem[MMW08]{MartinMorinWagner2008}
Jeremy~L. Martin, Matthew Morin, and Jennifer~D. Wagner.
\newblock On distinguishing trees by their chromatic symmetric functions.
\newblock {\em Journal of Combinatorial Theory, Series A}, 115(2):237--253, feb
  2008.
\newblock \href {http://dx.doi.org/10.1016/j.jcta.2007.05.008}
  {\path{doi:10.1016/j.jcta.2007.05.008}}.

\bibitem[MR83]{MetropolisRota1983}
Nicholas Metropolis and Gian-Carlo Rota.
\newblock {W}itt vectors and the algebra of necklaces.
\newblock {\em Advances in Mathematics}, 50(2):95--125, nov 1983.
\newblock \href {http://dx.doi.org/10.1016/0001-8708(83)90035-x}
  {\path{doi:10.1016/0001-8708(83)90035-x}}.

\bibitem[Mur37]{Murnaghan1937}
Francis~D. Murnaghan.
\newblock The characters of the symmetric group.
\newblock {\em Amer. J. Math.}, 59:739--753, 1937.

\bibitem[MW14]{McNamaraWard2014}
Peter R.~W. McNamara and Ryan~E. Ward.
\newblock Equality of {$P$ }-partition generating functions.
\newblock {\em Annals of Combinatorics}, 18(3):489--514, jul 2014.
\newblock \href {http://dx.doi.org/10.1007/s00026-014-0236-7}
  {\path{doi:10.1007/s00026-014-0236-7}}.

\bibitem[Nak40]{Nakayama1940}
Tadashi Nakayama.
\newblock On some modular properties of irreducible representations of a
  symmetric group. i and ii.
\newblock {\em Jap. J. Math.}, 17:165--184,411--423, 1940.

\bibitem[Ram98]{Ram1998}
Arun Ram.
\newblock An elementary proof of {R}oichman's rule for irreducible characters
  of {I}wahori--{H}ecke algebras of type {A}.
\newblock In {\em Mathematical Essays in honor of {G}ian-{C}arlo {R}ota},
  Progr.\ Math., chapter~17, pages 335--342. Birkh\"{a}user Boston, 1998.
\newblock \href {http://dx.doi.org/10.1007/978-1-4612-4108-9_17}
  {\path{doi:10.1007/978-1-4612-4108-9_17}}.

\bibitem[Roi97]{Roichman1997}
Yuval Roichman.
\newblock A recursive rule for {K}azhdan--{L}usztig characters.
\newblock {\em Advances in Mathematics}, 129(1):25--45, jul 1997.
\newblock \href {http://dx.doi.org/10.1006/aima.1996.1629}
  {\path{doi:10.1006/aima.1996.1629}}.

\bibitem[Slo16]{OEIS}
Neil J.~A. Sloane.
\newblock The on-line encyclopedia of integer sequences, 2016.
\newblock URL: \url{https://oeis.org}.

\bibitem[SS93]{StanleyStembridge1993}
Richard~P. Stanley and John~R. Stembridge.
\newblock On immanants of {J}acobi--{T}rudi matrices and permutations with
  restricted position.
\newblock {\em Journal of Combinatorial Theory, Series A}, 62(2):261--279, mar
  1993.
\newblock \href {http://dx.doi.org/10.1016/0097-3165(93)90048-d}
  {\path{doi:10.1016/0097-3165(93)90048-d}}.

\bibitem[SSW11]{SaganShareshianWachs2011}
Bruce Sagan, John Shareshian, and Michelle~L. Wachs.
\newblock Eulerian quasisymmetric functions and cyclic sieving.
\newblock {\em Advances in Applied Mathematics}, 46(1):536--562, jan 2011.
\newblock \href {http://dx.doi.org/10.1016/j.aam.2010.01.013}
  {\path{doi:10.1016/j.aam.2010.01.013}}.

\bibitem[Sta72]{Stanley1972}
Richard~P. Stanley.
\newblock {\em Ordered structures and partitions}, volume 119.
\newblock American Mathematical Society, Providence, Rhode Island, 1972.

\bibitem[Sta95]{Stanley95Chromatic}
Richard~P. Stanley.
\newblock A symmetric function generalization of the chromatic polynomial of a
  graph.
\newblock {\em Advances in Mathematics}, 111(1):166--194, 1995.
\newblock \href {http://dx.doi.org/10.1006/aima.1995.1020}
  {\path{doi:10.1006/aima.1995.1020}}.

\bibitem[Sta98]{Stanley98Chromatic}
Richard~P. Stanley.
\newblock Graph colorings and related symmetric functions: ideas and
  applications.
\newblock {\em Discrete Mathematics}, 193(1):267--286, 1998.
\newblock \href {http://dx.doi.org/10.1016/S0012-365X(98)00146-0}
  {\path{doi:10.1016/S0012-365X(98)00146-0}}.

\bibitem[Sta01]{StanleyEC2}
Richard~P. Stanley.
\newblock {\em Enumerative {C}ombinatorics: {V}olume 2}.
\newblock Cambridge University Press, 1st edition, 2001.
\newblock URL: \url{http://www.worldcat.org/isbn/0521789877}.

\bibitem[Sta05]{Stanley2005}
Richard~P. Stanley.
\newblock Some remarks on sign-balanced and maj-balanced posets.
\newblock {\em Advances in Applied Mathematics}, 34(4):880--902, may 2005.
\newblock \href {http://dx.doi.org/10.1016/j.aam.2003.12.002}
  {\path{doi:10.1016/j.aam.2003.12.002}}.

\bibitem[Sta11]{StanleyEC1}
Richard~P. Stanley.
\newblock {\em Enumerative {C}ombinatorics: {V}olume 1}.
\newblock Cambridge University Press, 2nd edition, 2011.

\bibitem[SW10]{ShareshianWachs2010}
John Shareshian and Michelle~L. Wachs.
\newblock Eulerian quasisymmetric functions.
\newblock {\em Advances in Mathematics}, 225(6):2921--2966, 2010.
\newblock \href {http://dx.doi.org/10.1016/j.aim.2010.05.009}
  {\path{doi:10.1016/j.aim.2010.05.009}}.

\bibitem[SW16]{ShareshianWachs2016}
John Shareshian and Michelle~L. Wachs.
\newblock Chromatic quasisymmetric functions.
\newblock {\em Advances in Mathematics}, 295(4):497--551, jun 2016.
\newblock \href {http://dx.doi.org/10.1016/j.aim.2015.12.018}
  {\path{doi:10.1016/j.aim.2015.12.018}}.

\bibitem[Thi01]{Thibon2001}
Jean-Yves Thibon.
\newblock The cycle enumerator of unimodal permutations.
\newblock {\em Annals of Combinatorics}, 5(3-4):493--500, dec 2001.
\newblock \href {http://dx.doi.org/10.1007/s00026-001-8024-6}
  {\path{doi:10.1007/s00026-001-8024-6}}.

\bibitem[vL00]{Leeuwen2000}
Marc A.~A. van Leeuwen.
\newblock Some bijective correspondences involving domino tableaux.
\newblock {\em Electronic Journal of Combinatorics}, 7:1--25, 2000.

\end{thebibliography}

\end{document}